\documentclass[amsa]{ipart}

\newif\ifamsa 
\amsatrue     

\RequirePackage[OT1]{fontenc}
\RequirePackage{amsthm,amsmath,amssymb}
\RequirePackage[numbers,square]{natbib} %
\RequirePackage[colorlinks]{hyperref}

\newtheorem{thm}{Theorem}

\pubyear{0000}
\volume{0}
\issue{0}
\firstpage{1}
\lastpage{54}

\received{\smonth{1} \sday{1}, \syear{2016}} 

\usepackage{mathtools}
\usepackage[normalem]{ulem} 
\usepackage{subcaption}
\usepackage{enumerate}
\usepackage[ruled,vlined]{algorithm2e}
\usepackage{amsfonts}
\usepackage{amssymb}
\usepackage{booktabs}
\usepackage{multirow}

\newcommand{\rev}[1]{{\color{blue}#1}}
\newcommand{\commwy}[1]{{\color{red}(#1 -- Wotao)}} 
\newcommand{\commzp}[1]{{\color{red}(#1 -- Zhimin)}} 
 
\newcommand{\remove}[1]{{}}
\newcommand{\cut}[1]{}

\newcommand{\vx}{{\mathbf{x}}}

\newcommand{\vz}{{\mathbf{z}}}

\newcommand{\vZ}{{\mathbf{Z}}}

\newcommand{\cA}{{\mathcal{A}}}
\newcommand{\cB}{{\mathcal{B}}}
\newcommand{\cC}{{\mathcal{C}}}

\newcommand{\cE}{{\mathcal{E}}}
\newcommand{\cF}{{\mathcal{F}}}

\newcommand{\cI}{{\mathcal{I}}}
\newcommand{\cJ}{{\mathcal{J}}}

\newcommand{\cM}{{\mathcal{M}}}
\newcommand{\cN}{{\mathcal{N}}}

\newcommand{\cP}{{\mathcal{P}}}

\newcommand{\cR}{{\mathcal{R}}}
\newcommand{\cS}{{\mathcal{S}}}
\newcommand{\cT}{{\mathcal{T}}}

\newcommand{\cV}{{\mathcal{V}}}

\newcommand{\cZ}{{\mathcal{Z}}}

\newcommand{\FF}{\mathbb{F}}
\newcommand{\RR}{\mathbb{R}}

\newcommand{\HH}{\mathbb{H}}
\newcommand{\II}{\mathbb{I}}
\newcommand{\JJ}{\mathbb{J}}

\newcommand{\GG}{\mathbb{G}}

\newcommand{\sign}{\mathrm{sign}}

\newcommand{\St}{{\mathrm{subject~to}}} 
\newcommand{\Diag}{{\mathrm{Diag}}} 
\newcommand{\Proj}{{\mathrm{Proj}}}
\newcommand{\prj}{{\mathbf{proj}}}
\newcommand{\prox}{\mathbf{prox}}

\newcommand{\TS}{{\cT_{\mathrm{3S}}}}
\newcommand{\TFBS}{{\cT_{\mathrm{FBS}}}}
\newcommand{\TBFS}{{\cT_{\mathrm{BFS}}}}
\newcommand{\TDRS}{{\cT_{\mathrm{DRS}}}}

\newcommand{\TFBFS}{{\cT_{\mathrm{FBFS}}}}
\newcommand{\TFDRS}{{\cT_{\mathrm{FDRS}}}}
\newcommand{\TVC}{{\cT_{\textnormal{CV}}}}

\newcommand{\Grph}{\mathrm{Grph}}
\DeclareMathOperator*{\argmin}{arg\,min}
\DeclareMathOperator*{\argmax}{arg\,max}
\DeclareMathOperator*{\Min}{minimize}

\newcommand{\nops}[2]{\ensuremath{\mathfrak{M}\left[{#1}\mapsto {#2}\right]}}

\newcommand{\bc}{\begin{center}}
\newcommand{\ec}{\end{center}}

\newcommand{\bdm}{\begin{displaymath}}
\newcommand{\edm}{\end{displaymath}}

\newcommand{\beq}{\begin{equation}}
\newcommand{\eeq}{\end{equation}}

\newcommand{\bfl}{\begin{flushleft}}
\newcommand{\efl}{\end{flushleft}}

\newcommand{\bt}{\begin{tabbing}}
\newcommand{\et}{\end{tabbing}}

\newcommand{\beqn}{\begin{align}}
\newcommand{\eeqn}{\end{align}}

\newcommand{\beqs}{\begin{align*}} 
\newcommand{\eeqs}{\end{align*}}  

\newtheorem{assumption}{Assumption}

\ifamsa

    \newtheorem{definition}{Definition}
    
    \newtheorem{remark}{Remark}
    \newtheorem{lemma}{Lemma}
    \newtheorem{proposition}{Proposition}
    \newtheorem{example}{Example}
\fi

\usepackage{amsmath}
\RequirePackage[normalem]{ulem} 
\RequirePackage{color}\definecolor{RED}{rgb}{1,0,0}\definecolor{BLUE}{rgb}{0,0,1} 
\providecommand{\DIFaddtex}[1]{{\protect\color{blue}\uwave{#1}}} 
\providecommand{\DIFaddbegin}{} 
\providecommand{\DIFaddend}{} 
\providecommand{\DIFdelbegin}{} 
\providecommand{\DIFdelend}{} 
\providecommand{\DIFadd}[1]{\texorpdfstring{\DIFaddtex{#1}}{#1}} 

\begin{document}

\begin{frontmatter}
\title{Coordinate friendly structures, algorithms and applications\thanksref{T1}}
\thankstext{T1}{This work is supported by NSF Grants DMS-1317602 and ECCS-1462398.}
\runtitle{Coordinate friendly structures, algorithms, and applications}

\begin{aug}
\author{\fnms{Zhimin} \snm{Peng}\ead[label=e1]{zhimin.peng@math.ucla.edu}},
\address{PO Box 951555\\
UCLA Math Department\\
Los Angeles, CA 90095 \\
\printead{e1}}
\author{\fnms{Tianyu} \snm{Wu}\ead[label=e2]{wuty11@math.ucla.edu}},
\address{PO Box 951555\\
UCLA Math Department\\
Los Angeles, CA 90095 \\
\printead{e2}}
\author{\fnms{Yangyang} \snm{Xu}\ead[label=e3]{yangyang@ima.umn.edu}},
\address{207 Church St SE \\
University of Minnesota, Twin Cities \\
Minneapolis, MN 55455 \\
\printead{e3}}
\author{\fnms{Ming} \snm{Yan}\ead[label=e4]{yanm@math.msu.edu}},
\address{Department of Computational Mathematics, Science and Engineering\\
Department of Mathematics\\
Michigan State University \\
East Lansing, MI 48824 \\
\printead{e4}}
\and
\author{\\\fnms{Wotao} \snm{Yin}
\ead[label=e5]{wotaoyin@math.ucla.edu}}
\address{PO Box 951555\\
UCLA Math Department\\
Los Angeles, CA 90095 \\
\printead{e5}}


\affiliation{Some University and Another University}

\end{aug}

\begin{abstract}
This paper focuses on  \emph{coordinate update methods}, which are useful for solving problems involving large or high-dimensional datasets. They decompose a problem into simple subproblems, where each  updates one, or a small block of, variables while fixing others. These methods can deal with linear and nonlinear mappings,  smooth and nonsmooth functions, as well as convex and nonconvex problems. In addition, they are easy to parallelize.  

The great performance of coordinate update methods depends on solving simple subproblems. To derive simple subproblems for several new classes of applications, this paper systematically studies  \emph{coordinate friendly} operators that perform low-cost coordinate updates. 


Based on the discovered coordinate friendly operators, as well as operator splitting techniques, we obtain new coordinate update algorithms for a variety of problems in machine learning, image processing, as well as sub-areas of optimization. Several problems are treated with coordinate update for the first time in history. The obtained algorithms are scalable to  large instances through parallel and even asynchronous computing. We present numerical examples to illustrate how effective these algorithms are.

\end{abstract}


\end{frontmatter}

\newpage


\section{Introduction}
This paper studies  \emph{coordinate update methods}, which reduce a large problem to smaller subproblems and are useful for solving large-sized problems. These methods handle both  linear and nonlinear maps, smooth and nonsmooth functions, and convex and nonconvex problems. 
The common special examples of these methods are the Jacobian and Gauss-Seidel algorithms for solving a linear system of equations, and they are also commonly used for solving differential equations (e.g., \emph{domain decomposition}) and optimization problems (e.g., \emph{coordinate descent}).

After coordinate update methods were initially introduced in each topic area, their evolution  had been slow until recently, when data-driven applications (e.g., in signal processing,  image processing, and statistical and machine learning) impose strong demand for scalable numerical solutions; consequently,  numerical methods of \emph{small footprints}, including  coordinate update methods, become increasingly popular. These methods are generally applicable to many problems involving large
or high-dimensional datasets.

Coordinate update methods generate simple subproblems that update one variable, or a small block of variables, while fixing others. The variables can be updated in  the \textit{cyclic}, \textit{random}, or \textit{greedy} orders, which can be selected to adapt to the problem. The subproblems that perform coordinate updates also have different forms. 
Coordinate updates can be  applied either sequentially on a single thread or concurrently on multiple threads, or even in an asynchronous parallel fashion. They have been demonstrated to give rise to very powerful and scalable algorithms.

\cut{\begin{table}
\begin{center}
\begin{tabular}{c|c|c}
\hline
Benefits & Coordinate Update & Full Update\\\hline\hline
memory footprint & small & big \\\hline
per iteration complexity & $O(n)$ & $O(n^2)$\\\hline
epoch & less (depend on updating order and stepsize) & more \\\hline
scalability & scalable & not scalable\\\hline
\end{tabular}
\end{center}
\caption{Benefits of coordinate update}\label{table:benefits}
\end{table}}

Clearly, the strong performance of coordinate
update methods relies on solving \emph{simple} subproblems. The cost of each subproblem must be proportional to how many coordinates it updates. 
When there are totally $m$ coordinates, the cost of updating one coordinate should not exceed the average per-coordinate cost of the full update (made to all the coordinates at once). Otherwise, coordinate update is not \emph{computationally worthy}. For example, let  $f:\RR^m\to \RR$ be  a $C^2$ function, and consider the Newton update  $x^{k+1} \gets x^k - \big(\nabla^2 f(x^k)\big)^{-1}\nabla f(x^k)$. Since updating each $x_i$ (keeping others  fixed) still requires forming the Hessian matrix $\nabla^2 f(x)$ (at least $O(m^2)$ operations) and factorizing it ($O(m^3)$ operations), there is little to save in computation compared to updating all the components of $x$ at once; hence, the Netwon's method is generally not amenable to coordinate update.

The recent coordinate-update literature has introduced new algorithms. However, they are primarily applied to a few, albeit important, classes of problems  that arise in machine learning. For many complicated problems,  it remains open whether simple subproblems can be obtained. We provide positive answers to several new classes of applications and introduce their coordinate update algorithms. 
Therefore, the focus of this paper is to build a set of tools for deriving simple subproblems and extending  coordinate updates to new territories of applications.

We will frame each application into an equivalent fixed-point problem
\beq\label{fpprob}
x = \cT x
\eeq
by specifying the operator  $\cT:\HH \to \HH$, where $x=(x_1,\ldots,x_m)\in \HH$, and $\HH=\HH_1\times \cdots \times \HH_m$ is a Hilbert space.
In many cases, the operator $\cT$ itself represents an iteration:
\beq\label{fpi}
x^{k+1} = \cT x^k
\eeq
such that the limit of the sequence $\{x^k\}$  exists and is a fixed point of $\cT$, which is also a solution to the application or from which a solution to the application can be obtained. We call the scheme~\eqref{fpi} a \emph{full update}, as opposed to updating one $x_i$ at a time. The  scheme~\eqref{fpi} has a number of interesting special cases including  methods of gradient descent, gradient projection, proximal gradient,  operator splitting, and many others.

We study the structures of $\cT$ that make the  following coordinate update algorithm \emph{computationally worthy}
\beq\label{cuitr}
x^{k+1}_i = x_i^k - \eta_k (x^k-\cT x^k)_i,
\eeq
where  $\eta_k$ is a  step size and $i\in [m] := \{1,\ldots,m\}$ is arbitrary. Specifically, the cost of performing ~\eqref{cuitr}  is roughly $\frac{1}{m}$, or lower, of that of performing~\eqref{fpi}. We call such $\cT$ a \textbf{Coordinate Friendly (CF)} operator, which we will formally define.

This paper will explore a variety of CF operators. Single CF operators include linear maps, projections to certain simple sets, proximal maps and gradients of (nearly) separable functions, as well as gradients of sparsely supported functions. There are many more composite CF operators, which are built from single CF and non-CF operators under a set of rules.  The fact that some of these operators are CF is not obvious.

These CF operators let us derive powerful coordinate update algorithms for a variety of applications including, but not limited to, linear and second-order cone programming, variational image processing, support vector machine, empirical risk minimization, portfolio optimization, distributed computing, and nonnegative matrix factorization. For each application, we present an algorithm in the form of~\eqref{fpi} so that its coordinate update~\eqref{cuitr} is efficient. 
{In this way we obtain new coordinate update algorithms for these applications, some of which  are treated with coordinate update for the first time. 
}
\cut{\rev{We would like to point out our coordinate update scheme is different from the domain decomposition approach. Domain decomposition splits the variable domain into small subdomains and solves the original problem in each subdomains, with variables coherent on the intersections or the boundary. Our coordinate update scheme, on the other hand, is based on operator splitting. Variables are divided into coordinates corresponding to the algorithm instead of any physical grid domains and updating each coordinate involves solving a subproblem different from the original one.
}
}


The developed coordinate update algorithms are easy to parallelize. In addition,  the work in this paper  gives rise to parallel and asynchronous extensions to  existing algorithms including the Alternating Direction Method of Multipliers (ADMM), primal-dual splitting algorithms, and others.

The paper is organized as follows. \S\ref{sec:literature} reviews the existing frameworks of coordinate update algorithms. \S\ref{sec:cuf} defines the CF operator and discusses different classes of CF operators. \S\ref{sec:comp-cuf}  introduces a set of rules to obtain composite CF operators and applies the results to operator splitting methods. \S\ref{sec:p-d} is dedicated to  primal-dual splitting methods with CF operators, where existing ones are reviewed and a new one is introduced. Applying the results of previous sections, \S\ref{sec:applications} obtains novel coordinate update algorithms for a variety of applications, some of which have been tested with their numerical results presented in \S\ref{sec:numerical}.

Throughout this paper, all functions $f,g,h$ are proper closed convex and can take the extended value $\infty$, and all sets $X,Y,Z$ are nonempty closed convex. \cut{, and  an operator $\cT:\HH\to\HH$ is single-valued unless otherwise stated.} The indicator function $\iota_X(x)$ returns $0$ if $x\in X$, and $\infty$ elsewhere. For a positive integer $m$, we let $[m]:=\{1,\ldots,m\}$.

\cut{

decompose  the variables in a large problem into a number of small blocks, giving rise to simple subproblems that have low complexity, small memory footprints, and can be solved either sequentially or in parallel.

Coordinate methods perform coordinate updates by keeping the other coordinates fixed. This often reduces to a lower dimensional subproblem and has lower per iteration computation complexity and space complexity compared to the fully update. This type of methods are often easy to implement.

For huge scale problems, there is a strong demand to solve the problem in a parallel, distributed and decentralized fashion. This is also in conformity with the ever increasing power of high performance computing systems.  However, due to the sequential (Gauss-Seidel approach) nature, it is often not straightforward to parallelize BC methods.]

We study the

The goal of this paper is to identify CF maps where parallelization can be applied without incurring high overhead.}

\subsection{Coordinate Update Algorithmic Frameworks}\label{sec:literature}
This subsection reviews the  \emph{sequential} and \emph{parallel} algorithmic frameworks for coordinate updates, as well as the relevant literature. 

The general framework of coordinate update is
\begin{enumerate}
\item set $k\gets 0$ and initialize $x^0\in\HH=\HH_1\times \cdots \times \HH_m$
\item while \emph{not converged} do
\item \quad select an index $i_k\in [m]$;
\item \quad update $x^{k+1}_{i}$ for $i={i_k}$ while keeping $x_i^{k+1}=x_i^k$, $\forall\,i\not ={i_k}$;
\item \quad $k\gets k+1$;
\end{enumerate}
Next we review the index rules and the methods to update $x_i$.

\subsubsection{Sequential Update} In this framework, there is a sequence of coordinate indices $i_1,i_2,\ldots$ chosen according to one of the following rules: cyclic, cyclic permutation, random, and greedy rules. At  iteration $k$, only the $i_k$th coordinate is updated:
$$ \begin{cases}
x^{k+1}_{i} = x_{i}^k - \eta_k(x^k-\cT x^k)_{i},& i=i_k,\\
x^{k+1}_{i} = x_i^k,&\text{for all } i\not= i_k.
\end{cases}
$$
Sequential updates have been applied to many problems such as the Gauss-Seidel iteration for solving a linear system of equations, alternating projection \cite{von1949rings,bauschke1993convergence} for finding a point in the intersection of two sets,  ADMM~\cite{glowinski1975ADMM, gabay1976ADMM} for solving monotropic programs, and Douglas-Rachford Splitting (DRS)~\cite{douglas1956DRS}  for finding a zero to the sum of two operators. 

In optimization,  \emph{coordinate descent} algorithms, at each iteration, minimize the function $f(x_1,\ldots,x_m)$ by fixing all but one variable $x_i$. Let 
$$x_{i-}:=(x_1,\ldots,x_{i-1}),\quad x_{i+}=(x_{i+1},\ldots,x_m)$$
collect all but the $i$th coordinate of $x$. Coordinate descent solves one of the following subproblems:
\begin{subequations}\label{coordes}
\begin{align}
(\cT x^k)_i & = \argmin_{x_i}f(x_{i-}^k,x_i,x_{i+}^k),\\
(\cT x^k)_i & = \argmin_{x_i}f(x_{i-}^k,x_i,x_{i+}^k)+\frac{1}{2\eta_k}\|x_i-x_i^k\|^2,\\
(\cT x^k)_i & = \argmin_{x_i}\,\langle \nabla_i f(x^k),x_i \rangle+\frac{1}{2\eta_k}\|x_i-x_i^k\|^2,\\
(\cT x^k)_i & = \argmin_{x_i}\,\langle \nabla_i f^{\mathrm{diff}}(x^k),x_i \rangle+f_i^{\mathrm{prox}}(x_i)+\frac{1}{2\eta_k}\|x_i-x_i^k\|^2,
\end{align}
\end{subequations}
which are called \emph{direct} update, \emph{proximal} update,  \emph{gradient} update, and \emph{prox-gradient} update, respectively. The last update applies to the function $$f(x) = f^{\mathrm{diff}}(x)+\sum_{i=1}^mf^{\mathrm{prox}}_i(x_i),$$ where $f^{\mathrm{diff}}$ is differentiable and each $f^{\mathrm{prox}}_i$ is proximable (its proximal map takes $O\big(\dim(x_i)\,\mathrm{polylog}(\dim(x_i))\big)$ operations to compute).


\textbf{Sequential-update literature.} Coordinate descent algorithms date back to the 1950s~\cite{hildreth1957quadprog}, when the \emph{cyclic} index rule was used. Its convergence has been established under a variety of cases, for both convex and nonconvex objective functions; see~\cite{Warga-63,zadeh1970note, Sargent-Sebastian-73,Han-88,luo1992convergence, Tseng-93, Grippo-Sciandrone-00, Tseng-01, razaviyayn2013unified, beck2013convergence, hong2015iteration, wright2015coordinate}. Proximal updates are studied in~\cite{Grippo-Sciandrone-00, attouch2010proximal} and developed into prox-gradient updates in~\cite{tseng2009_CGD, tseng2009block-linear, bolte2014proximal} and mixed updates in~\cite{XY_2013_multiblock}.

The \emph{random} index rule first appeared in~\cite{nesterov2012cd} and then~\cite{richtarik2014iteration, Lu_Xiao_rbcd_2015}. Recently,~\cite{XY_2014_ecd,Xu2015_APG_NTD} compared the convergence speeds of cyclic and stochastic update-orders. The gradient update has been relaxed to  stochastic gradient update for large-scale problems in~\cite{DangLan-SBMD, XY_2015_bsg}.

The \emph{greedy} index rule leads to fewer iterations but is often impractical since it requires a lot of effort to calculate  scores for all the coordinates. However, there are cases where calculating the scores is inexpensive~\cite{bertsekas1999nonlinear, li2009gcoord, wu2008coordinate} and the save in the total number of iterations significantly outweighs the extra calculation \cite{tseng2009_CGD, dhillon2011nearest, PYY_2013_GRock, schmidt2014coordinate}.

\textbf{A simple example.} We present the coordinate update algorithms under different index rules for solving a simple least squares problem:  
$$\Min_{x}\, f(x):= \frac{1}{2} \|A x - b\|^2,$$
where $A \in \RR^{p \times m}$ and $b \in \RR^p$ are Gaussian random. Our goal is to numerically demonstrate the advantages of coordinate updates over the  full update of  gradient descent:
$$x^{k+1} = x^k - \eta_k A^{\top}(A x^k - b).$$
The four tested index rules are: cyclic, cyclic permutation, random, and greedy under the Gauss-Southwell\footnote{it selects $i_k=\argmax_i\|\nabla_if(x^k)\|$.} rule.
Note that because this example is very special, the comparisons of different index rules are far from conclusive.

In the full update, the step size $\eta_k$ is set to the theoretical upper bound $\frac{2}{\|A\|_2^2}$, where  $\|A\|_2$ denotes the matrix operator norm and equals the largest singular value of $A$. For each coordinate update to $x_i$, the step size $\eta_k$ is set to $\frac{1}{(A^{\top}A)_{ii}}$. All of the full and coordinate updates have the same \emph{per-epoch} complexity, so we plot the objective errors in Figure \ref{fig:ls_full_vs_coord}.
\begin{figure}[!htbp] \centering
\includegraphics[width=50mm]{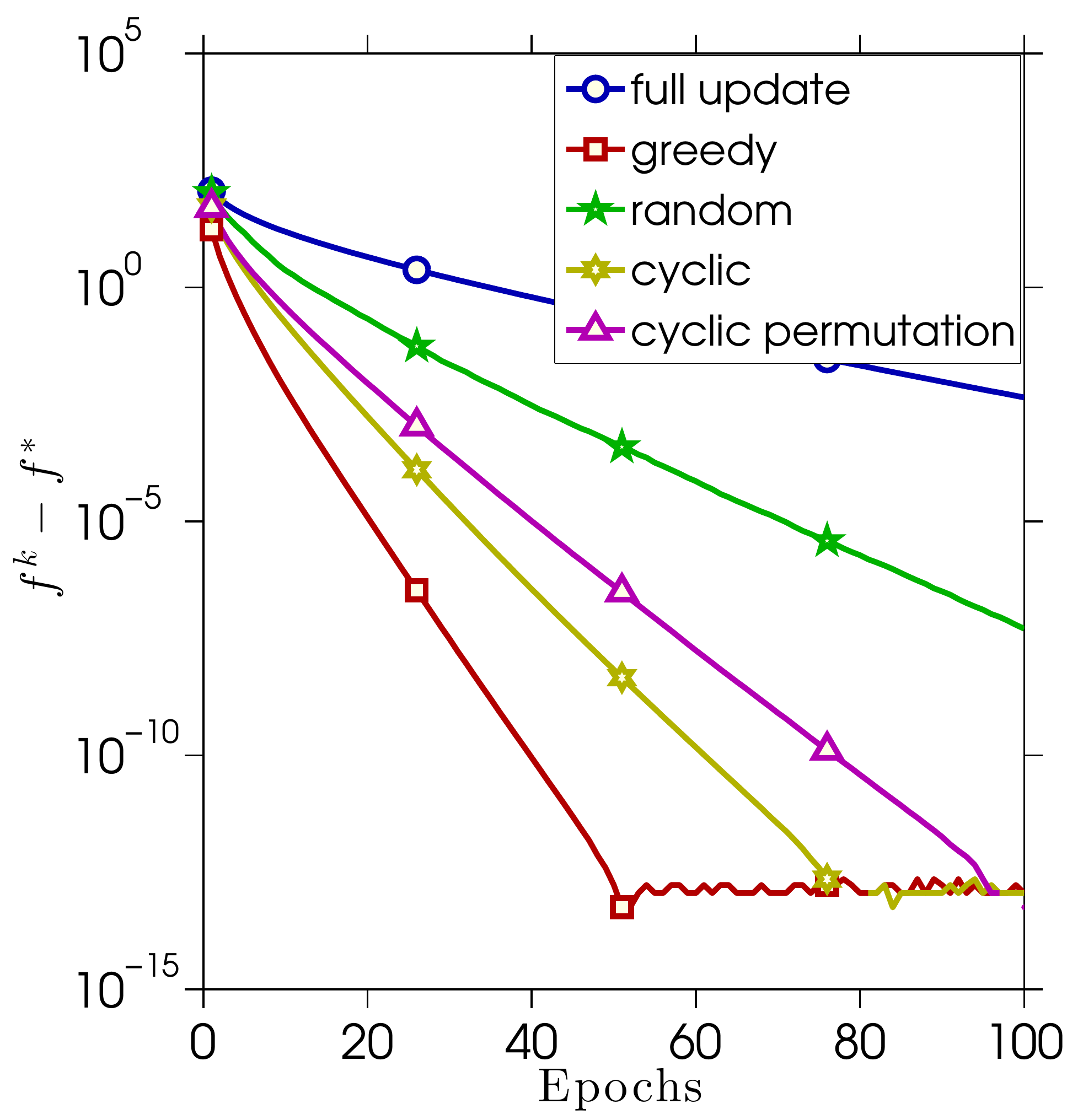}

\caption{Gradient descent: the coordinate updates are faster than the full update since the former can take larger steps at each step.}
\label{fig:ls_full_vs_coord}
\end{figure}

\subsubsection{Parallel Update} As one of their main advantages, coordinate update algorithms are easy to parallelize. In this subsection, we discuss  both synchronous (sync) and asynchronous (async) parallel updates.

\textbf{Sync-parallel (Jacobi) update} specifies a sequence of index subsets $\II_1,\II_2,\ldots \subseteq [m]$, and at each iteration $k$,  the coordinates in $\II_k$ are updated in parallel by multiple agents:
$$ \begin{cases}
x^{k+1}_{i} = x_{i}^k - \eta_k(x^k-\cT x^k)_{i},& i\in \II_k,\\
x^{k+1}_{i} = x_i^k,&i\not\in \II_k.
\end{cases}
$$
Synchronization across all agents ensures that  all $x_i$ in $\II_k$ are updated and also written to the memory before the next iteration starts. Note that, if $\II_k= [m]$ for all $k$, then all the coordinates are updated and, thus, each iteration reduces to the full update: $x^{k+1} =  x^k - \eta_k(x^k-\cT x^k).$ \cut{\color{red}This, of course, does not mean solving $x=\cT x$ in one step; therefore, do not confuse performing an  update in~\eqref{coordes} for all the coordinates in parallel from the joint minimization over all the variables.}

\textbf{Async-parallel update.} In this setting, a set of agents  still perform parallel updates, but synchronization is eliminated or weakened. Hence, each agent continuously applies~\eqref{fm:async}, which reads  $x$  from and writes $x_i$ back to the shared memory (or through communicating with other agents without shared memory):
\beq\label{fm:async} \begin{cases}
x^{k+1}_{i} = x_{i}^k - \eta_k \left((\cI-\cT) x^{k-d_k}\right)_{i},& i=i_k,\\
x^{k+1}_{i} = x_i^k,& \text{for all }i\not= i_k.
\end{cases}
\eeq
Unlike before, $k$ increases  whenever any agent completes an update.

The lack of synchronization often results in computation with out-of-date information. During the computation of the $k$th update, other agents make $d_k$ updates to $x$ in the shared memory; when the $k$th update is written, its input is already $d_k$ iterations out of date. This number is  referred to as the asynchronous delay. In~\eqref{fm:async}, the agent reads $x^{k-d_k}$ and commits the update to $x_{i_k}^k$. Here we have assumed  \emph{consistent} reading, i.e., $x^{k-d_k}$ lying in the set $\{x^j\}_{j=1}^k$. This requires implementing a memory lock. Removing the lock can lead to  \emph{inconsistent} reading, which still has convergence guarantees; see~\cite[Section 1.2]{Peng_2015_AROCK} for more details.

\begin{figure} \centering
    \begin{subfigure}[b]{0.45\linewidth}
        \includegraphics[width=60mm]{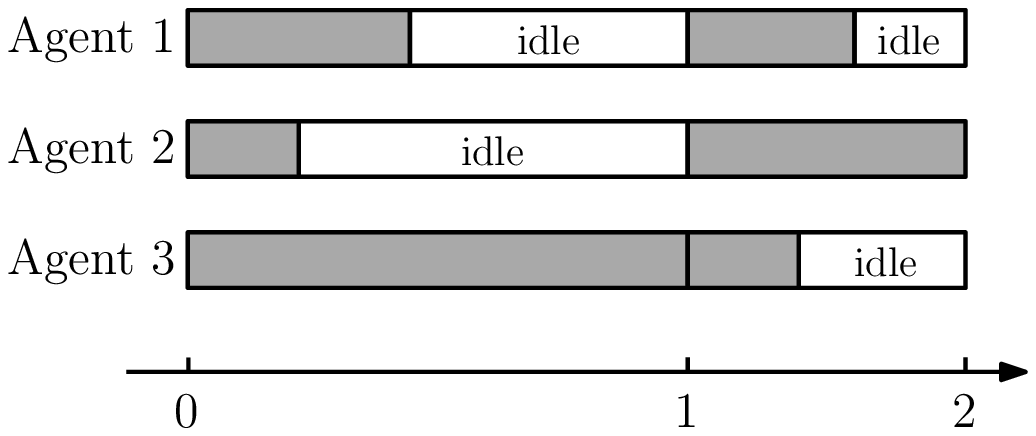}
        \caption{sync-parallel computing}
        \label{fig:parallel_a}
    \end{subfigure} %
    \quad
    \begin{subfigure}[b]{0.45\linewidth}
        \includegraphics[width=60mm]{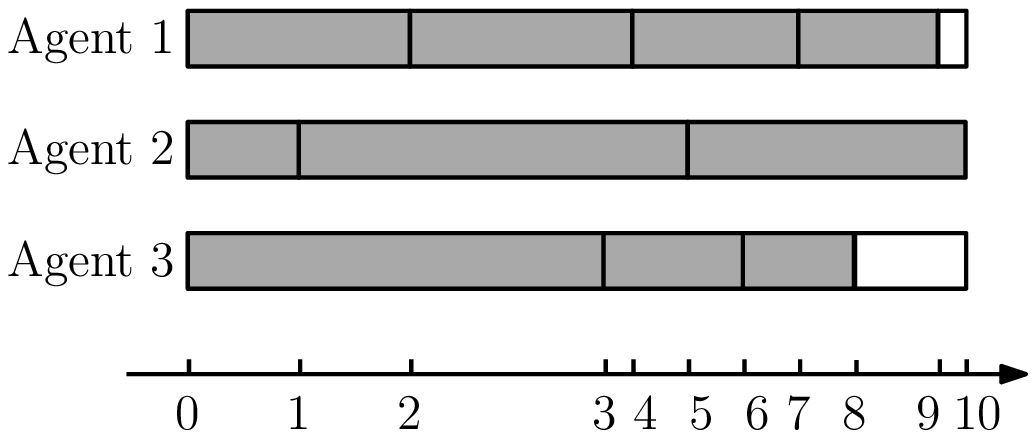}
        \caption{async-parallel computing}
        \label{fig:parallel_b}
    \end{subfigure} %
    \caption{ Sync-parallel computing (left) versus async-parallel computing (right). On the left, all the agents must wait at idle (white boxes) until the slowest agent has finished.}
    \label{fig:comp_sync_async}
\end{figure}


Synchronization across all agents means that all agents will wait for  the last (slowest) agent to complete. 
Async-parallel updates eliminate such idle time, spread out memory access and  communication, and thus often run much faster.  However, async-parallel is more difficult to analyze because of the asynchronous delay.


\remove{The sync-parallel update is a special case of the async-parallel update where the number of agents equals $|\II_k|$ and  the asynchronous delay is uniformly zero.}

\textbf{Parallel-update literature.}
Async-parallel methods can be traced back to~\cite{chazan1969chaotic} for systems of linear equations. For function minimization,~\cite{bertsekas1989parallel} introduced an async-parallel gradient projection method. Convergence rates are obtained in~\cite{tseng1991rate-asyn}.
Recently, \cite{bradley2011parallel,richtarik2012parallel} developed parallel randomized methods.

For fixed-point problems, async-parallel methods date back to~\cite{Baudet_1978_asynchronous} in 1978. In  the pre-2010 methods \cite{BMR1997asyn-multisplit,bertsekas1983distributed,Baz200591,el1998flexible} and the review~\cite{Frommer2000201}, each agent updates its own subset of coordinates. Convergence is established under the \emph{$P$-contraction} condition and its variants~\cite{bertsekas1983distributed}. Papers~\cite{Baz200591,Baz1998429} show convergence for async-parallel iterations with simultaneous reading and writing to the same set of components. Unbounded but stochastic delays are considered in~\cite{Strikwerda2002125}.

Recently, random coordinate selection appeared in~\cite{Patrick_2015} for fixed-point problems. The works \cite{nedic2001distributed,recht2011hogwild,liu2013asynchronous,liu2014asynchronous,hsieh2015passcode} introduced async-parallel stochastic methods for function minimization.
For fixed-point problems,~\cite{Peng_2015_AROCK} introduced  async-parallel stochastic methods, as well as several applications.

\subsection{Contributions of This Paper}
The paper systematically discusses the CF properties found in both single and composite operators underlying many interesting applications. We introduce approaches to recognize CF operators and develop coordinate-update algorithms based on them.
We provide a variety of applications to illustrate our approaches.
In particular, we obtain new coordinate-update algorithms for image deblurring, portfolio optimization, second-order cone programming, as well as matrix decomposition. Our analysis also provides guidance to the implementation of coordinate-update algorithms by specifying how to compute certain operators and maintain certain quantities in memory. We also provide numerical results to illustrate the efficiency of the proposed coordinate update algorithms.

This paper does \emph{not} focus on the convergence perspective of coordinate update algorithms, though a convergence proof is provided in the appendix for a new primal-dual coordinate update algorithm. In general, in fixed-point algorithms, the iterate convergence is ensured by the monotonic decrease of the distance between the iterates and the solution set, while in minimization problems, the objective value convergence is ensured by the monotonic decrease of a certain energy function. The reader is referred to the existing literature for details.


The structural properties of operators discussed in this paper are irrelevant to  the convergence-related properties such as nonexpansiveness (for an operator) or convexity (for a set or function). Hence, the algorithms developed can be still applied to nonconvex problems.



\section{Coordinate Friendly Operators}\label{sec:cuf}

\subsection{Notation}
For convenience, we do not distinguish \emph{a coordinate} from \emph{a block of coordinates} throughout this paper. We assume our variable $x$ consists of $m$ coordinates: $$x = (x_1, \ldots, x_m) \in\HH := \HH_1 \times \cdots \times \HH_m\quad\mbox{and}\quad x_i\in\HH_i,~ i=1,\ldots,m.$$ For simplicity, we assume that $\HH_1,\ldots,\HH_m$ are finite-dimensional real Hilbert spaces, though most results hold for general Hilbert spaces.
A function maps from $\HH$ to $\RR$, the set of real numbers, and an operator maps from $\HH$ to $\GG$, where the definition of $\GG$ depends on the context. \cut{Again, we define $[m]$ as the index set $\{1, \ldots, m\}$.}\cut{The operator $\cT$ and those operators in the splitting methods in \S\ref{sec:splitting} map from $\HH$ to $\HH$.}

Our discussion often involves two points $x,x^+\in\HH$ that \emph{differ over one coordinate}: there exists an index $i\in [m]$ and a point $\delta\in\HH$ supported on $\HH_i$, such that
\beq\label{singleupdate}
x^+ = x+\delta.
\eeq
Note that $x^+_j=x_j$ for all $j\not=i$. Hence, $x^+=(x_1,\ldots,x_i+\delta_i,\ldots,x_m)$.
\begin{definition}[number of operations]
We let $\nops{a}{b}$ denote \emph{the number of basic operations} that it takes to compute the quantity $b$ from the  input $a$.
\end{definition}
For example, $\nops{x}{(\cT x)_i}$ denotes the number of operations to compute the $i$th component of $\cT x$ given $x$. We explore the possibility to compute  $(\cT x)_i$ with much fewer operations than what is needed to first compute $\cT x$ and  then take its $i$th component. 

\subsection{Single Coordinate Friendly Operators}
This subsection studies a few classes of CF operators and then formally defines the CF operator. We motivate the first class through an example.

In the example below,  we let $A_{i,:}$ and $A_{:,j}$ be the $i$th row and $j$th column of a matrix $A$, respectively. Let $A^\top$ be the transpose of $A$ and $A^\top_{i,:}$ be $(A^\top)_{i,:}$, i.e., the $i$th row of the transpose of $A$.

\begin{example}[least squares I]\label{ex:lsq1}
Consider the least squares problem
\begin{equation}
\Min_{x} f(x):=\frac{1}{2} \|A x - b\|^2,\label{lsq}
\end{equation}
where $A \in \RR^{p \times m}$ and $b \in \RR^p$. In this example, assume that $m=\Theta(p)$, namely, $m$ and $p$ are of the same order.  We compare the full update of gradient descent to its coordinate update.\footnote{Although gradient descent is seldom used to solve least squares, it often appears as a part in first-order algorithms for problems involving a least squares term.}
The full update is referred to as the iteration
$x^{k+1}= \cT x^k $ where $\cT$ is given by
\begin{equation}\label{eq:LSfull}
\cT x:=x-\eta\nabla f(x)=x-\eta A^\top Ax + \eta A^\top b.
\end{equation}
Assuming that $ A^\top A$ and $ A^\top b$ are already computed, we have $\nops{x}{\cT x}=O(m^2)$. The coordinate update at the $k$th iteration performs
$$x_{i_k}^{k+1}=(\cT x^k)_{i_k} =x^k_{i_k}-\eta\nabla_{i_k} f(x^k),$$
and $x_j^{k+1}=x_j^{k},\forall j\neq i_k$, where $i_k$ is some selected coordinate.

Since for all $i$, $\nabla_i f(x^k)=\left(A^\top (Ax-b)\right)_{i}=(A^\top A)_{i,:}\cdot x-(A^\top b)_{i}$,
we have $\nops{x}{(\cT x)_i }=O(m)$ and thus $\nops{x}{(\cT x)_i }=O(\frac{1}{m}\nops{x}{\cT x})$. Therefore, the coordinate gradient descent is computationally worthy.
\end{example}
The operator $\cT$ in the above example is a special \emph{Type-I CF} operator.
\begin{definition}[Type-I CF]
For an operator $\cT: \HH\to\HH$, let $\nops{x}{(\cT x)_i}$ be the number of operations for computing the $i$th coordinate of $\cT x$ given $x$ and $\nops{x}{\cT x}$ the number of operations for computing $\cT x$ given $x$.
We say $\cT$ is \emph{Type-I CF} (denoted as $\cF_1$) if for any $x\in\HH$ and $i\in [m]$, it holds
$$\nops{x}{(\cT x)_i}= O\bigg(\frac{1}{m}\nops{x}{\cT x}\bigg).$$
\end{definition}
\begin{example}[least squares II]\label{ex:lsq2}
We can implement the coordinate update in Example~\ref{ex:lsq1} in a different manner by maintaining the result $\cT x^k$ in the memory. This approach works when $m=\Theta(p)$ or $p\gg m$. The full update~\eqref{eq:LSfull} is unchanged.
At each coordinate update, from the maintained quantity $\cT x^k$, we immediately obtain $x_{i_k}^{k+1}=(\cT x^k )_{i_k}$. But we need to update $\cT x^k$ to $\cT x^{k+1}$.
Since $x^{k+1}$ and $x^k$  differ only over  the coordinate $i_k$, this update can be computed as
$$\cT x^{k+1}=\cT x^k+x^{k+1}-x^k-\eta(x_{i_k}^{k+1}-x_{i_k}^k)(A^\top A)_{:,i_k} , $$
which is a scalar-vector multiplication followed by vector addition, taking only $O(m)$ operations. Computing $\cT x^{k+1}$ from scratch involves  a matrix-vector multiplication, taking $O(\nops{x}{\cT (x)})=O(m^2)$ operations.
Therefore,
$$\nops{\{ x^k,\cT x^k, x^{k+1}\}}{ \cT x^{k+1}}= O\bigg(\frac{1}{m}\DIFdelbegin  \nops{x^{k+1}}{\cT  x^{k+1}} \bigg).$$
\end{example}
The operator $\cT$ in the above example is a special \emph{Type-II CF} operator.

\begin{definition}[Type-II CF]
An operator $\cT$ is called \emph{Type-II CF} (denoted as $\cF_2$) if, for any $i,x$ and $x^+:=\big(x_1,\ldots,(\cT x)_i,\ldots,x_m\big)$, the following holds
\begin{equation}\label{op-cuf2} \nops{\{ x,\cT x, x^+\}}{ \cT x^+}= O\bigg(\frac{1}{m}\nops{x^+}{\cT x^+}\bigg).
\end{equation}
\end{definition}
The next example illustrates an efficient coordinate update by maintaining certain quantity other than $\cT x$.
\begin{example}[least squares III]\label{ex:lsq3}
For the case $p\ll m$, we should avoid pre-computing the relative large matrix $A^\top A$, and it is cheaper to compute $A^\top(Ax)$ than $(A^\top A)x$. Therefore, we change the implementations of both the full and coordinate updates in Example \ref{ex:lsq1}. In particular, the full update
$$x^{k+1}=\cT x^k =x^k-\eta\nabla f(x^k)=x^k-\eta A^\top (Ax^k-b),$$
pre-multiplies $x^k$ by $A$ and then $A^\top$. Hence,
$\nops{x^k}{\cT(x^k)}=O(mp)$.

We change the coordinate update to maintain the intermediate quantity $Ax^k$. In the first step, the coordinate update computes
$$(\cT x^k)_{i_k}=x^k_{i_k}-\eta (A^\top (Ax^k)-A^\top b)_{i_k},$$
by pre-multiplying $Ax^k$ by $A^{\top}_{i_k,:}$. Then, the second step updates $Ax^k$ to $Ax^{k+1}$ by adding $(x^{k+1}_{i_k}-x^k_{i_k}) A_{:,i_k}$ to $A x^k$. Both steps take $O(p)$ operations, so
\begin{displaymath}{\nops{\{x^k,Ax^k\}}{\{x^{k+1},Ax^{k+1}\}}=O(p)=O\left(\frac{1}{m}\nops{x^k}{\cT x^k}\right)}.\end{displaymath}
\end{example}
Combining Type-I and Type-II CF operators with the last example, we arrive at the following CF definition.

\begin{definition}[CF operator]
We say that an operator $\cT:\HH\to\HH$ is \emph{CF} if, for any $i,x$ and $x^+:=\big(x_1,\ldots,(\cT x)_i,\ldots,x_m\big)$,
the following holds
\begin{equation}\label{op-cuf} \nops{\{x,\cM(x)\}}{\{x^+,\cM(x^+)\}}= O\bigg(\frac{1}{m}\nops{x}{\cT x}\bigg),
\end{equation}
where $\cM(x)$ is some quantity maintained in the memory to facilitate each coordinate update and refreshed to $\cM(x^+)$. $\cM(x)$ can be empty, i.e., except $x$, no other varying quantity is maintained.
\end{definition}

\cut{\begin{definition}[CF operator]
We say that an operator $\cT:\HH\to\GG$ is \emph{coordinate-update (CU) friendly} if for a generic $\eta\in \HH$, $\forall i\in\{1,\ldots,\dim\GG\}$ and $\forall j\in\{1,\ldots,\dim\HH\}$:
\begin{equation}\label{op-cuf}\nops{\{x, \eta_j, \cT x\} }{(\cT (x+\eta_j e_j))_i}\le O\bigg(\frac{1}{\dim \GG}\nops{\{x,\eta,\cT x\}}{\cT (x+\eta)}\bigg).
\end{equation}
(Maintaining quantities other than $x$ and $\cT x$ in  memory is allowed.)
The \emph{set} of CF operators is denoted by  $\cF$. \commzp{$e_j$ is not defined.}
\end{definition}
}

The left-hand side of~\eqref{op-cuf} measures the cost of performing one coordinate update (including the cost of updating $\cM(x)$ to $\cM(x^+)$)  while the right-hand side measures the average per-coordinate cost of updating all the coordinates together. When~\eqref{op-cuf} holds, $\cT$ is amenable to coordinate updates.

\cut{\begin{example}Consider for $a,x\in\RR^m$, $$f(x):=\frac{1}{2}\big(\max(0,1-\beta a^\top x)\big)^2,$$ which is the squared hinge loss function. Consider the operator
\beq\label{qhlossT}
\cT x:=\nabla f(x)=-\beta \max(0,1-\beta a^\top x) a.
\eeq
Let us maintain $\cM(x)=a^\top x$. For arbitrary $x$ and $i$,  let $$x^+_i:=(\cT x)_i=-\beta \max(0,1-\beta a^\top x) a_i$$ and $x^+_j:=x_j,\,\forall j\neq i$. Then, computing $x^+_i$ from $x$ and $a^\top x$ takes $O(1)$ (as $a^\top x$ is maintained), and computing $a^\top x^+$ from $x^+_i-x_i$ and $a^\top x$ costs $O(1)$. Formally, we have
\begin{align*}
&\nops{\{x,a^\top x\}}{\{x^+,a^\top x^+\}}\\
=&\nops{\{x,a^\top x\}}{x^+}+\nops{\{a^\top x,x^+_i-x_i\}}{a^\top x^+}\\
=&O(1)+O(1)=O(1).
\end{align*}
On the other hand, $\nops{x}{\cT x}=O(m)$. Therefore,~the inequality~\eqref{op-cuf} holds, and $\cT$ defined in~\eqref{qhlossT} is CF.
\end{example}
}
%
%
%
%
By definition, a Type-I CF operator $\cT$ is CF without maintaining any quantity, i.e., $\cM(x)=\emptyset$.

A Type-II CF operator $\cT$ satisfies~\eqref{op-cuf} with $\cM(x)= \cT x$, so it is also CF. Indeed, given any $x$ and $i$, we can compute $x^+$ by immediately letting $x^+_i=(\cT x)_i$ (at $O(1)$ cost) and keeping $x^+_{j}=x_{j},\,\forall j\neq i$; then, by~\eqref{op-cuf2}, we update $\cT x$ to $\cT x^+$ at a low cost. Formally, letting $\cM(x)= \cT x$, 
\begin{align*}
&\nops{\{x, \cM(x)\}}{\{x^+,\cM(x^+)\}}\\
\le &\nops{\{x, \cT x\}}{x^+} + \nops{\{x, \cT x, x^+\}}{\cT x^+}\\
\overset{\eqref{op-cuf2}}= & O(1) +O\bigg(\frac{1}{m}\nops{x^+}{\cT x^+}\bigg)\\
= & O\bigg(\frac{1}{m}\nops{x}{\cT x}\bigg).
\end{align*} 

%
\cut{ let $a_i$ be the $i$th row of $A$. We have $\nops{x}{\cT(x)}=\nops{x}{Ax+b}=m^2+m$, and  $\nops{x}{(\cT }$}
\DIFdelbegin 

In general, the set of CF operators is much larger than the union of Type-I and Type-II CF operators.

Another important subclass of CF operators are operators $\cT:\HH\to\HH$ where   $(\cT x)_i$ only depends on one, or a few, entries among $x_1,\ldots,x_m$.\cut{, obviously $(\cT x)_i$ is relatively easy to compute. Therefore, they belong to $\cF$.} Based on how many input coordinates they depend on, we partition them into three subclasses.
 \DIFdelbegin 

\DIFdelend \begin{definition}[separable operator]\label{def:sep-op} Consider $\mathfrak{T}:=\{\cT ~|~\cT:\HH\to\HH\}$. We have the partition $\mathfrak{T}=\cC_1\cup\cC_2\cup\cC_3$, where
\cut{We divide the operators into three categories base on the dependency of other coordinates when updating one coordinate.}
\begin{itemize}
\item \emph{separable operator:} $\cT\in\cC_1$ if, for any index $i$, there exists $\cT_i: \HH_i \rightarrow \HH_i$ such that $(\cT x)_i = \cT_i x_i$, that is,   $(\cT x)_i$ only depends on $x_i$.
\item \emph{nearly-separable operator:} $\cT\in \cC_2$ if, for any index $i$, there exists $\cT_i$ and index set $\mathbb{I}_i$ such that $(\cT x)_i = \cT_i(\{x_j\}_{j\in \mathbb{I}_i})$ with $|\mathbb{I}_i| \ll m$, that is, each $(\cT x)_i$ depends on a few  coordinates of $x$.
\item \emph{non-separable operator:} $\cC_3:=\mathfrak{T}\setminus(\cC_1\cup \cC_2)$. If $\cT\in\cC_3$, there exists some  $i$ such that $(\cT x)_i$ depends on many coordinates of $x$.
\end{itemize}
\end{definition}

\cut{For operators in $\cC_1$, the evaluation of $(\cT x)_i$ only involves coordinates $x_i$ and small operator $\cT_i$.  Thus, all the coordinates and the $\cT_i$s are independent with each other, and this is ideal case for coordinate update. For operators in $\cC_2$, the evaluation of $(\cT x)_i$ needs knowledge of some components of $x$.}

Throughout the paper, we assume the coordinate update of a (nearly-) separable operator costs roughly the same for all coordinates. Under this assumption, separable operators are both Type-I CF and Type-II CF, and nearly-separable operators are Type-I CF.\footnote{\label{note1}Not all nearly-separable operators are Type-II CF. Indeed, consider a sparse matrix $A\in \RR^{m\times m}$ whose  non-zero entries are only located  in the last column. Let $\cT x=Ax$ and $x^+=x+\delta_m$. As $x^+$ and $x$ differ over the last entry, $\cT x^+=\cT x+(x^+_m-x_m)A_{:,m}$ takes $m$ operations. Therefore, we have $\nops{\{x,\cT x,x^+\}}{\cT  x^+}=O(m)$. Since $\cT x^+=x^+_mA_{:,m}$ takes $m$ operations, we also have $\nops{x^+}{\cT x^+}=O(m)$. Therefore,~\eqref{op-cuf2} is violated, and there is no benefit  from maintaining $\cT x$.}  

\cut{There are CF operators that are not Type-I . For example, let $f(x)=\sum_{i=1}^m \log [1+\exp(-b_i a_i^\top x)]$. The gradient operator $\nabla f$ is not Type-I CF, but it is CF by caching $a_i^\top x,\, \forall i$.}

\cut{Next, we defined \emph{easy-to-update} operators. While they are not CF,  $\cT x$ can  be easily updated after the input $x$ is changed on one or a few of its coordinates. Therefore, by maintaining $\cT x$  in memory, they are also amenable to coordinate update computing. 

\begin{definition}[Easy-to-update]
We say an operator $\cT:\HH\to\GG$ is  \emph{easy-to-update} if, for any $x\in \HH$ and  $c\in\RR$, it holds that $\forall i=1,\ldots,\dim\HH$: $$\nops{\{x,\cT x, c\}}{\cT(x+c e_i)}\le O\left(\frac{1}{\dim\HH}\big(\nops{\{x, c\}}{x+c e_i} +\nops{x+c e_i}{\cT (x+c e_i)}\big)\right).$$
The \emph{set} of {easy-to-update} operators is denoted by $\cE$.
\end{definition}

\cut{A (nearly-) separable operator $\cT$ (i.e., in $\cC_1\cup\cC_2$) is easy-to-coordinate-update if $\cT x$ is maintained in the memory. }

In coordinate-update algorithms, easy-to-update operators \emph{can be treated as} a subclass of CF operators. Seemingly, an easy-to-update operator is not CF since it is not necessarily cheaper to compute $(\cT x)_i$ than the entire $\cT  x$. However, computing the entire $\cT  x$ and changing the entire $x$ add a significantly higher cost to \emph{the next iteration} than computing just $(\cT x)_i$, changing just $x_i$, and taking advantage of the easy-to-update property.
}

%

\cut{
\begin{definition}[Easy-to-update]
We say an operator $\cT:\HH\to\GG$ is  \emph{easy-to-update} if, for any $x$ and  $\delta=\sum_{i\in\II}\eta_i e_i$ of arbitrary coefficient $\eta_i\in\RR$ and a small index set $\II$ ($|\II|\ll m$), it holds that  $$\nops{\{x,\cT x,(\eta_i)_{i\in\II}\}}{\cT(\tilde x)}\ll\nops{\{x,(\eta_i)_{i\in\II}\}}{\tilde x} +\nops{\tilde x}{\cT\tilde x},$$
where $\tilde x=x+\delta$. The \emph{set} of {easy-to-update} operators is denoted by $\cE$.
\end{definition}
}

 \cut{When one or a few coordinates of $x$ are changed, giving the new point $\tilde x$, we compute $(\cT \tilde x)_i$ by refreshing $\cT x$ at a low cost. It would be much more expensive to first form $\tilde x$ and then compute $T\tilde x$.}
\subsection{Examples of CF Operators}\label{sec:exp-cuf}
In this subsection, we give  examples of CF operators arising in different areas including linear algebra, optimization, and machine learning.
\begin{example}[(block) diagonal matrix]\label{exp:diagmat}
Consider the diagonal matrix
$$A = \begin{bmatrix}a_{1,1} & ~ & 0 \\  & \ddots&  \\ 0& ~& ~a_{m,m}\end{bmatrix}\in\RR^{m\times m}.$$
Clearly $\cT:x\mapsto Ax$ is separable.
\end{example}
\begin{example}[gradient and proximal maps of a separable function]
Consider a \emph{separable} function
$$f(x) = \sum_{i=1}^m f_i(x_i).$$
Then, both $\nabla f$ and $\prox_{\gamma f}$ are separable, in particular,
$$(\nabla f(x))_i = \nabla f_i (x_i) \quad \mbox{and}\quad (\prox_{\gamma f}(x))_i= \prox_{\gamma f_i} (x_i).$$\cut{$$\nabla f(x) = \begin{bmatrix} \nabla f_1 (x_1) \\ \vdots \\ \nabla f_m(x_m) \end{bmatrix} \quad \mbox{and}\quad \prox_{\gamma f}(x) = \begin{bmatrix}  \prox_{\gamma f_1} (x_1) \\ \vdots \\  \prox_{\gamma f_m}(x_m) \end{bmatrix}.$$}
Here, $\prox_{\gamma f} (x)$ ($\gamma >0$) is the \emph{proximal operator} that we  define in Definition~\ref{def-prox-map} in Appendix~\ref{sec:op-concept}. 
\end{example}

\begin{example}[projection to box constraints]\label{exp:proj-box}
Consider the ``box" set $B:=\{x:a_i\leq x_i\leq b_i,~i \in [m]\}\subset\RR^m$. Then, the projection operator $\prj_{B}$ is separable. Indeed,
$$\big(\prj_{B}(x)\big)_i=\max(b_i,\,\min(a_i,\, x_i)).$$
\end{example}

\begin{example}[sparse matrices] If every row of the matrix $A\in\RR^{m\times m}$ is sparse,   $\cT :x\mapsto Ax$ is nearly-separable.

Examples of sparse matrices arise from various finite difference schemes for differential equations, problems defined on sparse graphs. When most pairs of a set of random variables  are conditionally independent, their inverse covariance matrix is sparse.
\end{example}

\begin{example}[sum of sparsely supported functions]
Let $E$ be a class of index sets and every $e\in E$ be a small subset of $[m]$, $|e|\ll m$. In addition $\#\{e:i\in e\}\ll \#\{e\}$ for all $i\in [m]$.
Let $x_e:=(x_i)_{i\in e}$, and
$$f(x) = \sum_{e \in E} f_e (x_e).$$
The gradient map $\nabla f$ is nearly-separable.

An application of this example arises in wireless communication over  a graph of $m$ nodes. Let each $x_i$ be the spectrum assignment to node $i$, each $e$ be a neighborhood of nodes,  and each   $f_e$ be a utility function. The input of $f_e$ is $x_e$ since the utility depends on the spectra assignments in the neighborhood.

In machine learning, if each observation only involves a few features, then each function of the optimization objective will depend on a  small number of components of $x$. This is the case when graphical models are used~\cite{rue2005gaussian,bengio2006label}. \cut{The gradient of $f(x)$ will also only depend on a few coordinates of $x$. \commwy{could you be specific in the machine learning example?}}
\end{example}

\cut{ 
\begin{example}[matrix-vector multiplication]\label{mtxexam}
With a (dense) matrix $A\in\RR^{m\times m}$, the operator $\cT:x\mapsto A x$ is non-separable yet CF. Indeed, let $a_i$ be the $i$th row of $A$. Then, we can compare $\nops{x}{\cT(x)}=\nops{x}{Ax}=m^2$, and  $\nops{x}{(\cT x)_i}=\nops{x}{a_i^\top  x}=m$ is much smaller.
\end{example}
\begin{example}[affine transformation preserves $\cC_1$]
Let $\cT\in\cC_1$, $b \in \HH$. Then $\cT':x\mapsto\cT x + b$ belongs to  $\cC_1$.
\end{example}
}
\cut{\begin{example}[projection to the Euclidean ball]\label{ex:projball}
Let $B_2=\{x\in\RR^m:\|x\|_2\leq r\}$ and $$\cT x:=\prj_{B_2}(x).$$ By definition, 
$\cT(x) =\xi x,~\mbox{where}~\xi=\min\left\{1,\frac{r}{\|x\|_2}\right\}.$
Since $\cT x$ depends on all the entries of $x$,  $\cT$ is non-separable.

Nonetheless, the norm map $\cT':x\mapsto \|x\|_2$ is easy-to-maintain. Indeed, given $x$, we can maintain $\cT'(x)=\|x\|_2$. If $x_i$ is updated,  letting $e_i$ be the $i$th standard basis and writing $\bar{x}=x+\eta e_i$, it follows $\|\bar{x}\|_2=\sqrt{\|x\|_2^2+2\eta x_i+\eta^2}$. Therefore,   $\nops{\{x,\|x\|_2,\eta\}}{\|\bar{x}\|_2}=O(1)$ while $\nops{\bar{x}}{\|\bar{x}\|_2}=O(m)$.

\end{example}
}


\cut{\begin{example}[$\cC_2 + \cC_2= \cC_3$]
Let $A_1,A_2\in\RR^{m\times m}$ be sparse matrices. If $A_1+A_2$ is a dense matrix, then $\cT x= Ax=A_1x+A_2x$ belongs to $\cC_3$ though $\cT_1 x= A_1x$ and $\cT_2=A_2x$ both belong to $\cC_2$.
\end{example}
}


%
%
\cut{
Based on the categories of an operator and the examples, we know that all operators in $\cC_1$ are CF, operators in $\cC_2$ with $c\ll m$ is also CF. However, some operator in $\cC_3$ are CF, while others are not. There are other ways to determine whether a combination of several operators is CF or not. }

\cut{\begin{remark}
If the complexity of evaluating $\cT_2 x$ is smaller or equals to the complexity of evaluating $\cT_{1,i} x$, then whether the operators $\cT_1+\cT_2$ and $\cT_1\circ\cT_2$ is CF or not depends on whether $\cT_2$ is CF or not.
\end{remark}}


%
%
%

%



\cut{When we apply coordinate update algorithms, we may choose to store some intermediate variables which are easy to update after some coordinates are updated.}

\cut{
\begin{definition} Let $\mathfrak{M}[x\to \cT(x)]$ denote the number of  operations of computing $\cT(x)$ given $x$. Let $\mathfrak{M}[y\to \tilde y]$ denote the number of  operations of computing $\tilde y$ in place given  $y$.
\end{definition}}

\cut{
\begin{remark}By  definition, a separable operator is easy to update (and no quantity needs to be maintained), but not all nearly-separable operators easy to update. For example, consider a sparse matrix $A\in R^{m\times m}$ where all of its non-zero entries are in the last column. Let $\cT x=Ax$ and, given a point $x$ nd $\tilde x=x+\eta_me_m$, compute $(T(x+\eta_me_m))_m$. We have $\nops{\{x,\cT x,\eta_m\}}{\cT \tilde x}=O(m)$ and $\nops{\{x,\eta_m\}}{\tilde x} +\nops{\tilde x}{\cT(\tilde x)}=O(1)+O(m)=O(m)$. There is not benefit  from maintaining $\cT x$.
\end{remark}
}
\begin{example}[squared hinge loss function]\label{ex:qhloss}
Consider for $a,x\in\RR^m$, $$f(x):=\frac{1}{2}\big(\max(0,1-\beta a^\top x)\big)^2,$$ which is known as the squared hinge loss function. Consider the operator
\beq\label{qhlossT}
\cT x:=\nabla f(x)=-\beta \max(0,1-\beta a^\top x) a.
\eeq
Let us maintain $\cM(x)=a^\top x$. For arbitrary $x$ and $i$,  let $$x^+_i:=(\cT x)_i=-\beta \max(0,1-\beta a^\top x) a_i$$ and $x^+_j:=x_j,\,\forall j\neq i$. Then, computing $x^+_i$ from $x$ and $a^\top x$ takes $O(1)$ (as $a^\top x$ is maintained), and computing $a^\top x^+$ from $x^+_i-x_i$ and $a^\top x$ costs $O(1)$. Formally, we have
\begin{align*}
&\nops{\{x,a^\top x\}}{\{x^+,a^\top x^+\}}\\
\le&\nops{\{x,a^\top x\}}{x^+}+\nops{\{a^\top x,x^+_i-x_i\}}{a^\top x^+}\\
=&O(1)+O(1)=O(1).
\end{align*}
On the other hand, $\nops{x}{\cT x}=O(m)$. Therefore,~\eqref{op-cuf} holds, and $\cT$ defined in~\eqref{qhlossT} is CF.
\end{example}

\cut{\rev{\subsection{Compare Full Update with Coordinate Update}\label{lsqexperiment}
In this subsection, we compare the efficiency of four different coordinate update schemes (cyclic, cyclic permutation, random, and greedy with Gauss-Southwell rule) with the full gradient descent method for solving the least square problem
$$\Min_{x} \frac{1}{2} \|A x - b\|^2,$$
where $A \in \RR^{p \times m}$ and $b \in \RR^p$. We solve the above problem with the following update scheme
$$x^{k+1} = x^k - \eta_k A^{\top}(A x^k - b),$$
where $\eta_k$ is the step size. This test uses three datasets, which are summarized in Table \ref{tab:ls-data}.
\begin{table}[!hbtp]
\centering
 \begin{tabular}{lrrrr}
  \toprule
     & $p$  & $m$ & $A$ & $b$\\
   \midrule
   Dataset I & 1000 & 1000 & \texttt{diag([1:m])} & \texttt{ones(m, 1)} \\
   Dataset II & 1000 & 500 & \texttt{randn(m, n)} & \texttt{ones(m, 1)} \\
   Dataset III & 1000 & 500 & \texttt{rand(m, n)} & \texttt{ones(m, 1)} \\
   \bottomrule
\end{tabular}
 \caption{Three datasets for the least square problem\label{tab:ls-data}}
\end{table}

For both full gradient descent method, the step size $\eta$ is set to $\frac{2}{\|A\|_2^2}$. For the four coordinate update methods, if coordinate $i$ is selected, then $\eta_k$ is set to $\frac{1}{(A^{\top}A)_{ii}}$. Since all of the methods have same per epoch complexity, so we measure and compare the per epoch progress in terms of objective error. Figure \ref{fig:ls_a} shows that for Dataset I, both cyclic update, cyclic permutation, and greedy update converge to the optimal solution with one epoch. Random coordinate update converges to the optimal solution after $8$ epochs. However, due to the small step size ($\eta_k = 10^{-6}$), the gradient descent algorithm converges very slowly. For the other two datasets, we observe that greedy coordinate update converges faster than the other methods, and random coordinate update and cyclic shuffle coordinate update give consistent better performance than the full gradient descent algorithm.
\begin{figure}[!htbp] \centering
    \begin{subfigure}[b]{0.3\linewidth}
        \includegraphics[width=40mm]{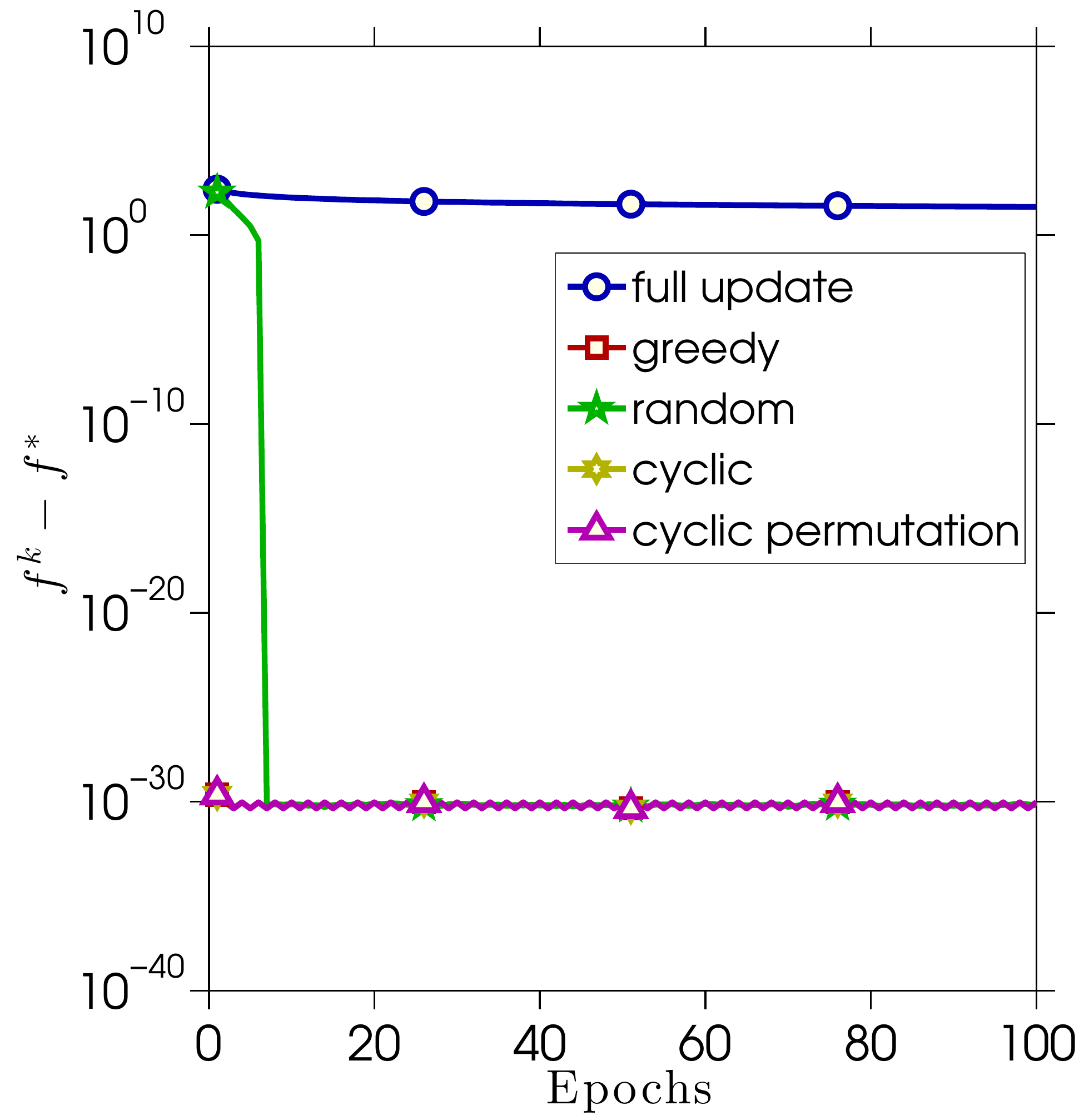}
        \caption{Dataset I}
        \label{fig:ls_a}
    \end{subfigure} %
    \quad
    \begin{subfigure}[b]{0.3\linewidth}
        \includegraphics[width=40mm]{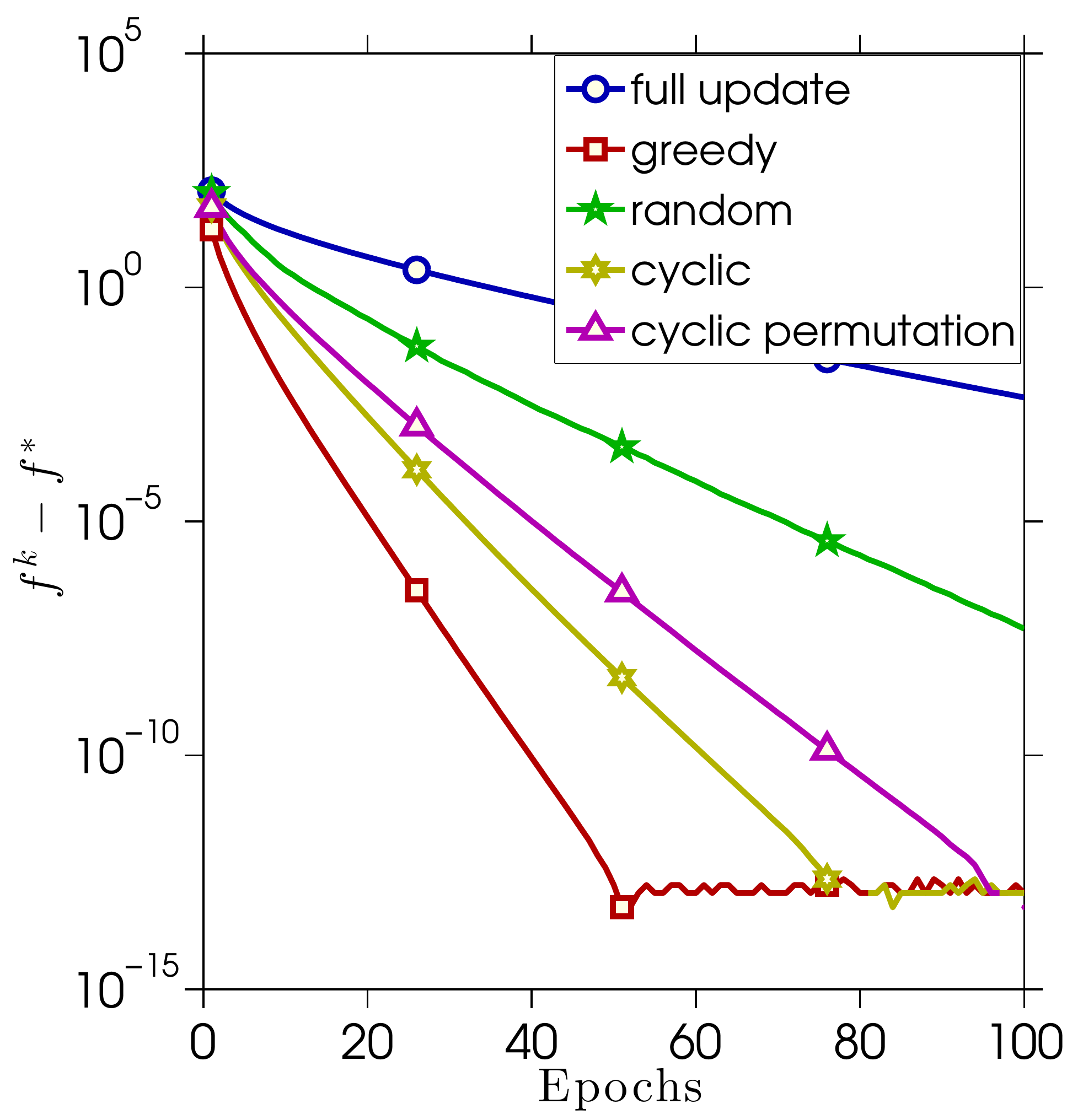}
        \caption{Dataset II}
        \label{fig:ls_b}
    \end{subfigure} %
    \quad
    \begin{subfigure}[b]{0.3\linewidth}
        \includegraphics[width=40mm]{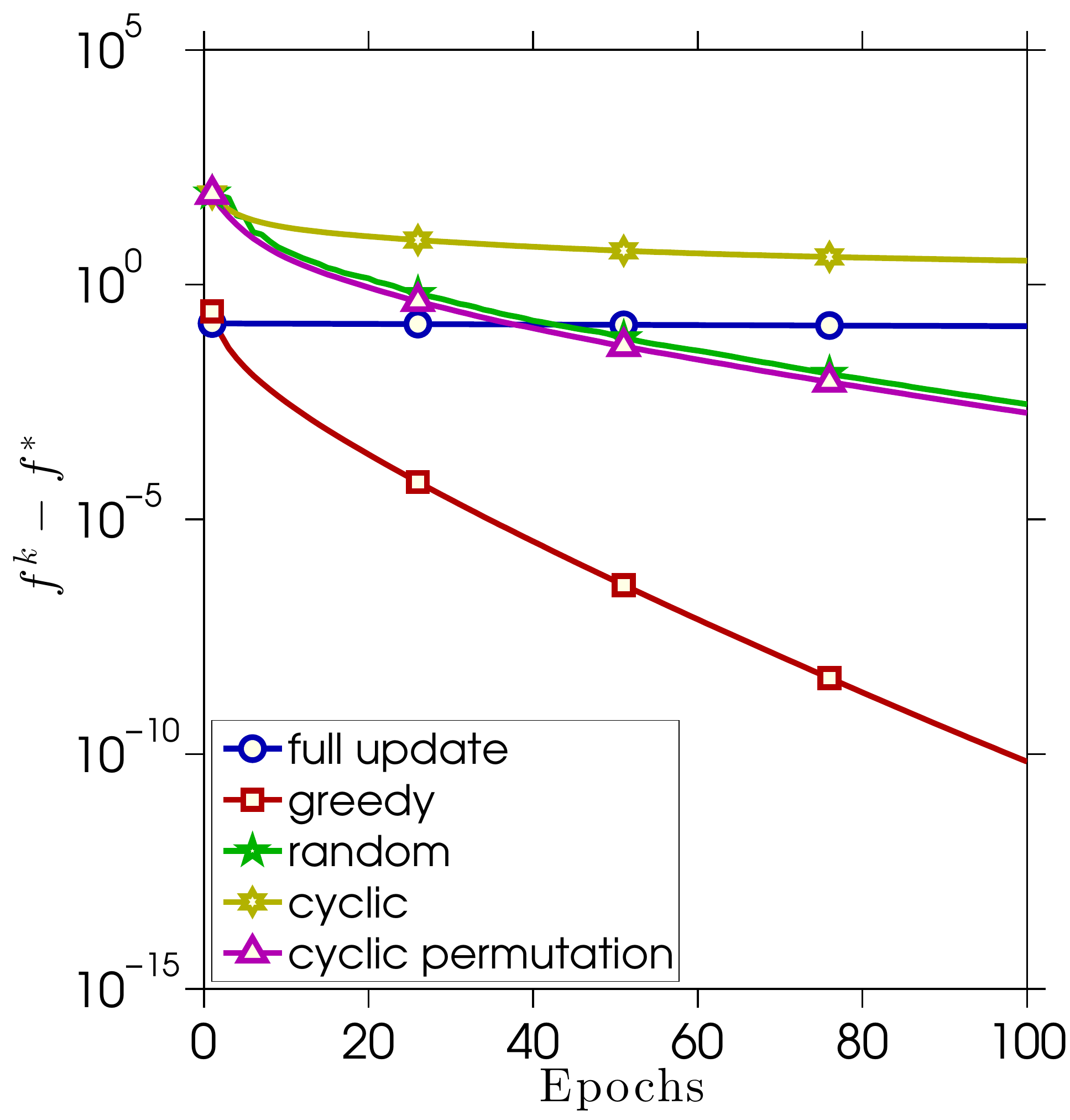}
        \caption{Dataset III}
        \label{fig:ls_c}
    \end{subfigure} %
    \caption{Compare the convergence of four different coordinate update algorithms with full gradient descent algorithm.}
    \label{fig:3s_results}
\end{figure}
}
}

\cut{
\begin{example}[scalar map pre-composing affine function]\label{exp:log-grad} Let $a_j\in \RR^m, b_j\in \RR$ and $\phi_j:\RR\to \RR$ be differentiable functions, $j=1,\ldots,p$. Let $$f(x)=\sum_{j=1}^p \phi_j(a_j^\top x +b_j).$$ Assume that evaluating $\phi'_j$ costs $O(1)$ for all $j$. Then, $\nabla f$ is CF (by maintaining $a_j^\top x+b_j,\,\forall j$), but it is neither Type-I nor Type-II CF if $A=[a_1,\ldots,a_p]^\top$ is dense. Indeed, let $$\cT_1 y:=A^\top y,\quad \cT_2 y := \Diag(\phi_1'(y_1),\ldots,\phi_p'(y_p)), \quad \cT_3 x := Ax+b.$$ Then $\nabla f(x)= \cT_1\circ\cT_2\circ\cT_3 x$. For any $x$ and $i\in[m]$, let $x^+_i=\nabla_i f(x)$ and $x^+_j=x_j,\forall j\neq i$, and let $\cM(x)=\cT_3 x$. We can first compute $\cT_2\circ\cT_3 x$ from $\cT_3 x$ for $O(p)$ operations, then compute $\nabla_i f(x)$ and thus $x^+$ from $\{x, \cT_2\circ\cT_3 x\}$ for $O(p)$, and finally update the maintained $\cT_3 x$ to $\cT_3 x^+$ from $\{x, x^+,\cT_3 x\}$ for another $O(p)$. Formally,
\begin{align*}
&\nops{\{x,\cT_3 x\}}{\{x^+, \cT_3x^+\}}\cr
=& \nops{\cT_3x}{\cT_2\circ\cT_3 x}+\nops {\{x,\cT_2\circ\cT_3 x\}}{x^+}+\nops{\{x,\cT_3 x, x^+\}}{\{\cT_3x^+\}}\cr
=& O(p) + O(p) +O(p)=O(p).\nonumber
\end{align*}
On the other hand, $\nops{x}{\nabla f(x)}=O(pm)$. Therefore
$\nabla f= \cT_1\circ\cT_2\circ\cT_3$ is CF.

Once $p=m$, $\cT_1,\cT_2,\cT_3$ all map from $\RR^m$ to $\RR^m$. Then, it is easy to check that $\cT_1$ is Type-I CF, $\cT_2$ is separable, and $\cT_3$ is Type-II CF. The last one is crucial since not maintaining $\cT_3 x$ would disqualify $\cT$ from CF. Indeed, to obtain $(\cT x)_i$, we must multiple $A_i^\top$ to all the entries of $\cT_2\circ\cT_3 x$, which in turn needs all the entries of $\cT_3 x$, computing which from scratch costs $O(pm)\gg O(p)$.

There are some  rules to preserve Type-I and Type-II CU-friendliness. For example, $\cT_1\circ \cT_2$ is still Type-I CF, and $\cT_2\circ\cT_3$ is still CF but there are counter examples where  $\cT_2\circ\cT_3$  can be neither Type-I or Type-II CF in general. Such properties are important for developing efficient coordinate update algorithms for complicated problems; see the next section.
\end{example}
}


\section{Composite Coordinate Friendly Operators}\label{sec:comp-cuf}
Compositions of two or more operators arise in  algorithms for problems that have composite functions, as well as  algorithms that are derived from operator splitting methods. To update the variable $x^k$ to $x^{k+1}$, two or more operators are sequentially applied, and therefore the structures of all  operators  determine whether the update is CF. This is where CF structures become less trivial but more interesting. This section studies composite CF operators. The exposition leads to the recovery of existing algorithms, as well as powerful new algorithms.

\cut{Operator splitting algorithms decompose awkward combinations of operators into simpler subproblems. They give rise to a lot of efficient algorithms. \rev{Aside from the widely used ADMM and the primal-dual algorithms \cite{chambolle2011first}\cite{condat2013primal}\cite{vu2013splitting}, the forward-backward splitting, the forward-backward-forward splitting, the Douglas-Rachford splitting, the forward-Douglas-Rachford splitting also find numerous applications \cite{daubechies2003iterative}\cite{combettes2005signal}\cite{o2014primal}\cite{briceno2015FDRS} in machine learning, signal processing and imaging. Morover, the development of operator splitting schemes can produce more powerful algorithms \cite{davis2015three}. Combining operator splitting and coordinate updating will give us algorithms enjoying benefits from both sides. However, naively applying coordinate update to existing algorithms may results in divergence or wrong answers.
\begin{example}[three block ADMM]
The problem:
\begin{equation}
\begin{array}{l}
\underset{x,y,z\in\RR^m}{\textnormal{minimize  }} ~f(x)+g(y)+h(z)\\
\textnormal{subject to } Ax+Bx+Cz=b,
\end{array}\label{3block}
\end{equation}
can be solved by the following ADMM iterations:
\begin{equation}
\left\{
\begin{array}{ll}
x^{k+1}&=\argmin_x f(x)+\frac{\eta}{2}\|Ax+By^k+Cz^k+\frac{s^k}{\eta}-b\|^2,\\
\begin{pmatrix}
y^{k+1}\\
z^{k+1}
\end{pmatrix}&=\argmin_{(y,z)^\top }g(y)+h(z)+\frac{\eta}{2}\|Ax^{k+1}+By+Cz+\frac{s^k}{\eta}-b\|^2,\\
s^{k+1}&=s^k+\eta(Ax^{k+1}+By^{k+1}+Cz^{k+1}-b).
\end{array}\right.
\end{equation}
However, if we apply sequential update to the $(y,z)$ subproblem, we will have the following iterative scheme:
\begin{equation}
\left\{
\begin{array}{l}
x^{k+1}=\argmin_x f(x)+\frac{\eta}{2}\|Ax+By^k+Cz^k+\frac{s^k}{\eta}-b\|^2,\\
y^{k+1}=\argmin_y g(y)+\frac{\eta}{2}\|Ax^{k+1}+By+Cz^k+\frac{s^k}{\eta}-b\|^2,\\
z^{k+1}=\argmin_z h(z)+\frac{\eta}{2}\|Ax^{k+1}+By^{k+1}+Cz+\frac{s^k}{\eta}-b\|^2,\\
s^{k+1}=s^k+\eta(Ax^{k+1}+By^{k+1}+Cz^{k+1}-b).
\end{array}\right.
\end{equation}
It is the direct extension of ADMM to three block, which is proved to be divergent in \citep{chen2014direct}
\end{example}}
This motivates us to study the mechanism to develop coordinate update algorithms based on operator splitting. In order to do this, we study when the combinations of operators is CF in  In \S\ref{sc:comb}. Then, in \S\ref{sec:splitting} review several operator splitting schemes and provide examples. \rev{We point out here that the coordinate update methods we propose, based on operator splitting schemes, all have convergence guarantee, at least for async-parallel (thus also stochastic) update \cite{Peng_2015_AROCK}}

There are many popular operator-splitting based algorithms, e.g., proximal gradient method, ADM, and primal-dual method, which reduce the original problem to simpler subproblems, each corresponding to a part of the objective or constraints. Coordinate updates can be combined with operator splitting to further simplify their subproblems and even offer better parallelism. Most operator splitting algorithms are sequential compositions of two or more operators. This section studies when their compositions are CF operators.

}

\subsection{Combinations of Operators}\label{sc:comb}
We start by an example with numerous applications. It is a generalization of Example~\ref{ex:qhloss}.
\begin{example}[scalar map pre-composing affine function]\label{exp:log-grad} Let $a_j\in \RR^m, b_j\in \RR$, and $\phi_j:\RR\to \RR$ be differentiable functions, $j \in [p]$. Let $$f(x)=\sum_{j=1}^p \phi_j(a_j^\top x +b_j).$$ Assume that evaluating $\phi'_j$ costs $O(1)$ for each $j$. Then, $\nabla f$ is CF. Indeed, let $$\cT_1 y:=A^\top y,\quad \cT_2 y := [\phi_1'(y_1);\ldots;\phi_p'(y_p)], \quad \cT_3 x := Ax+b,$$
where $A=[a_1^\top; a_2^\top; \ldots; a_p^\top]\in \RR^{p\times m}$ and $b=[b_1;b_2;\ldots;b_p]\in\RR^{p\times 1}$. Then we have $\nabla f(x)= \cT_1\circ\cT_2\circ\cT_3 x$. For any $x$ and $i\in[m]$, let $x^+_i=\nabla_i f(x)$ and $x^+_j=x_j,\forall j\neq i$, and let $\cM(x):=\cT_3 x$. We can first compute $\cT_2\circ\cT_3 x$ from $\cT_3 x$ for $O(p)$ operations, then compute $\nabla_i f(x)$ and thus $x^+$ from $\{x, \cT_2\circ\cT_3 x\}$ for  $O(p)$ operations, and finally update the maintained $\cT_3 x$ to $\cT_3 x^+$ from $\{x, x^+,\cT_3 x\}$ for another $O(p)$ operations. Formally,
\begin{align*}
&\nops{\{x,\cT_3 x\}}{\{x^+, \cT_3x^+\}}\cr
\le& \nops{\cT_3x}{\cT_2\circ\cT_3 x}+\nops {\{x,\cT_2\circ\cT_3 x\}}{x^+}+\nops{\{x,\cT_3 x, x^+\}}{\{\cT_3x^+\}}\cr
=& O(p) + O(p) +O(p)=O(p).\nonumber
\end{align*}
Since $\nops{x}{\nabla f(x)}=O(pm)$, therefore
$\nabla f= \cT_1\circ\cT_2\circ\cT_3$ is CF.

If $p=m$, $\cT_1,\cT_2,\cT_3$ all map from $\RR^m$ to $\RR^m$. Then, it is easy to check that $\cT_1$ is Type-I CF, $\cT_2$ is separable, and $\cT_3$ is Type-II CF. The last one is crucial since not maintaining $\cT_3 x$ would disqualify $\cT$ from CF. Indeed, to obtain $(\cT x)_i$, we must multiply $A_i^\top$ to all the entries of $\cT_2\circ\cT_3 x$, which in turn needs all the entries of $\cT_3 x$, computing which from scratch would cost $O(pm)$.

There are general rules to preserve Type-I and Type-II CF. For example, $\cT_1\circ \cT_2$ is still Type-I CF, and $\cT_2\circ\cT_3$ is still CF, but there are counter examples where  $\cT_2\circ\cT_3$  can be neither Type-I nor Type-II CF. Such properties are important for developing efficient coordinate update algorithms for complicated problems; we will formalize them in the following.
\end{example}

The operators $\cT_2$ and $\cT_3$ in the above example are prototypes of \emph{cheap} and \emph{easy-to-maintain} operators from $\HH$ to $\GG$ that arise in operator compositions.
\begin{definition}[cheap operator] For a composite operator $\cT=\cT_1 \circ\cdots \circ \cT_p$, an operator $\cT_i:\HH\to\GG$ is cheap if $\nops{x}{\cT_i x}$ is less than or equal to the number of remaining coordinate-update operations, in order of magnitude.
\end{definition}

\begin{definition}[easy-to-maintain operator]
For a composite operator $\cT=\cT_1 \circ\cdots \circ \cT_p$, 
the operator $\cT_p:\HH\to\GG$ is \emph{easy-to-maintain}, if for any $x,i,x^+$ satisfying~\eqref{singleupdate}, $\nops{\{ x,\cT_p x, x^+\}}{\cT_p x^+}$
is less than or equal to the number of remaining coordinate-update operations, in order of magnitude, \emph{or} belongs to $O(\frac{1}{\dim \GG}\nops{x^+}{\cT x^+})$.
\end{definition}
\cut{\begin{example}[projection to the Euclidean ball]\label{ex:projball}
Let $B_2=\{x\in\RR^m:\|x\|_2\leq r\}$ and $$\cT x:=\prj_{B_2}(x).$$ By definition, 
$\cT(x) =\xi x,~\mbox{where}~\xi=\min\left\{1,\frac{r}{\|x\|_2}\right\}.$
Since $\cT x$ depends on all the entries of $x$,  $\cT$ is non-separable.

Nonetheless, the norm map $\cT':x\mapsto \|x\|_2$ is easy-to-maintain. Indeed, given $x$, we can maintain $\cT'(x)=\|x\|_2$. If $x_i$ is updated,  letting $e_i$ be the $i$th standard basis and writing $\bar{x}=x+\eta e_i$, it follows $\|\bar{x}\|_2=\sqrt{\|x\|_2^2+2\eta x_i+\eta^2}$. Therefore,   $\nops{\{x,\|x\|_2,\eta\}}{\|\bar{x}\|_2}=O(1)$ while $\nops{\bar{x}}{\|\bar{x}\|_2}=O(m)$.

\end{example}
}

\cut{\begin{example}[Projection to Euclidean ball]
Let $\cT \, x = \Proj_{\|x\|_2 \leq r} (x)$, $\tilde x  = x + \sum_{i \in \II} \eta_i e_i$
\begin{equation}
\cT \, (\tilde x) =
\left\{
\begin{array}{ll}
\tilde x &\text{ if } \|\tilde x\| \leq r \\
\frac{r}{\|\tilde x\|} \tilde x &\text{ if } \|\tilde x\| > r,
\end{array}
\right.
\end{equation}
Note that $\|\tilde x\|^2 = \|x\|^2 + 2 \sum_{i \in \II} \langle x_i + \eta_i, \eta_i \rangle$. If $ \|\tilde x\| \leq r$, then $\cT$ is easy to update if and only if $ \|\tilde x\| \leq r$.
\end{example}}

The splitting schemes in \S\ref{sec:splitting} below will be based on $\cT_1+\cT_2$ or $\cT_1\circ\cT_2$, as well as a sequence of such combinations. If $\cT_1$ and $\cT_2$ are both CF, $\cT_1+\cT_2$ remains CF, but $\cT_1\circ\cT_2$ is not necessarily so. This subsection discusses how $\cT_1\circ\cT_2$ inherits the properties from $\cT_1$ and $\cT_2$. Our results are summarized in Tables \ref{table:comp1-op} and \ref{table:comp-op} and explained in detail below.

The combination \cut{$\cT_1+\cT_2$ and }$\cT_1\circ \cT_2$ generally inherits the \emph{weaker} property from $\cT_1$ and $\cT_2$. 

The separability ($\cC_1$) property  is  preserved by composition. If $\cT_1,\ldots,\cT_n$ are separable, then $\cT_1\circ\cdots\circ \cT_n$ is separable.  However, combining  nearly-separable ($\cC_2$) operators  may not yield a nearly-separable operator since\cut{ Even if  $\cT_1,\ldots,\cT_n\in\cC_2$ and each $(\cT_jx)_i$ depends on $c>1$ coordinates of $x$, the composition $(\cT_1\circ\cdots\circ \cT_n x)_i$ can depend on the much more $\min\{c^n,m\}$ components of $x$, as} composition introduces more dependence among the input entries. Therefore, composition of nearly-separable operators can be either nearly-separable or non-separable.


\begin{table}
\begin{center}
\begin{tabular}{c|l|l|l}
\hline
Case & $\cT_1\in$ & $\cT_2\in$ & $~(\cT_1\circ\cT_2)\in ~$\\\hline\hline
1 & $\cC_1$ (separable) & $\cC_1$, $\cC_2$, $\cC_3$ & $\cC_1$, $\cC_2$, $\cC_3$, respectively \\\hline
2 & $\cC_2$ (nearly-sep.)& $\cC_1$, $\cC_3$ & $\cC_2$, $\cC_3$, resp. \\\hline
3 & $\cC_2$ & $\cC_2$ & $\cC_2$ or $\cC_3$, case by case \\\hline
4 & $\cC_3$ (non-sep.) & $\cC_1\cup\cC_2\cup\cC_3$ & $\cC_3$ \\\hline
\end{tabular}\end{center}
\caption{$\cT_1\circ\cT_2$ inherits the weaker separability property from those of $\cT_1$ and $\cT_2$.}\label{table:comp1-op}\end{table}

\begin{table}
\begin{center}
\begin{tabular}{c|l|l|l|l}
\hline
Case & $\cT_1\in$ & $\cT_2\in$ & $~(\cT_1\circ\cT_2)\in ~$& Example\\\hline\hline

5 & $\cC_1\cup\cC_2$ & $\cF$, $\cF_1$\cut{, $\cF_2$} & $\cF$, $\cF_1$\cut{, $\cF_2$}, resp. & Examples~\ref{exp:sp-dens} and~\ref{alg:prox-grad}\\\hline
6 & $\cF$, $\cF_2$ \cut{, $\cF_1$, $\cF_2$} & $\cC_1$ & $\cF$, $\cF_2$\cut{, $\cF_1$, $\cF_2$}, resp.& Example \ref{exp:log-grad}\\\hline
7 & $\cF_1$ & $\cF_2$& $\cF$ &Example \ref{den-den}\\\hline
8 & cheap & $\cF_2$ & $\cF$ & Example \ref{alg:prox-grad}\\\hline
9 & $\cF_1$ & cheap& $\cF_1$ &Examples~\ref{exp:log-grad} and~\ref{alg:prox-grad}\\\hline
\end{tabular}\end{center}
\caption{Summary of how $\cT_1\circ\cT_2$ inherits CF properties from those of $\cT_1$ and $\cT_2$.}\label{table:comp-op}\end{table}

Next, we discuss how $\cT_1\circ \cT_2$ inherits the CF properties from $\cT_1$ and $\cT_2$. For simplicity, we only use matrix-vector multiplication as examples to illustrate the ideas; more interesting examples will be given later.
\begin{itemize}

\item If $\cT_1$ is separable or nearly-separable ($\cC_1\cup\cC_2$), then as long as $\cT_2$ is CF ($\cF$), $\cT_1\circ\cT_2$ remains CF. In addition, if $\cT_2$ is Type-I CF ($\cF_1$), so is $\cT_1\circ\cT_2$.
\begin{example}\label{exp:sp-dens}
Let $A\in\RR^{m\times m}$ be \emph{sparse} and $B\in\RR^{m\times m}$ \emph{dense}. Then $\cT_1 x=Ax$ is nearly-separable and $\cT_2 x=Bx$ is Type-I CF\footnote{For this example, one can of course pre-compute $AB$ and claim that $(\cT_1\circ\cT_2)$ is Type-I CF. Our arguments keep $A$ and $B$ separate and only use the nearly-separability of $\cT_1$ and Type-I CF property of $\cT_2$, so our result holds for any such composition even when $\cT_1$ and $\cT_2$ are nonlinear.}.  For any $i$, let $\II_i$ index the set of nonzeros on the $i$th row of $A$. We first compute $(Bx)_{\II_i}$, which costs $O(|\II_i|m)$, and then $a_{i,\II_i} (Bx)_{\II_i}$, which costs $O(|\II_i|)$, where $a_{i,\II_i}$ is formed by the nonzero entries on the $i$th row of $A$. Assume $O(|\II_i|)=O(1),\,\forall i$. We have, from the above discussion, that $\nops{x}{(\cT_1\circ\cT_2 x)_i}=O(m)$,
while $\nops{x}{\cT_1\circ\cT_2 x}=O(m^2)$. Hence, $\cT_1\circ\cT_2$ is Type-I CF.
\end{example}

\item Assume that $\cT_2$ is separable ($\cC_1$). It is easy to see that if $\cT_1$ is CF ($\cF$)\cut{ ($\cF, \cF_1, \cF_2$)}, then $\cT_1\circ \cT_2$ remains CF\cut{ ($\cF, \cF_1, \cF_2$, respectively)}. In addition if $\cT_1$ is Type-II CF ($\cF_2$), so is $\cT_1\circ\cT_2$; see Example \ref{exp:log-grad}.\cut{ and \ref{alg:prox-grad}.} 

Note that, if $\cT_2$ is nearly-separable, we do not always have CF properties for $\cT_1\circ\cT_2$. This is because $\cT_2 x$ and $\cT_2 x^+$ can be totally different (so updating $\cT_2 x$ is expensive) even if $x$ and $x^+$ only differ  over one coordinate; see the  footnote~\ref{note1} on Page \pageref{note1}. 

\item Assume that $\cT_1$ is Type-I CF ($\cF_1$). If $\cT_2$ is Type-II CF ($\cF_2$), then $\cT_1\circ\cT_2$ is CF ($\cF$).
\begin{example}\label{den-den}
Let $A,B\in\RR^{m\times m}$ be dense. Then $\cT_1 x=Ax$ is Type-I CF and $\cT_2 x=Bx$ Type-II CF (by maintaining $Bx$; see {Example \ref{ex:lsq2}}). For any $x$ and $i$, let $x^+$ satisfy~\eqref{singleupdate}. Maintaining $\cT_2 x$, we can compute $(\cT_1\circ\cT_2 x)_j$ for $O(m)$ operations for any $j$ and update $\cT_2 x^+$ for $O(m)$ operations. On the other hand, computing $\cT_1\circ\cT_2 x^+$ without maintaining $\cT_2 x$ takes $O(m^2)$ operations.
\end{example}


\item Assume that one of $\cT_1$ and $\cT_2$ is cheap\cut{\footnote{By ``cheap'', we mean their computational complexity differ at least one order. For example, if computing $\cT_1 x$ costs $O(m)$ and $\cT_2 x$ costs $O(m^2)$, then $\cT_1$ is cheap.} compared to the other one}. If $\cT_2$ is cheap, then as long as $\cT_1$ is Type-I CF ($\cF_1$), $\cT_1\circ \cT_2$ is Type-I CF. If $\cT_1$ is cheap, then as long as $\cT_2$ is Type-II CF ($\cF_2$), $\cT_1\circ \cT_2$ is CF ($\cF$); see Example~\ref{alg:prox-grad}.
\end{itemize}

We will see more examples of the above cases in the rest of the paper.
\cut{
\rev{\subsection{Demonstration with Logistic Regression}
In this subsection, we compare the efficiency of three coordinate update schemes (cyclic, cyclic permutation, and random) with the full gradient descent method for solving the regularized logistic regression problem
\begin{equation}\label{eqn:l2_log}
\Min_{x} \frac{\lambda}{2} \|x\|_2^2 + \sum_{i = 1}^p \log\left(1 + \exp(- b_i \cdot a_i^\top x)\right).
\end{equation}
The goal of this experiment is to show the efficiency of the coordinate update methods for composition operators. We solve~\eqref{eqn:l2_log} with the following gradient descent method
\begin{equation}\label{eqn:gd_for_log_loss}
x^{k+1} = x^k - \eta_k \left(\lambda  x^k + \sum_{i = 1}^p \frac{-b_i}{1 + \exp{(b_i \cdot a_i^\top  x^k)}} a_i\right),
\end{equation}
where $\eta_k$ is the step size. The update~\eqref{eqn:gd_for_log_loss} can be treated as a combination of the four operators ($\cI, A, A^\top, \cT$), i.e.,
\begin{equation}\label{eqn:log-reg-update}
x^{k + 1} = (1 - \lambda \eta_k) \cI x^k + \eta_k A^\top \circ \cT \circ A x^k,
\end{equation}
where $\cT(y) = (\frac{b_1}{1 + \exp(b_1 \cdot y_1)}, ..., \frac{b_p}{1 + \exp(b_p \cdot y_p)})^\top$ and $A = (a_1,  ..., a_p)^\top \in \RR^{p \times m}$. As explained in Example \ref{exp:log-grad}, the update~\eqref{eqn:log-reg-update} is CF if $Ax^k$ is maintained. It is worth mentioning that greedy coordinate update with Gauss-Southwell rule is not an efficient choice for solving~\eqref{eqn:l2_log}, since the complexity of computing the scores is $O(mp)$ even though $Ax^k$ is maintained. We test~\eqref{eqn:log-reg-update} with $A$ and $x$ generated with standard normal distribution, and $b = \sign (Ax)$. We set $m = p = 100$, and set $\lambda = 0$ and $\lambda = 1$ for the first and the second experiment respectively. When $\lambda = 0$, the objective is convex, but not strongly convex, so Figure \ref{fig:log_a} shows that all of the methods converge with sublinear rate. When $\lambda = 1$, the objective function is strongly convex, all methods converge with linear rate as shown in Figure \ref{fig:log_b}. In both scenarios, coordinate update methods converge faster than the full gradient descent method.
\begin{figure}[!htbp] \centering
    \begin{subfigure}[b]{0.48\linewidth}
        \includegraphics[width=50mm]{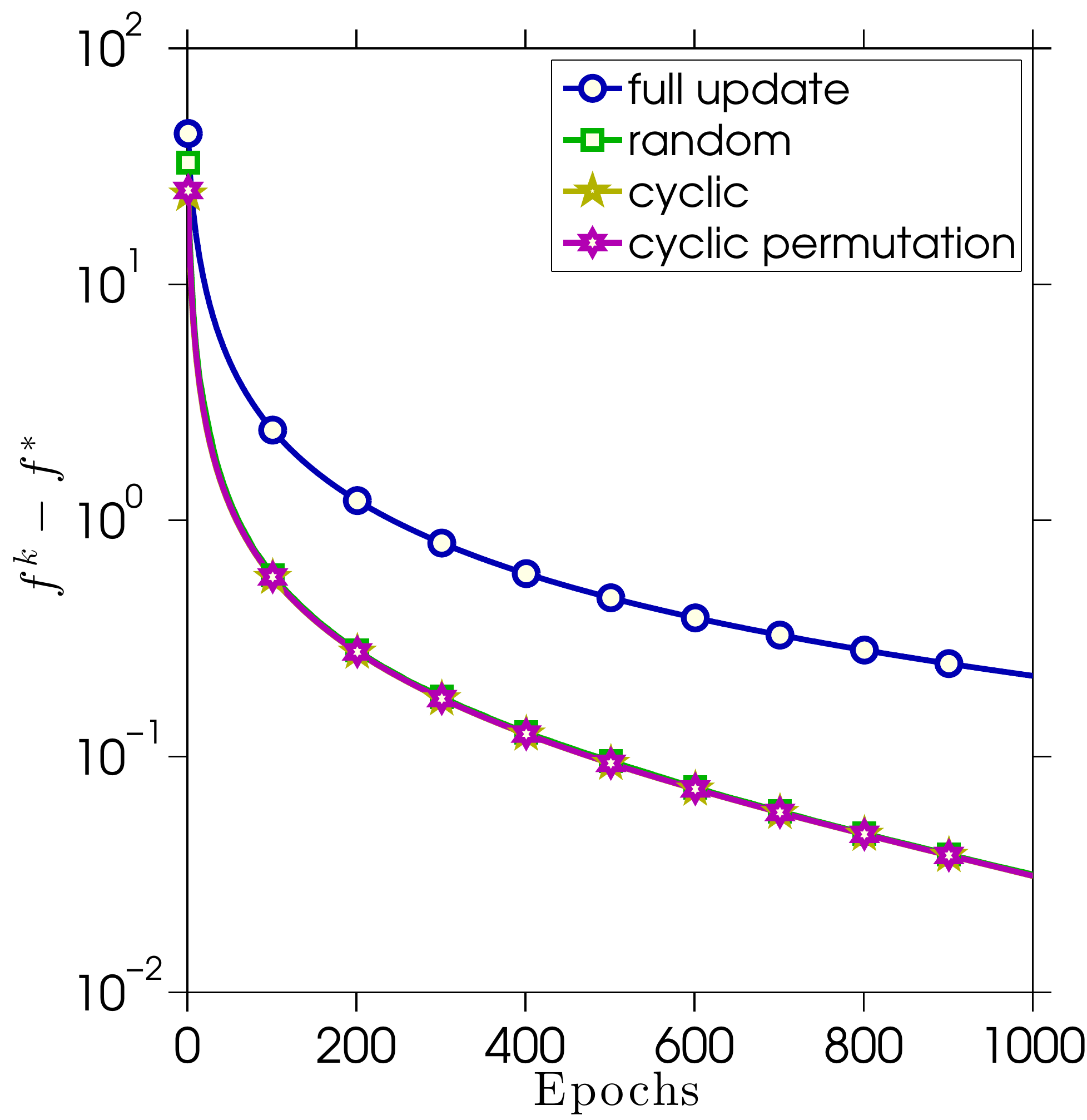}
        \caption{$\lambda = 0$}
        \label{fig:log_a}
    \end{subfigure} %
    \quad
    \begin{subfigure}[b]{0.48\linewidth}
        \includegraphics[width=50mm]{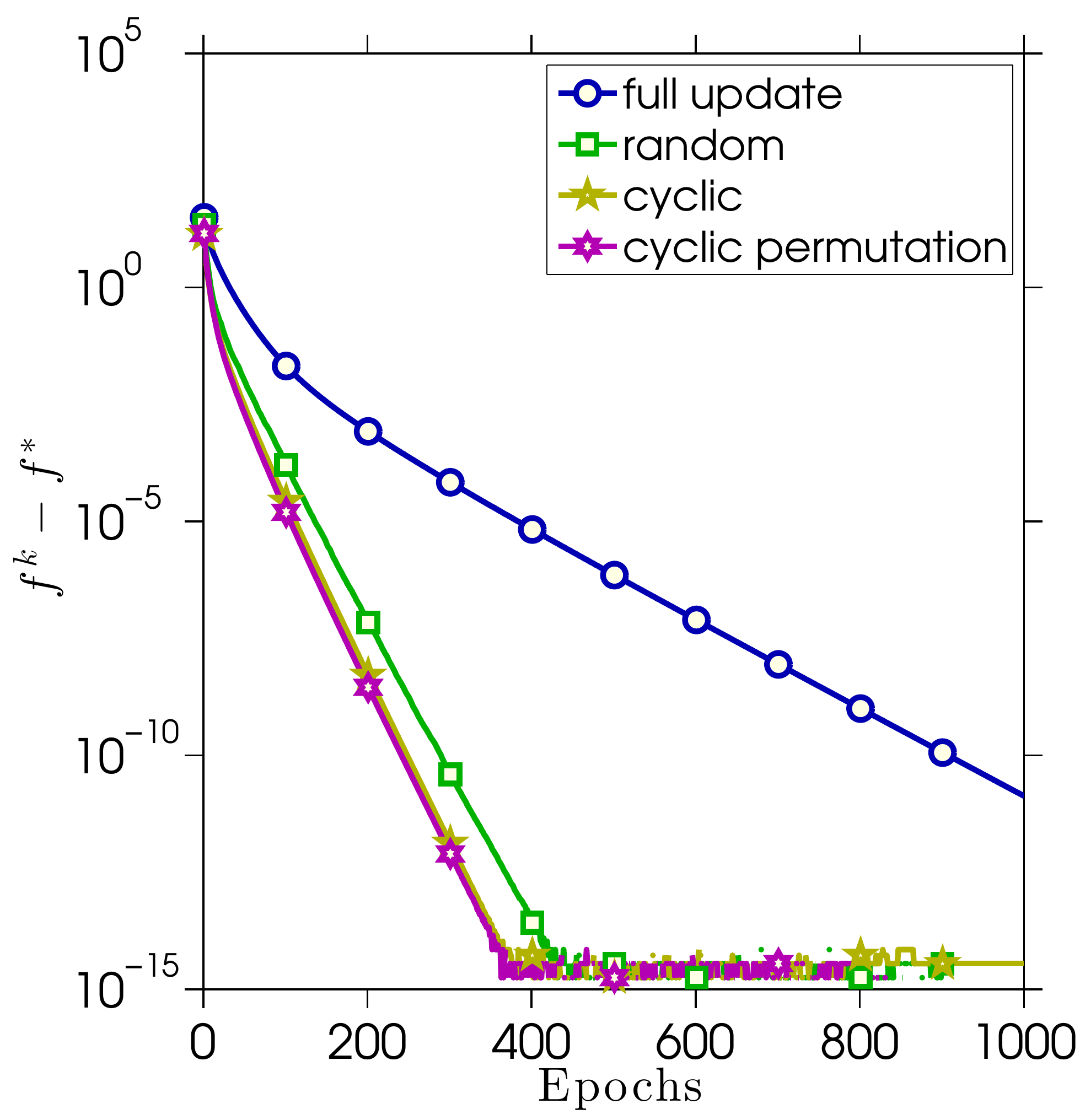}
        \caption{$\lambda = 1$}
        \label{fig:log_b}
    \end{subfigure} %
    \caption{Compare the convergence of three different coordinate update algorithms with full gradient descent algorithm for solving logistic regression.}
    \label{fig:l2_log_results}
\end{figure}
}

\cut{If $\cT_1\in \cC_1$ and $\cT_2\in \cE$ (is easy to update), then  $\cT_1\circ \cT_2\in \cF$ is amenable for coordinate update. Specifically, we shall cache $\cT_2 x$, and when one or a few coordinates of $x$ change, we can compute $((\cT_1\circ \cT_2) x)_i$ by first refreshing $\cT_2 x$ at a low cost and then computing $((\cT_1\circ \cT_2) \tilde x)_i=(\cT_1(\cT_2x))_i$.}
}
\cut{
\begin{example}[linear composing easy-to-update] SOCP
\end{example}

\begin{example}[$\cC_1\cup \cC_2$ composing easy-to-update]
 Prox-linear / forward-backward
\end{example}
}

\cut{
In this section, we consider operators which can be written as a composition of several simpler operators. We will discuss when the composite operators will be BC-friendly.

Assume that operator $\cT$ is a composition of two operators, i.e.,
\begin{equation}
\cT = \cT_1 \circ \cT_2.
\end{equation}
We consider the following three main cases, where \uline{BC-friendly of individual operators induces BC-friendly} of composite operator.

\begin{itemize}
\item If $\cT_1$ is with property $\cC_1$, $\cT_1\circ \cT_2$ has the same property as $\cT_2$, i.e., if $\cT_2$ is with property $\cC_1$ (or $\cC_2$, $\cC_3$, $\cB_1$, $\cB_2$, $\cB_3$), then $\cT_1\circ\cT_2$ is with property $\cC_1$ (or $\cC_2$, $\cC_3$, $\cB_1$, $\cB_2$, $\cB_3$, respectively),
\item If $\cT_1$ is with property $\cC_2$, and $\cT_2$ is with property $\cC_1$, then $\cT_1\circ\cT_2$ is with property $\cC_2$.
\item If $\cT_1$ is with property $\cC_2$, and $\cT_2$ is with property $\cC_2$, then $\cT_1\circ\cT_2$ is with property $\cC_2$.
\item If $\cT_1$ is with property $\cC_2$, and $\cT_2$ is with property $\cC_3$, then $\cT_1\circ\cT_2$ is with property $\cC_3$.
\item If $\cT_1$ is with property $\cC_2$, and $\cT_2$ is with property $\cB_1$, then $\cT_1\circ\cT_2$ is with property $\cB_1$?
\item $\cdots$
\end{itemize}

Case 1. If $\cT_1=\cT_{1,1}\times\cdots\times \cT_{1,m}$ is block-separable, then
$$ \cS_i \circ (\cT_1 \circ \cT_2) = \cT_{1,i}\circ (\cS_i \circ \cT_2).$$
If $\cT_2$ is BC-friendly, then the composite map $\cT$ is also BC-friendly.

Case 2.  If $\cT_1$ is  sparse supported, then $\exists\, \II_{i}$ such that $$(\cS_i\circ \cT_1)x = (\cT_1x)_i = \cT_{1,i}(\{x_j\}_{j\in \II_i}),$$
where $|I_i|$ is small, and thus
$$ \cS_i \circ (\cT_1 \circ \cT_2) = \cT_{1,i}(\{\cS_j \circ \cT_2\}_{just}).$$
In this case, if $\cT_2$ is BC friendly, then the composite map $\cT$ is also BC-friendly.

Case 3.  If $\cT_1$ is BC-friendly and $\cT_2 \, x$ is  easy-to-compute/update, then $\cS_i \circ (\cT_1 \circ \cT_2)(x) = (\cS_i \circ \cT_1)(\cT_2 \, x )$ and $\cT$ is BC-friendly.

}

\subsection{Operator Splitting Schemes}\label{sec:splitting}
We will apply our discussions above to operator splitting and  obtain new algorithms. But first, we review several major operator splitting schemes and discuss their CF properties. We will encounter important concepts such as \emph{(maximum) monotonicity} and \emph{cocoercivity}, which are given in Appendix~\ref{sec:op-concept}. For a monotone  operator $\cA$, the \emph{resolvent operator} $\cJ_{\cA}$ and the \emph{reflective-resolvent operator}  $\cR_{\cA}$ are also defined there, in~\eqref{def-resolvent} and~\eqref{def-ref}, respectively. \cut{The coordinate update methods we propose, based on operator splitting schemes, all have convergence guarantee, at least for async-parallel (thus also stochastic) update \cite{Peng_2015_AROCK}. However, naively extending existing convergent algorithms to coordinate update schemes may result in divergence or wrong solutions; see also Appendix  \ref{sec:op-concept} for a counterexample.}  

Consider the following problem: given three operators $\cA,\cB,\cC$, possibly set-valued,  \begin{equation}\label{eqn:3s_problem}
\text{find } x \in \HH \qquad \text{ such that }  \qquad 0 \in \cA x + \cB x +\cC x,
\end{equation}
where ``$+$" is the Minkowski sum.
This is a high-level abstraction of many problems or their optimality conditions. The study began in the 1960s, followed by a large number of algorithms and applications over the last fifty years. Next, we review a few basic methods for solving \eqref{eqn:3s_problem}.

When $\cA, \cB$ are maximally monotone (think it as the subdifferential $\partial f$ of a proper convex function $f$) and $\cC$ is $\beta$-cocoercive (think it as the gradient $\nabla f$ of a $1/\beta$-Lipschitz differentiable function $f$),  a solution can be found by the iteration \eqref{fpi} with $\cT=\TS$, introduced recently in \cite{davis2015three}, where  \cut{three operator splitting (3S) for solving \eqref{eqn:3s_problem} is defined by}
\beq\label{3s}
\TS := \cI- \cJ_{\gamma \cB}+ \cJ_{\gamma \cA}\circ(2 \cJ_{\gamma \cB}- \cI - \gamma \cC\circ \cJ_{\gamma \cB}).
\eeq \cut{It can be shown that if $\cC$ is $\beta$-cocoercive, then}Indeed, by setting  $\gamma\in(0,2\beta)$, $\cT_{3S}$ is $(\frac{2\beta}{4\beta-\gamma})$-averaged (think it as a property weaker than the Picard contraction; in particular,  $\cT$ may not have a fixed point). Following the standard convergence result (cf. textbook \cite{bauschke2011convex}), provided that $\cT$ has a  fixed point, the sequence from~\eqref{fpi} converges to a fixed-point $x^*$ of $\cT$. Note that, instead of $x^*$, $\cJ_{\gamma\cB}(x^*)$ is a solution to~\eqref{eqn:3s_problem}.  \cut{, and one can choose $\gamma\in(0,2\beta)$ for convergence.}

Following \S\ref{sc:comb},  $\TS$ {is CF if } $\cJ_{\gamma \cA}$ is separable ($\cC_1$), $\cJ_{\gamma \cB}$ is \cut{easy-to-compute or }Type-II CF ($\cF_2$), and $\cC$ is Type-I CF ($\cF_1$). 

We give a few special cases of $\TS$ below, which have much longer history. They all converge to a fixed point $x^*$ whenever a solution exists and $\gamma$ is properly chosen. If $\cB\neq 0$, then {$\cJ_{\gamma \cB}(x^*)$, instead of $x^*$, is a solution to \eqref{eqn:3s_problem}.

 \textbf{Forward-Backward Splitting (FBS):} Letting $\cB=0$ yields $\cJ_{\gamma\cB}=\cI$. Then, $\TS$ reduces to FBS \cite{passty1979FBS}:
 \begin{equation}\label{eq:FBS}
 \TFBS:=\cJ_{\gamma \cA}\circ(\cI-\gamma \cC)
 \end{equation}
 for solving the problem $0\in\cA x+\cC x$.

\textbf{Backward-Forward Splitting (BFS):} Letting $\cA=0$ yields $\cJ_{\gamma\cA}=\cI$. Then, $\TS$ reduces to BFS:
  \begin{equation}\label{eq:BFS}\TBFS:=(\cI-\gamma \cC)\circ \cJ_{\gamma \cB}
  \end{equation}
for solving the problem $0\in\cB x+\cC x$. When $\cA=\cB$, $\TFBS$ and $\TBFS$ apply the same pair of operators in the opposite orders, and they solve the same problem. Iterations based on $\TBFS$ are rarely used in the literature because they  need an extra application of $\cJ_{\gamma B}$ to return the solution, so $\TBFS$ is seemingly an unnecessary variant of $\TFBS$. However, they become  different for coordinate update; in particular, $\TBFS$ is CF (but $\TFBS$ is generally not) when $\cJ_{\gamma \cB}$ is \cut{easy-to-compute or }Type-II CF ($\cF_2$) and $\cC$ is Type-I CF ($\cF_1$). Therefore, $\TBFS$ is worth discussing alone.

\textbf{Douglas-Rachford Splitting (DRS):} Letting $\cC=0$, $\TS$ reduces to
  \begin{equation}\label{eq:DRS}\TDRS:=\cI-\cJ_{\gamma \cB}+\cJ_{\gamma\cA}\circ(2\cJ_{\gamma \cB}-\cI)=\frac{1}{2}(\cI+\cR_{\gamma\cA}\circ \cR_{\gamma\cB})
  \end{equation}
introduced in \cite{douglas1956DRS} for solving the problem $0\in\cA x+\cB x$. A more general splitting is the  Relaxed Peaceman-Rachford Splitting (RPRS) with $\lambda\in[0,1]$:
 \begin{equation}
\cT_{\text{RPRS}} = (1 - \lambda)\, \cI + \lambda \, \cR_{\gamma \cA} \circ \cR_{\gamma \cB},
\end{equation}
which recovers $\TDRS$ by setting $\lambda=\frac{1}{2}$ and Peaceman-Rachford Splitting (PRS) \cite{peaceman1955PRS} by letting $\lambda=1$.

\textbf{Forward-Douglas-Rachford Splitting (FDRS):} Let $V$ be  a linear subspace, and $\cN_V$ and $\cP_V$ be its normal cone and projection operator, respectively. The FDRS~\cite{briceno2015FDRS}
 $$\TFDRS=\cI-\cP_V+\cJ_{\gamma\cA}\circ(2\cP_V-\cI-\gamma \cP_V\circ\tilde{\cC}\circ\cP_V),$$
aims at finding a point $x$ such that $0\in\cA\,x+\tilde{\cC}\,x+\cN_V\,x$. If an optimal $x$ exists, we have $x\in V$ and $\cN_Vx$ is the orthogonal complement of $V$. Therefore, the problem is equivalent to finding $x$ such that $0\in\cA\,x+\cP_V\circ\tilde{\cC}\circ\cP_V\,x+\cN_V\,x$. Thus, $\TS$ recovers $\TFDRS$ by letting $\cB=\cN_V$ and $\cC=\cP_V\circ\tilde{\cC}\circ\cP_V$.

\textbf{Forward-Backward-Forward Splitting (FBFS):} Composing $\TFBS$ with one more forward step gives $\TFBFS$ introduced in \cite{FBF_Tseng}:
\begin{align}\label{eqn:fbf}
\TFBFS = -\gamma \cC  + (\cI-\gamma \cC)\cJ_{\gamma \cA}(\cI-\gamma \cC).
\end{align} 
$\TFBFS$ is not  a special case of $\TS$. At the expense of one more application of $(\cI-\gamma \cC)$, $\TFBFS$ relaxes the convergence condition  of $\TFBS$ from  the cocoercivity of $\cC$ {to its monotonicity. (For example, a nonzero skew symmetric matrix is monotonic but not cocoercive.)}
From Table \ref{table:comp-op}, we know that $\TFBFS$ is CF if both $\cC$ and $\cJ_{\gamma \cA}$ are separable.

\subsubsection{Examples in Optimization}
Consider the optimization problem
\begin{equation}\label{eq:opt-example}
\Min_{x\in X}\, f(x)+g(x),
\end{equation}
where $X$ is the feasible set and $f$ and $g$ are objective functions. We present examples of operator splitting methods discussed above.

\cut{We discuss a few well-known optimization methods for solving problems in the form of \eqref{eq:opt-example} and relate them to the previous operator splitting methods.}

\begin{example}[{proximal gradient method}]\label{alg:prox-grad} Let $X=\RR^m$, $f$ be differentiable, and $g$ be proximable in \eqref{eq:opt-example}. Setting $\cA=\partial g$ and $\cC=\nabla f$ in \eqref{eq:FBS} gives $\cJ_{\gamma \cA}=\prox_{\gamma g}$ and reduces $x^{k+1}=\TFBS(x^k)$ to prox-gradient iteration: \cut{. one can apply the proximal gradient method with update}
\begin{equation}\label{eq:prox-grad}x^{k+1}=\prox_{\gamma g}(x^k-\gamma\nabla f(x^k)).
\end{equation}
A special case of \eqref{eq:prox-grad} with $g=\iota_X$ is the projected gradient iteration:
\begin{equation}\label{eq:proj-grad}
x^{k+1}=\cP_X(x^k-\gamma \nabla f(x^k)).
\end{equation}

If $\nabla f$ is CF and $\prox_{\gamma g}$ is (nearly-)separable (e.g., $g(x)=\|x\|_1$ or the indicator function of a box constraint) or if $\nabla f$ is Type-II CF and $\prox_{\gamma g}$ is cheap (e.g., $\nabla f(x)=Ax-b$ and $g=\|x\|_2$), then the FBS iteration~\eqref{eq:prox-grad} is CF. In the latter case, we can also apply the BFS iteration~\eqref{eq:BFS} (i.e,  compute $\prox_ {\gamma g}$ and then perform the gradient update), which is also CF.

\end{example}


\DIFdelbegin 
\DIFdelend \DIFaddbegin \begin{example}[ADMM]\DIFaddend \label{alg:admm} Setting $X=\RR^m$ simplifies \eqref{eq:opt-example} to
\begin{equation}\label{eq:compx-y}
\Min_{x,y}~f(x)+g(y),\quad\St~ x-y=0.
\end{equation}
The ADMM method iterates:
\begin{subequations}\label{eq:admmx-y}
\begin{align}
&x^{k+1}=\prox_{\gamma f}(y^k-\gamma s^k),\\
&y^{k+1}=\prox_{\gamma g}(x^{k+1}+\gamma s^{k}),\\
&s^{k+1}=s^k+\frac{1}{\gamma}(x^{k+1}-y^{k+1}).
\end{align}
\end{subequations}
(The iteration can be generalized to handle the constraint $Ax-By=b$.) The dual problem of \eqref{eq:compx-y} is $\min_s f^*(-s)+g^*(s)$, where $f^*$ is the convex conjugate of $f$. Letting $\cA=-\partial f^*(-\cdot)$ and $\cB=\partial g^*$ in \eqref{eq:DRS} recovers the iteration  \eqref{eq:admmx-y} through (see the derivation in Appendix \ref{sec:drs-admm})
\begin{align*}
&t^{k+1}=\TDRS(t^k)=t^k-\cJ_{\gamma \cB}(t^k)+\cJ_{\gamma\cA}\circ(2\cJ_{\gamma \cB}-\cI)(t^k).
\end{align*}  
From the results in \S\ref{sc:comb}, a sufficient condition for the above iteration to be CF is that $\cJ_{\gamma\cA}$ is (nearly-)separable and $\cJ_{\gamma\cB}$ being CF.
\end{example}

The above abstract operators and their CF properties will be applied in~\S\ref{sec:applications} to give interesting algorithms for several applications.

%

\cut{
\begin{subsection}{Relaxed Peaceman-Rachford splitting (RPRS)}
The relaxed Peaceman-Rachford splitting (RPRS) also targeted at solving \eqref{eqn:monotone_problem}. The RPRS is define in the following
\begin{equation}
\cT_{\text{RPRS}} = (1 - \lambda)\, \cI + \lambda \, \cR_{\gamma \cA} \circ \cR_{\gamma \cB},
\end{equation}
where $\cR_{\gamma \cA}:= 2 J_{\gamma \cA} - I$ is the reflection operator. Note that $\lambda = 1$ gives the Peaceman-Rachford splitting (PRS) operator, and $\lambda = \frac{1}{2}$ gives the Douglas-Rachford splitting (DRS) operator.

Applying the results of composite maps, the RPRS splitting operator is CF in either of the following two cases:

Case 1. $ \cR_{\gamma A}$ is block-separable  and $\cR_{\gamma B}$ is CF.

Case 2. $ \cR_{\gamma A}$ is Type-I CF and $\cR_{\gamma B}$ is easy-to-update.
\end{subsection}

\subsection{Forward-Backward-Forward operator splitting}
The Forward-Backward-Forward operator splitting (FBF)~\cite{FBF_Tseng} for solving the problem of finding the zero in the sum of two operators, i.e.,
\begin{equation}\label{eqn:2_problem}
\text{find } x \in \HH \qquad \text{ such that }  \qquad 0 \in \cA\,x + \cC\,x,
\end{equation}
is defined in the following
\begin{align}\label{eqn:fbf}
\cT_{\text{FBF}} = -\gamma \cC  + (I-\gamma \cC)\cJ_{\gamma \cA}(I-\gamma \cC).
\end{align}
The advantage of FBF comparing to FBS is that the cocoercivity condition of $\cC$ is relaxed into monotonicity at the expense of additional computations.
Then, we know that $\cT_{\text{FBF}}$ is CF, if both $\cC$ and $\cJ_{\gamma \cA}$ are separable.
}

%


\section{Primal-dual Coordinate Friendly Operators}\label{sec:p-d}
{We study how to solve the  problem
\begin{equation}
\underset{x\in\mathbb{H}}{\text{\normalfont minimize }} f(x)+g(x)+h(Ax),\label{pdproblem}
\end{equation}
{with primal-dual splitting algorithms, as well as their coordinate update versions. Here, $f$ is differentiable 
and $A$ is a ``$p$-by-$m$" linear operator from $\mathbb{H}=\HH_1\times\cdots\times\HH_m$ to $\mathbb{G}=\GG_1\times\cdots\times\GG_p$.
\cut{[do we mention saddle-point and VI at the end of this section?]} Problem~\eqref{pdproblem} abstracts many applications in image processing and machine learning.
\begin{example}[image deblurring/denoising]
Let $u^0$ be an image, where $u_i^0\in[0,255]$, and $B$ be the blurring linear operator. Let $\|\nabla u\|_1$ be the anisotropic\footnote{Generalization to the isotropic case is straightforward by grouping variables properly.} total variation of $u$ (see \eqref{eqn:tv_def} for definition). Suppose that $b$ is a noisy observation of $Bu^0$. Then, we can try to recover $u^0$ by solving
\begin{equation}
\Min_u\, \frac{1}{2}\|Bu-b\|^2+\iota_{[0,255]}(u)+\lambda\|\nabla u\|_1,
\end{equation}
which can be written in the form of $\eqref{pdproblem}$ with $f=\frac{1}{2}\|B\cdot-b\|^2$, $g=\iota_{[0,255]}$, $A=\nabla$, and $h=\lambda\|\cdot\|_{1}$.
\end{example}
More examples with the formulation \eqref{pdproblem} will be given in \S\ref{sec:emp}. 
In general, primal-dual methods are capable of solving complicated problems involving constraints and the compositions of proximable and linear maps like $\|\nabla u\|_1$.

In many applications, although $h$ is  proximable, $h\circ A$  is generally non-proximable and non-differentiable. To avoid using slow subgradient methods,  we can consider the  primal-dual splitting approaches to separate $h$ and $A$ so that $\prox_{h}$ can be applied.
\cut{
The first approach is \emph{variable splitting}: first rewrite the problem~\eqref{pdproblem} as $$\Min_{x,y}\, f(x)+g(x)+h(y),\quad \St~ y=Ax,$$ and then apply ADMM~\eqref{eq:admmx-y}. The $y$-subproblem of ADMM reduces to computing $\prox_{h}$. The $x$-subproblem of ADMM has the form
\beq\label{admmx}
\Min_x f(x)+g(x)+\frac{\eta}{2}\|Ax-y\|^2+(\mbox{linear terms in }x).
\eeq
When $f+g$ is differentiable or proximable,~\eqref{admmx} can be solved by an iterative procedure. In image deblurring, with $g=0$ and proper boundary conditions, even a closed-form solution can be found~\cite{wang2008new}. Generally, it is not easy to directly solve~\eqref{admmx}. By introducing more auxiliary variables, $A$ can also be separated from $g$ and $h$, but the resulting subproblem involving $A$ will need form and invert  $A^\top A$~\cite{o2014primal}. (A remedy is to linearize $\|Ax-y\|^2$ in the subproblem, yet it can be shown as a special case of the primal-dual splitting approach below.)

The second approach is \emph{primal-dual splitting}. 
}
We derive that the equivalent form (for convex cases) of $\eqref{pdproblem}$ is to find $x$ such that
\begin{equation}
0\in (\nabla f+\partial g+A^\top\circ\partial h\circ A)(x).
\end{equation}
Introducing the dual variable $s\in\mathbb{G}$ and applying the biconjugation property:  $s\in \partial h(Ax)\Leftrightarrow Ax\in \partial h^*(s)$, yields the equivalent condition
\begin{equation}
0\in\bigg(\underbrace{\begin{bmatrix}
\nabla f & 0\\
0 & 0
\end{bmatrix}}_{\mbox{operator}~\cA}+\underbrace{
\begin{bmatrix}
\partial g  &  0\\
0 & \partial h^*
\end{bmatrix}+\begin{bmatrix}
0&A^\top\\
-A&0
\end{bmatrix}}_{\mbox{operator}~\cB}\bigg) \underbrace{\begin{bmatrix}
x\\
s
\end{bmatrix}}_{z},\label{pdkkt}
\end{equation}
which we shorten as $0\in\cA z+\cB z$, with $z\in\HH\times\GG=:\FF$.

\cut{\begin{remark}
In a variant of problem $\eqref{pdproblem}$\cite{}, $h(Ax)$ is replaced by $(h\Box l)(Ax)$ with $l$ a strongly convex function (thus $l^*$ is differentiable). $\Box$ is the infimal convolution operator defined as $(h\Box l)(y)=\inf_{z\in\GG} h(z)+l(y-z)$, which is used, for example, in dual smoothing\cite{?}. Because $s\in (\partial h\Box\partial l)(Ax)\Leftrightarrow Ax\in \partial h^*(s)+\partial l^*(s)$, $\eqref{pdkkt}$ becomes: $0\in\begin{bmatrix}
\nabla f(x)\\
0
\end{bmatrix}+
\begin{bmatrix}
\partial g(x)\\
\partial h^*(s)
\end{bmatrix}+\begin{bmatrix}
0&A^\top\\
-A&0
\end{bmatrix}\begin{bmatrix}
x\\
s
\end{bmatrix}$, which can be solved in the same way as $\eqref{pdproblem}$.
\end{remark}
}
Problem $\eqref{pdkkt}$ can be solved  by the Condat-V\~{u} algorithm~\cite{condat2013primal, vu2013splitting}:
\begin{equation}
\left\{
\begin{array}{l}
s^{k+1}=\prox_{\gamma h^*} (s^k+\gamma Ax^k),\\
x^{k+1}=\prox_{\eta g}(x^k-\eta(\nabla f(x^k)+A^\top(2s^{k+1}-s^k))),
\end{array}
\right.\label{vucondat}
\end{equation}
which explicitly applies $A$ and $A^\top$ and updates $s,x$ in a Gauss-Seidel style~\footnote{By the Moreau identity: $\prox_{\gamma h^*}=\cI -\gamma \prox_{\frac{1}{\gamma}h}(\frac{\cdot}{\gamma})$, one can compute $\prox_{\frac{1}{\gamma} h}$ instead of $\prox_{\gamma h^*}$, which inherits the same separability properties from $\prox_{\frac{1}{\gamma} h}$.}.
We introduce an operator $\TVC:\FF\to \FF$ and write
$$\mbox{iteration~\eqref{vucondat}}\quad\Longleftrightarrow\quad z^{k+1}=\TVC(z^k).$$


Switching the orders of  $x$ and $s$ yields the following algorithm: 
\begin{equation}
\left\{
\begin{array}{l}
x^{k+1}=\prox_{\eta g}(x^k-\eta(\nabla f(x^k)+A^\top s^k)),\\
s^{k+1}=\prox_{\gamma h^*} (s^k+\gamma A(2x^{k+1}-x^k)),
\end{array}
\right.{\text{ as }z^{k+1}=\cT'_{\textnormal{CV}} z^k.}\label{vucondat2}
\end{equation}
It is known from \cite{combettes2014forward,davis2014convergence} that both $\eqref{vucondat}$ and $\eqref{vucondat2}$  reduce to iterations of nonexpansive operators (under a special metric), i.e., $\TVC$ is  nonexpansive; see Appendix~\ref{sec:vc-op} for the reasoning. 
\begin{remark}
Similar primal-dual algorithms can be used to solve other problems such as saddle point problems~\cite{lebedev1967duality,mclinden1974extension,briceno2013monotone} and variational inequalities~\cite{tseng1991applications}. Our coordinate update algorithms below apply to these problems as well.
\end{remark}
\subsection{Primal-dual Coordinate Update Algorithms}\label{sec:pdcu}
In this subsection, we make the following assumption.
\begin{assumption}
Functions $g$ and $h^*$ in the problem~\eqref{pdproblem} are separable and proximable. Specifically, $$g(x)=\displaystyle\sum_{i=1} ^m g_i(x_i)\quad\mbox{and}\quad h^*(y)=\displaystyle\sum_{j=1}^ph^*_i(y_i).$$\label{pdassum}
Furthermore, $\nabla f$ is CF.
\end{assumption}
\begin{proposition}\DIFaddbegin \label{prop1}
\DIFaddend Under Assumption $\ref{pdassum}$, the followings hold:
\begin{enumerate}[(a)]
\item when $p=O(m)$, the Condat-Vu operator $\TVC$ in~\eqref{vucondat} is CF, more specifically, $$\nops{\{z^k,Ax\}}{\{z^+,Ax^+\}}=O\left(\frac{1}{m+p}\nops{z^k}{\TVC z^k}\right);$$
\item when $m\ll p$ and $\nops{x}{\nabla f(x)}=O(m)$, the Condat-Vu operator $\cT'_\textnormal{CV}$ in~\eqref{vucondat2} is CF, more specifically, $$\nops{\{z^k,A^\top s\}}{\{z^+,A^\top s^+\}}=O\left(\frac{1}{m+p}\nops{z^k}{\cT'_{\textnormal{CV}} z^k}\right).$$
\end{enumerate}
\end{proposition}

\cut{
\begin{assumption}
$g(x)=\displaystyle\sum_{i=1} ^m g_i(x_i)$ and $h^*(y)=\displaystyle\sum_{j=1}^ph^*_j(y_j)$ where the $g_i$'s and $h^*_j$'s are proximable functions, i.e. $\prox_g,\prox_{h^*}\in\cC_1$ (can be relaxed to $\cC_2$). $\nabla f$ (and $\nabla l^*$) are easy-to-update.\label{pdassum}
\end{assumption}

We show under this assumption either $\cT_1$ or $\cT_2$ is CF. }
\begin{proof}
Computing $z^{k+1}=\TVC z^k$ involves evaluating $\nabla f$, $\prox_g$, and $\prox_{h^*}$, applying $A$ and $A^\top$, and adding vectors.
It is easy to see  $\nops{z^k}{\TVC z^k}=O(mp+m+p)+\mathfrak{M}[x\to\nabla f(x)]$, and $\nops{z^k}{\cT'_{\textnormal{CV}}z^k}$ is the same.\\
(a) We assume $\nabla f\in\cF_1$ for simplicity, and other cases are similar.
\begin{enumerate}
\item If $(\TVC z^k)_j=s^{k+1}_i$, computing it involves: adding $s^k_i$ and $\gamma (Ax^k)_i$, and evaluating $\prox_{\gamma h^*_i}$. In this case $\nops{\{z^k,Ax\}}{\{z^+,Ax^+\}}=O(1)$.
\item If $(\TVC z^k)_j=x^{k+1}_i$, computing it involves evaluating: the entire $s^{k+1}$ for $O(p)$ operations, $(A^\top(2s^{k+1}-s^k))_i$ for $O(p)$ operations, $\prox_{\eta g_i}$ for $O(1)$ operations, $\nabla_i f({x}^{k})$ for $O(\frac{1}{m}\nops{x}{\nabla f(x)})$ operations, as well as updating $Ax^+$ for $O(p)$ operations.
In this case\\ $\nops{\{z^k,Ax\}}{\{z^+,Ax^+\}}=O(p+\frac{1}{m}\nops{x}{\nabla f(x)})$.
\end{enumerate}
Therefore, $\nops{\{z^k,Ax\}}{\{z^+,Ax^+\}}=O\big(\frac{1}{m+p}\nops{z^k}{\TVC z^k}\big)$. 
\\[5pt]
(b) When $m\ll p$ and $\nops{x}{\nabla f(x)}=O(m)$, following  arguments similar to the above, we have\\$\nops{\{z^k,A^\top s\}}{\{z^+,A^\top s^+\}}=O(1)+\nops{x}{\nabla_i f(x)}$ if $(\cT'_{\textnormal{CV}} z^k)_j=x_i^{k+1}$; and $\nops{\{z^k,A^\top s\}}{\{z^+,A^\top s^+\}}=O(m)+\nops{x}{\nabla f(x)}$ if $(\cT'_{\textnormal{CV}} z^k)_j=s_i^{k+1}$.\\ In both cases $\nops{\{z^k,A^\top s\}}{\{z^+,A^\top s^+\}}=O(\frac{1}{m+p}\nops{z^k}{\cT'_{\textnormal{CV}} z^k})$.
\end{proof}
\subsection{Extended Monotropic Programming}\label{sec:emp}
We develop a primal-dual coordinate update algorithm for the extended monotropic program:
\begin{equation}
\begin{array}{rl}
\underset{x\in\mathbb{H}}{\text{minimize  }} &~g_1(x_1)+g_2(x_2)+\cdots+g_m(x_m)+f(x),\\
\St &~A_1x_1+A_2x_2+\cdots+A_mx_m=b,
\end{array}\label{emp}
\end{equation}
where $x = (x_1, \ldots, x_m) \in \HH=\HH_1\times\ldots \times \HH_m$ with $\HH_i$ being Euclidean spaces. It generalizes linear, quadratic, second-order cone, semi-definite programs by allowing  extended-valued objective functions $g_i$ and $f$.
It is a special case of~\eqref{pdproblem} by letting $g(x)=\displaystyle\sum_{i=1}^m g_i(x_i)$, $A=[A_1,\cdots, A_m]$ and
$h=\iota_{\{b\}}$.
\begin{example}[quadratic programming]
Consider the quadratic program
\begin{equation}
\underset{x\in\mathbb{R}^m}{\textnormal{minimize }} \frac{1}{2}x^\top Ux+c^\top x,~\St~ Ax=b,~x\in X,\label{qp}
\end{equation}
where $U$ is a symmetric positive semidefinite matrix and $X=\{x:x_i \geq 0~\forall i\}$.
Then, \eqref{qp} is a special case of \eqref{emp}\cut{ $$\underset{x\in\mathbb{R}^n}{\textnormal{minimize }} \frac{1}{2}x^\top Ux+c^\top x+\iota_{X}(x)+\iota_{\{b\}}(Ax).$$}
with  $g_i(x_i)=\iota_{\cdot\geq 0}(x_i)$, $f(x)=\frac{1}{2}x^\top Ux+c^\top x$ and $h=\iota_{\{b\}}$.
\end{example}
\begin{example}[Second Order Cone Programming (SOCP)]
The SOCP
\begin{align*}
\underset{x\in\mathbb{R}^m}{\textnormal{minimize }} ~c^\top x,&\quad\St~ Ax=b,\\
&\hspace{56pt} x\in X=Q_1\times \cdots\times Q_{n},
\end{align*}
(where the number of cones $n$ may not be equal to the number of blocks $m$,) can be written in the form of \eqref{emp}: $\Min_{x\in\mathbb{R}^m} \iota_X(x)+c^\top x+\iota_{\{b\}}(Ax).$ 
\end{example}
Applying iteration $\eqref{vucondat}$ to problem $\eqref{emp}$ and eliminating $s^{k+1}$ from the second row yield the Jacobi-style update (denoted as $\cT_\textnormal{emp}$):
\begin{equation}
\left\{
\begin{array}{l}
s^{k+1}=s^k+\gamma (Ax^k-b),\\
x^{k+1}=\prox_{\eta g}(x^k-\eta(\nabla f(x^k)+A^\top s^k+2\gamma A^\top Ax^k-2\gamma A^\top b)).
\end{array}
\right.\label{pdemp}
\end{equation}
To the best of our knowledge, this update is never found in the literature. Note that $x^{k+1}$ no longer depends on $s^{k+1}$, making it more convenient to perform coordinate updates.
%
\begin{remark}
In general, when the $s$ update is affine, we can decouple $s^{k+1}$ and $x^{k+1}$ by plugging the $s$ update into the $x$ update. It is the case when $h$ is affine or quadratic in problem \eqref{pdproblem}.
\end{remark}
A sufficient condition for $\cT_\textnormal{emp}$  to be CF is $\prox_g\in\cC_1$ i.e., separable. Indeed, we have $\cT_\textnormal{emp}=\cT_1\circ\cT_2$, where $$\cT_1=\begin{bmatrix}
\cI & 0\\0&\prox_{\eta g}
\end{bmatrix}, \cT_2\begin{bmatrix}
s\\x
\end{bmatrix}=\begin{bmatrix}
s+\gamma (Ax-b)\\
x-\eta(\nabla f(x)+A^\top s+2\gamma A^\top Ax-2\gamma A^\top b)
\end{bmatrix}.$$
  Following Case 5 of Table $\ref{table:comp-op}$, $\cT_\textnormal{emp}$ is CF. When  $m=\Theta(p)$, the separability condition on $\prox_g$ can be relaxed to $\prox_g\in\cF_1$ since in this case $\cT_2\in\cF_2$, and we can apply  Case 7 of Table $\ref{table:comp-op}$ (by maintaining $\nabla f(x)$, $A^\top s$, $Ax$ and $A^\top Ax$.)
\subsection{Overlapping-Block Coordinate Updates}\label{sec:overlap}
{In the coordinate update scheme based on \eqref{vucondat}, if we select $x_i$ to update then we must first compute $s^{k+1}$, because the variables   $x_i$'s and $s_j$'s are coupled through the matrix $A$. However, once $x_i^{k+1}$ is obtained, $s^{k+1}$  is discarded. It is not used to update $s$ or cached for further use. This subsection introduces ways to utilize the otherwise wasted computation. 

We define, for each $i$,  $\JJ(i)\subset [p]$ as the set of indices $ j$ such that $A^\top_{i,j}\neq 0$, and, for each $j$,  $\II(j)\subset [m]$ as the set of indices of $i$ such that $A^\top_{i,j}\neq 0$. We also let $m_j:=|\II(j)|$, and assume $m_j\neq 0,\forall j \in [p]$ without loss of generality.

We arrange the coordinates of $z=[x;s]$ into $m$ overlapping blocks. The $i$th block consists of the coordinate $x_i$ and  all $s_j$'s for $j\in\JJ(i)$. This way, each $s_j$ may appear in more than one block. We propose a block coordinate update scheme based on~\eqref{vucondat}. Because the blocks overlap, each $s_j$ may be updated in multiple blocks, so the $s_j$ update is relaxed with parameters $\rho_{i,j}\ge 0$ (see~\eqref{pdoverlap} below) that satisfy  $\sum_{i\in \II(j)}\rho_{i,j}=1,~~\forall j \in [p].$ The aggregated effect is to update $s_j$ without scaling. (Following the
KM iteration~\cite{krasnosel1955two}, we can also assign a relaxation parameter $\eta_k$ for the $x_i$ update; then, the $s_j$ update should be relaxed with $\rho_{i,j}\eta_k$.)

We propose the following  update scheme:
\begin{equation}
\left\{
\begin{array}{l}
\text{select }i\in[m], \text{ and then compute}\\
\quad\tilde{s}_j^{k+1}=\prox_{\gamma h_j^*} (s_j^k+\gamma (Ax^k)_j),~\mbox{for all}~ j\in\JJ(i),\\
\quad\tilde{x}_i^{k+1}=\prox_{\eta g_i}(x_i^k-\eta(\nabla_i f(x^k)+\sum_{j\in\JJ(i)}A_{i,j}^\top(2\tilde{s}_j^{k+1}-s_j^k))),\\
\quad\text{update }x_i^{k+1}=x_i^k+(\tilde{x}_i^{k+1}-x_i^k), \\
\quad\text{update }s_j^{k+1}=s_j^k+\rho_{i,j}(\tilde{s}_j^{k+1}-s_j^k),~\mbox{for all}~ j\in\JJ(i).
\end{array}
\right.\label{pdoverlap}
\end{equation}
\begin{remark}
The use of relaxation parameters $\rho_{i,j}$ makes our scheme different from that in \cite{pesquet2014class}. 
\end{remark}
Following the assumptions and arguments in \S$\ref{sec:pdcu}$, if we maintain $Ax$, the cost for each block coordinate update is $O(p)+\nops{x}{\nabla_i f(x)}$, which is $O(\frac{1}{m}\nops{z}{\TVC z})$. Therefore the coordinate update scheme \eqref{pdoverlap} is computationally worthy.

Typical choices of $\rho_{i,j}$ include: (1) one of the $\rho_{i,j}$'s is 1 for each $j$, others all equal to 0. This can be viewed as assigning the update of $s_j$ solely to a block containing $x_i$.
(2) $\rho_{i,j}=\frac{1}{m_j}$ for all $i\in\II(j)$. This approach spreads the update of $s_j$ over all the related blocks.

\begin{remark}
The recent paper \cite{fercoq2015coordinate} proposes a different primal-dual coordinate update algorithm.
The authors produce a new matrix $\bar{A}$ based on $A$, with only one nonzero entry in each row, i.e. $m_j=1$ for each $j$. They also modify $h$ to $\bar{h}$ so that the problem
\begin{equation}
\underset{x\in\mathbb{H}}{\textnormal{minimize }} f(x)+g(x)+\bar{h}(\bar{A}x)\label{fbpdproblem}
\end{equation}
has the same solution as $\eqref{pdproblem}$. Then they solve $\eqref{fbpdproblem}$ by the scheme $\eqref{pdoverlap}$. Because they have $m_j=1$, every dual variable coordinate is only associated with one primal variable coordinate.
They create non-overlapping blocks of $z$ by duplicating each dual variable coordinate $s_j$ multiple times. The computation cost for each block coordinate update of their algorithm is the same as $\eqref{pdoverlap}$, but  more memory is needed for the duplicated copies of each $s_j$.
\end{remark}

\subsection{Async-Parallel Primal-Dual Coordinate Update Algorithms and Their Convergence}

In this subsection, we propose two async-parallel primal-dual coordinate update algorithms using the algorithmic framework of~\cite{Peng_2015_AROCK} and state their convergence results.
When there is only one agent, all algorithms proposed in this section reduce to stochastic coordinate update algorithms~\cite{Patrick_2015}, and their convergence is a direct consequence of Theorem~\ref{thm:async-convergence}. Moreover, our convergence analysis also applies to sync-parallel algorithms.

The two algorithms are based on~\S\ref{sec:pdcu} and~\S\ref{sec:overlap}, respectively.


\begin{algorithm}[H]\label{alg:asyn_core}
\SetKwInOut{Input}{Input}\SetKwInOut{Output}{output}
\Input{$z^0\in\FF$,  $K>0$, a discrete distribution $(q_1,\ldots,q_{m+p})$
with $\sum_{i=1}^{m+p}q_i=1$ and $q_i>0,\forall i$, } set global iteration
counter $k=0$\; \While{$k < K$, every agent asynchronously and continuously}{
  select $i_k\in [m+p]$ with $\mathrm{Prob}(i_k=i)=q_i$\;
  perform an update to $z_{i_k}$ according to  \eqref{eqn:asyn_update}\;
  update the global counter $k \leftarrow k+1$\;
 }
 \caption{Async-parallel primal-dual coordinate update algorithm using $\TVC$}
\end{algorithm}

Whenever an agent updates a coordinate, the global iteration number $k$ increases by one.
The $k$th update is applied to $z_{i_k}$, with $i_k$ being independent random variables: $z_i=x_i$ when $i\leq m$ and $z_i=s_{i-m}$ when $i>m$.  Each coordinate update has
the form:
\begin{equation}
\label{eqn:asyn_update}
\left\{
\begin{array}{l}
z_{i_k}^{k+1} = z_{i_k}^k - \frac{\eta_k}{(m+p)q_{i_k}} \, (\hat{z}^k_{i_k}-(\TVC \hat{z}^k)_{i_k}),\\
z_i^{k+1}=z_i^k,\quad \forall i\neq i_k,
\end{array}
\right.
\end{equation}
where $\eta_k$ is the step size, $z^k$ denotes the state of $z$ in global memory just before the update~\eqref{eqn:asyn_update} is applied, and $\hat{z}^k$ is the result that $z$ in global memory is read by an agent to its local cache (see~\cite[\S 1.2]{Peng_2015_AROCK} for both consistent and inconsistent cases). While $(\hat{z}^k_{i_k}-(\TVC \hat{z}^k)_{i_k})$ is being computed, asynchronous parallel computing allows other agents to make updates to $z$, introducing so-called asynchronous delays. Therefore, $\hat{z}^k$
can be different from $z^k$. We refer the reader to~\cite[\S 1.2]{Peng_2015_AROCK} for more details.

The async-parallel algorithm using the overlapping-block coordinate update~\eqref{pdoverlap} is in Algorithm~\ref{alg:asyn_overlap} (recall that the overlapping-block coordinate update is introduced to save computation).

\begin{algorithm}[H]\label{alg:asyn_overlap}
\SetKwInOut{Input}{Input}\SetKwInOut{Output}{output}
\Input{$z^0\in\FF$,  $K>0$, a discrete distribution $(q_1,\ldots,q_{m})$
with $\sum_{i=1}^{m}q_i=1$ and $q_i>0,\forall i$, } set global iteration
counter $k=0$\; \While{$k < K$, every agent asynchronously and continuously}{
  select $i_k\in [m]$ with $\mathrm{Prob}(i_k=i)=q_i$\;
  Compute $\tilde{s}_j^{k+1},\forall j\in\JJ(i_k)$ and  $\tilde{x}_{i_k}^{k+1}$ according to $\eqref{pdoverlap}$\;
  update $x_{i_k}^{k+1}=x_{i_k}^k+\frac{\eta_k}{mq_{i_k}}(\tilde{x}_{i_k}^{k+1}-x_{i_k}^k)$\;
  let $x_{i}^{k+1}=x_{i}^k$ for $i\neq i_k$\;
  update $s_j^{k+1}=s_j^k+\frac{\rho_{i,j}\eta_k}{mq_{i_k}}(\tilde{s}_j^{k+1}-s_j^k),~\mbox{for all}~ j\in\JJ(i_k)$\;
  let $s_j^{k+1}=s_j^k,~\mbox{for all}~ j\notin\JJ(i_k)$\;
  update the global counter $k \leftarrow k+1$\;
 }
 \caption{Async-parallel primal-dual overlapping-block coordinate update algorithm using $\TVC$}
\end{algorithm}
Here we still allow asynchronous delays, so $\tilde{x}_{i_k}$ and $\tilde{s}_j^{k+1}$ are computed using some $\hat{z}^k$.
\begin{remark}
If shared memory is used, it is recommended to set all but one $\rho_{i,j}$'s to $0$ for each $i$.
\end{remark}

\begin{thm}\label{thm:async-convergence}
Let $Z^*$ be the set of  solutions to problem~\eqref{pdproblem} and $(z^k)_{k\geq0}\subset \FF$ be the sequence generated by Algorithm \ref{alg:asyn_core} or Algorithm \ref{alg:asyn_overlap} under the following conditions:
\begin{enumerate}[(i)]
\item $f,g,h^*$ are closed proper convex functions, $f$ is differentiable, and $\nabla f$ is Lipschitz continuous with constant $\beta$;
\item the delay for every coordinate is bounded by a positive number $\tau$, i.e. for every $1\leq i\leq m+p$, $\hat{z}^{k}_{i}=z_i^{k-d_{i}^k}$ for some $0\leq d_{i}^k\leq\tau$;
\item $\eta_k
\in [\eta_{\min}, \eta_{\max}]$ for certain $\eta_{\min},\eta_{\max}>0$.
\end{enumerate}
Then $(z^k)_{k\geq 0}$ converges to a $Z^*$-valued random variable with probability 1.
\end{thm}
The  formulas for $\eta_{\min}$ and $\eta_{\max}$, as well as the proof of Theorem~\ref{thm:async-convergence},  are given in Appendix \ref{pf:pdasync} along with additional remarks. The algorithms can be applied to solve  problem~\eqref{pdproblem}. A variety of examples are provided in~\S\ref{sec:ML} and~\S\ref{sec:Imaging}.


\section{Applications}\label{sec:applications}
In this section, we provide examples to illustrate how to develop coordinate update algorithms based on CF operators. The applications are categorized into five different areas. The first subsection discusses three well-known machine learning problems: empirical risk minimization, Support Vector Machine (SVM), and group Lasso. The second subsection  discusses image processing problems including image deblurring, image denoising, and  Computed Tomography (CT) image recovery. The remaining subsections provide applications in finance, distributed computing as well as certain stylized optimization models. Several applications are treated with coordinate update algorithms for the first time.

For each problem, we describe the operator $\cT$ and how to efficiently calculate $(\cT x)_i$. The final algorithm is obtained after plugging the update in a coordinate update framework in \S\ref{sec:literature} along with  parameter initialization, an index selection rule, as well as some termination criteria.

\subsection{Machine Learning}\label{sec:ML}
\subsubsection{Empirical Risk Minimization (ERM)}
We consider the following regularized empirical risk minimization problem
\begin{equation}\label{prob:erm}
\Min_{x\in\RR^m}~ \frac{1}{p}\sum_{j=1}^p\phi_j(a_j^\top x)+f(x)+g(x),
\end{equation}
where $a_j$'s are sample vectors, $\phi_j$'s are loss functions, and $f+g$ is a regularization function. We assume that $f$ is differentiable and $g$ is proximable. Examples of~\eqref{prob:erm} include linear SVM, regularized logistic regression, ridge regression, and Lasso. Further information on ERM can be found in \cite{hastie2005elements}. The need for coordinate update algorithms arises in many applications of~\eqref{prob:erm} where the number of samples or the dimension of $x$ is large. 

We define $A=[a_1^\top;a_2^\top;\dots;a_p^\top]$ and $h(y):=\frac{1}{p}\sum_{j=1}^p\phi_j(y_j)$. Hence, $h(Ax)=\frac{1}{p}\sum_{j=1}^p\phi_j(a_j^\top x)$, and problem~\eqref{prob:erm} reduces to form~\eqref{pdproblem}.
We can apply the primal-dual update scheme to solve this problem, for which we introduce the dual variable $s = (s_1, ..., s_p)^\top$. We use $p+1$ coordinates, where the $0$th coordinate is $x\in\RR^m$ and the $j$th coordinate is $s_j$, $j\in [p]$. The operator $\cT$ is given in~\eqref{vucondat2}. At each iteration, a coordinate is updated:
\begin{equation}
{\left\{
\begin{array}{l}
\text{if }x\text{ is chosen (the index 0), then compute}\\
\qquad {x}^{k+1}=\prox_{\eta g}(x^k-\eta(\nabla f(x^k)+A^\top s^k)),\\
\text{if }s_j\text{ is chosen (an index $j \in [p]$), then compute}\\
\qquad\tilde{x}^{k+1}=\prox_{\eta g}(x^k-\eta(\nabla f(x^k)+A^\top s^k)),\\
\qquad\text{and}\\
\qquad{s}_j^{k+1}=\frac{1}{p}\prox_{p\gamma \phi_j^*} (ps_j^k+p\gamma a_j^\top(2\tilde{x}^{k+1}-x^k)).
\end{array}
\right.
}\end{equation}

We maintain $A^\top s\in \RR^m$ in the memory. Depending on the structure of $\nabla f$, we can compute it each time or maintain it. When $\prox_g\in\cF_1$, we can consider breaking $x$ into coordinates $x_i$'s and also select an index $i$ to update $x_i$ at each time.

\subsubsection{Support Vector Machine}\label{sec:svm}
Given the training data $\{(a_i,\beta_i)\}_{i=1}^m$ with $\beta_i\in\{+1, -1\},\,\forall i$, the kernel support vector machine \cite{scholkopf2001learning} is 
\begin{equation}\label{eq:ksvm}
\begin{array}{rl}
\Min\limits_{x,\xi,y} &\frac{1}{2}\|x\|_2^2+C\sum_{i=1}^m \xi_i,\\
 \St\ & \beta_i(x^\top\phi(a_i)-y)\ge 1-\xi_i,\, \xi_i\ge 0,\, \forall i \in [m],
\end{array}
\end{equation}
where $\phi$ is a vector-to-vector map, mapping each data $a_i$ to a point in a (possibly) higher-dimensional space. If $\phi(a)=a$, then \eqref{eq:ksvm} reduces to the linear support vector machine. The model~\eqref{eq:ksvm}  can be interpreted as finding a hyperplane $\{w:x^\top w-y=0\}$ to separate two sets of points $\{\phi(a_i):\beta_i=1\}$ and $\{\phi(a_i):\beta_i=-1\}$. 

The dual problem of \eqref{eq:ksvm} is
\begin{equation}\label{eq:dksvm}
\Min_s \frac{1}{2}s^\top Q s- e^\top s,\ \St\ 0\le s_i\le C,\,\forall i,\, \sum_i \beta_i s_i=0,
\end{equation}
where $Q_{ij}=\beta_i\beta_j k(a_i,a_j)$, $k(\cdot,\cdot)$ is a so-called \emph{kernel function}, and $e = (1, ..., 1)^\top$. If $\phi(a)=a$, then $k(a_i,a_j)=a_i^\top a_j$.

\subsubsection*{Unbiased case}
If $y=0$ is enforced in \eqref{eq:ksvm}, then the solution hyperplane $\{w:x^\top w=0\}$ passes through the origin and is called \emph{unbiased}. Consequently, the dual problem \eqref{eq:dksvm} will no longer have the linear constraint $\sum_i \beta_i s_i=0$, leaving it with the coordinate-wise separable box constraints $0\le s_i\le C$.  To solve \eqref{eq:dksvm}, we can apply the FBS operator $\cT$ defined by~\eqref{eq:FBS}. Let $d(s):= \frac{1}{2}s^\top Q s- e^\top s$}, $\cA = \prox_{[0, C]}$, and $\cC = \nabla d$.  The coordinate update based on FBS is
$$s_i^{k+1}=\prj_{[0,C]}(s_i^k-\gamma_i\nabla_i d(s^k)),$$
where we can take $\gamma_i=\frac{1}{Q_{ii}}$. 

\subsubsection*{Biased (general) case} In this case, the mode \eqref{eq:ksvm} has $y\in\RR$, so the hyperplane $\{w:x^\top w-y=0\}$ may not pass the origin and is called \emph{biased}. Then,  the dual problem \eqref{eq:dksvm} retains the linear constraint $\sum_i \beta_i s_i=0$. In this case, we apply the primal-dual splitting scheme \eqref{vucondat} or the three-operator splitting scheme \eqref{3s}.

The coordinate update based on the full primal-dual splitting scheme~\eqref{vucondat} is:
\begin{subequations}\label{eq:pd-svm}
\begin{align}
t^{k+1}=&\,t^k+\gamma\sum_{i=1}^m \beta_i s_i^k,\label{eq:pd-svm-t}\\
s_i^{k+1}=&\,\prj_{[0,C]}\left(s_i^k-\eta\big(\nabla_i d(s^k)+\beta_i(2t^{k+1}-t^k)\big)\right), \label{eq:pd-svm-s}
\end{align}
\end{subequations}
{where $t,s$ are the primal and dual variables, respectively. Note that we can let $w:=\sum_{i=1}^m \beta_i s_i$ and maintain it. With variable $w$ and substituting \eqref{eq:pd-svm-t} into \eqref{eq:pd-svm-s}, we can equivalently write \eqref{eq:pd-svm} into

\begin{equation}
{\left\{
\begin{array}{l}
\text{if }t\text{ is chosen (the index 0), then compute}\\
\qquad t^{k+1}=\,t^k+\gamma w^k,\\
\text{if }s_i\text{ is chosen (an index $i \in [m]$), then compute}\\
\qquad s_i^{k+1}=\,\prj_{[0,C]}\left(s_i^k-\eta\big(q_i^\top s^k-1+\beta_i(2\gamma w^k+t^k)\big)\right)\\
\qquad w^{k+1}= w^k + \beta_i (s_i^{k+1} - s_i^k).
\end{array}
\right.
}\end{equation}


We can also apply the three-operator splitting  \eqref{3s} as follows. Let {$D_1:=[0,C]^m$ and $D_2:=\{s: \sum_{i=1}^m \beta_i s_i=0\}$. Let $\cA = \prj_{D_2}$, $\cB = \prj_{D_1}$, and $\cC(x) = Qx - e$,  The full update corresponding to $\cT = (I - \eta_k) \cI + \eta_k \TS$ is
\begin{subequations}\label{eq:3op-svm}
\begin{align}
s^{k+1}=&\prj_{D_2}(u^k),\label{eq:3op-svm-s}\\
u^{k+1}=&u^k+\eta_k\left(\prj_{D_1}\big(2s^{k+1}-u^k-\gamma (Qs^{k+1}-e)\big)-s^{k+1}\right)\label{eq:3op-svm-u},
\end{align}
\end{subequations}
where $s$ is just an intermediate variable.
Let $\tilde{\beta}:=\frac{\beta}{\|\beta\|_2}$ and $w:=\tilde{\beta}^\top u$. Then $\prj_{D_2}(u)=(I-\tilde{\beta}\tilde{\beta}^\top)u$. Hence, $s^{k+1}=u^k-w^k\tilde{\beta}$. Plugging it into~\eqref{eq:3op-svm-u} yields the following coordinate update scheme:
\begin{equation*}
{\left\{
\begin{array}{l}
\text{if }i \in [m]\text{ is chosen, then compute}\\
\quad s^{k+1}_i=u^k_i-w^k\tilde{\beta}_i,\\
\quad u^{k+1}_i=u^k_i+\eta_k\left(\prj_{[0,C]}\left(2s_i^{k+1}-u^k_i-\gamma \big(q_i^\top u^k- w^k (q_i^\top \tilde{\beta})-1\big)\right)-s_i^{k+1}\right) \\
\quad w^{k+1}=w^k+\tilde{\beta}_i(u_i^{k+1}-u_i^k),
\end{array}
\right.
}\end{equation*}
where $w^k$ is the maintained variable and $s^k$ is the intermediate variable.

\subsubsection{Group Lasso}\label{sec:glasso}
The group Lasso regression problem \cite{YL2006GrpLasso} is 
\begin{equation}\label{eq:glasso}
\Min_{x\in\RR^n}  f(x)+\sum_{i=1}^m\lambda_i\|x_i\|_2,
\end{equation} 
where $ f$ is a differentiable convex function, often bearing the form $\frac{1}{2}\|Ax-b\|_2^2$, and $x_i\in \RR^{n_i}$ is a subvector of $x\in\RR^n$ supported on $\II_i\subset [n]$, and $\cup_i \II_i= [n]$. If $\II_i\cap \II_j =\emptyset, \,\forall i\neq j$, it is called \emph{non-overlapping group Lasso}, and if there are two different groups $\II_i$ and $\II_j$ with a non-empty intersection, it is called \emph{overlapping group Lasso}. The model finds a coefficient vector $x$ that minimizes the fitting (or loss) function $f(x)$ and that is group sparse: all but a few $x_i$'s are zero.  

Let $U_i$ be formed by the columns of the identity matrix $I$ corresponding to the indices in $\II_i$, and let $U=[U_1^\top;\ldots;U_m^\top]\in\RR^{(\Sigma_i n_i)\times n}$. Then, $x_i=U_i^\top x$. Let $h_i(y_i)=\lambda_i\|y_i\|_2,\,y_i\in\RR^{n_i}$ for $i \in [m]$, and $h(y)=\sum_{i=1}^m h_i(y_i)$ for $y=[y_1;\ldots;y_m]\in\RR^{\Sigma_i n_i}$. In this way, \eqref{eq:glasso} becomes
\begin{equation}\label{eq:glasso2}
\Min_x f(x)+h(Ux).
\end{equation}

\subsubsection*{Non-overlapping case~\cite{YL2006GrpLasso}} In this case, we have $\II_i\cap \II_j =\emptyset, \,\forall i\neq j$, and can apply the FBS scheme \eqref{eq:FBS} to \eqref{eq:glasso2}. Specifically, let $\cT_1=\partial (h\circ U)$ and $\cT_2=\nabla f$. The FBS full update is 
$$x^{k+1}=\cJ_{\gamma \cT_1}\circ(\cI-\gamma\cT_2)(x^k).$$
The corresponding coordinate update is the following
\begin{equation}\label{update-nonlap-gl}
{\left\{
\begin{array}{l}
\text{if }i \in [m]\text{ is chosen, then compute}\\
\qquad x_i^{k+1}=\argmin_{x_i}\frac{1}{2}\|x_i-x_i^k+\gamma_i \nabla_i f(x^k)\|_2^2+ \gamma_i h_i(x_i),\\
\end{array}
\right.
}\end{equation}
where $\nabla_i f(x^k)$ is the partial derivative of $f$ with respect to $x_i$ and the step size can be taken to be $\gamma_i=\frac{1}{\|A_{:,i}\|^2}$. When $\nabla f$ is either cheap or easy-to-maintain,  the coordinate update in \eqref{update-nonlap-gl} is inexpensive.

\subsubsection*{Overlapping case~\cite{jacob2009group}} This case allows $\II_i\cap \II_j \neq\emptyset$ for some $i\neq j$, causing the evaluation of $\cJ_{\gamma\cT_1}$ to be generally difficult. However, we can apply the primal-dual update \eqref{vucondat} to this problem as
\begin{subequations}\label{eq:pd-glasso}
\begin{align}
s^{k+1}=&\,\prox_{\gamma h^*} (s^k+\gamma U x^k),\label{eq:pd-glasso-s}\\
x^{k+1}=&\,x^k-\eta(\nabla f(x^k)+U^\top(2s^{k+1}-s^k))\DIFaddbegin \DIFadd{,}\DIFaddend \label{eq:pd-glasso-x}
\end{align}
where $s$ is the dual variable. 
\end{subequations}
Note that 
$$h^*(s)=\left\{
\begin{array}{ll}
0,&\mbox{if }\|s_i\|_2\le \lambda_i,\,\forall i,\\
+\infty,&\mbox{otherwise,}
\end{array}
\right.$$
is cheap.
Hence, the corresponding coordinate update of \eqref{eq:pd-glasso} is
\begin{equation}\label{update-lap-gs}
{\left\{
\begin{array}{l}
\text{if $s_i$  is chosen for some $i \in [m]$, then compute}\\
\quad s_i^{k+1}=\,\prj_{B_{\lambda_i}}(s_i^k+\gamma x_i^k) \\
\text{if $x_i$  is chosen for some $i \in [m]$, then compute}\\
\quad x_i^{k+1}=\,x_i^k-\eta\left(U_i^T \nabla f(x^k)+U_i^T \sum_{j, U_i^T U_j \neq 0} U_j(2\prj_{B_{\lambda_j}}(s_j^k+\gamma x_j^k)-s_j^k)\right),\\
\end{array}
\right.
}\end{equation}
where $B_\lambda$ is the Euclidean ball of radius $\lambda$. When $\nabla f$ is easy-to-maintain, the coordinate update in \eqref{update-lap-gs} is inexpensive. To the best of our knowledge, the coordinate update method~\eqref{update-lap-gs} is new.

\subsection{Imaging}\label{sec:Imaging}

\subsubsection{DRS for Image Processing in the Primal-dual Form \cite{o2014primal}}
Many convex image processing problems have the general form
$$\Min_{x} f(x) + g(Ax),$$
where $A$ is a matrix such as a dictionary, sampling operator, or finite difference operator. We can reduce the problem to the system: $0\in\cA(z) + \cB(z)$, where $z=[x;s]$,
$$\cA(z):=\begin{bmatrix}\partial f(x)\\\partial g^*(s)\end{bmatrix},\quad\mbox{and}\quad \cB(z):=\begin{bmatrix}0 & ~A^\top  \\ -A & 0\end{bmatrix}\begin{bmatrix}x\\s\end{bmatrix}.$$
(see Appendix~\ref{sec:vc-op} for the reduction.) The work \mbox{
\cite{o2014primal} }
gives their resolvents    
\begin{align*}\cJ_{\gamma \cA} &= \begin{bmatrix}\prox_{\gamma f}&0\\0&\prox_{\gamma g^*}\end{bmatrix},\\ 
\cJ_{\gamma \cB} &= (I+\gamma \cB)^{-1}= \begin{bmatrix} 0 &~ 0 \\ 0 &~ I\end{bmatrix}+\begin{bmatrix}I\\ \gamma A\end{bmatrix}(I+\gamma^2 A^\top A)^{-1}\begin{bmatrix}I\\ -\gamma A\end{bmatrix}^\top ,
\end{align*}
where $\cJ_{\gamma \cA}$ is often cheap or separable and we can \emph{explicitly form}  $\cJ_{\gamma \cB}$  as a matrix or implement it based on a fast transform. With the defined $\cJ_{\gamma \cA}$ and $\cJ_{\gamma \cB}$, we can apply the RPRS method as $z^{k+1} = \cT_{\text{RPRS}} z^k$. The resulting RPRS operator is CF when $\cJ_{\gamma \cB}$ is CF. Hence, we can derive a new RPRS coordinate update algorithm. We leave the derivation to the readers. Derivations of coordinate update algorithms for more specific image processing problems are shown in the following subsections.

\subsubsection{Total Variation Image Processing}
We consider the following Total Variation (TV) image processing model
\begin{equation}\label{eqn:tvl2}
\Min_{x}~\lambda \|x\|_{\text{TV}} + \frac{1}{2} \|A\,x - b\|^2,
\end{equation}
where $x\in \RR^n$ is the vector representation of the unknown image, $A$ is an $m \times n$ matrix describing the transformation from the image to the measurements $b$. Common $A$ includes sampling matrices in MRI, CT, denoising, deblurring, etc.  Let $(\nabla_i^h,\nabla_i^v)$ be the discrete gradient at pixel $i$ and $\nabla x=(\nabla_1^hx,\nabla_1^vx,\dots,\nabla_n^hx,\nabla_n^vx)^\top$. Then the TV semi-norm $\|\cdot\|_{\text{TV}}$ in the isotropic and anisotropic fashions are, respectively,
\begin{subequations}\label{eqn:tv_def}
\begin{align}
\|x\|_{\text{TV}} &= {\sum_{i} \sqrt{(\nabla_i^h x)^2 + (\nabla_i^v x)^2},}\\
  \| x\|_{\text{TV}} &= \| \nabla x\|_1 = \sum_{i} \left(|\nabla_i^h x| + |\nabla_i^v x|\right).
	\end{align}
\end{subequations}

%
%
%

For simplicity, we use the anisotropic TV for analysis and in the numerical experiment in \S~\ref{sec:tv}. It is {slightly more complicated for the isotropic TV. Introducing the following notation
$$B \DIFaddbegin \DIFadd{:}\DIFaddend = \begin{pmatrix} \nabla \\ A \end{pmatrix}, \quad h (p, q) \DIFaddbegin \DIFadd{:}\DIFaddend = \lambda \|p\|_1 + \frac{1}{2} \|q - b\|^2, $$
 {we can reformulate }\eqref{eqn:tvl2} \DIFaddend as
$$\Min_x ~ h(B\,x) = h(\nabla x, A \, x),$$
which {reduces to the form of~\eqref{pdproblem} with $f=g=0$. Based on its definition, the convex conjugate of $h(p, q)$ and its proximal operator are, respectively, 
\begin{align}
h^* (s, t) &= \iota_{\|\cdot\|_{\infty} \leq \lambda} (s) + \frac{1}{2} \|t + b\|^2 - \frac{1}{2} \|b\|^2, \label{eqn:dual-h}\\
\prox_{\gamma h^*} (s, t) &= \prj_{\|\cdot\|_{\infty} \leq \lambda} (s) + \frac{1}{1 + \gamma} (t - \gamma b). \label{eqn:prox-dual-h}
\end{align}
Let $s, t$ be the dual variables corresponding to $\nabla x$ and $Ax$ respectively, then using \eqref{eqn:prox-dual-h} and applying \eqref{vucondat2} give the following full update:
\begin{subequations}\label{eqn:pd_tvl2}
\begin{align}
x^{k + 1} &= x^k - \eta (\nabla^\top s^k + A^\top  t^k), \\
s^{k + 1} &= \prj_{\|\cdot\|_{\infty} \leq \lambda} \left(s^k + \gamma \nabla (x^k - 2\eta (\nabla^\top s^k + A^\top t^k))\right), \\ 
t^{k+1} &= \frac{1}{1 + \gamma} \left(t^k + \gamma A (x^k - 2 \eta (\nabla^\top s^k + A^\top t^k)) - \gamma b \right).
\end{align}
\end{subequations}
To perform the coordinate updates {as described in \S\ref{sec:p-d}}, we can maintain $\nabla^\top s^k$ and $A^\top t^k$. Whenever a coordinate of $(s,t)$ is updated, the corresponding $\nabla^\top s^k$ (or $A^\top t^k)$ should also be updated. Specifically, we have the following coordinate update algorithm
\begin{equation}\label{update-lap-gl}
{\left\{
\begin{array}{l}
\text{if $x_i$  is chosen for some $i \in [n]$, then compute}\\
\qquad x_i^{k + 1} = x_i^k - \eta (\nabla^\top s^k + A^\top  t^k)_i; \\
\text{if $s_i$  is chosen for some $i \in [2n]$, then compute}\\
\qquad s_i^{k + 1} = \prj_{\|\cdot\|_{\infty} \leq \lambda} \left(s_i^k + \gamma \nabla_i (x^k - 2\eta (\nabla^\top s^k + A^\top t^k))\right) \\
\qquad \text{and update $\nabla^\top s^k$ to $\nabla^\top s^{k+1}$}; \\
\text{if $t_i$  is chosen for some $i \in [m]$, then compute}\\
\qquad t_i^{k+1} = \frac{1}{1 + \gamma} \left(t_i^k + \gamma A_{i,:} (x^k - 2 \eta (\nabla^\top s^k + A^\top t^k)) - \gamma b_i \right)\\
\qquad \text{and update $A^{\top} t^k$ to $A^{\top} t^{k+1}$}. \\
\end{array}
\right.
}\end{equation}

\subsubsection{3D Mesh Denoising}
Following an example in~\cite{repetti2015random}, we consider a 3D mesh described by their nodes $\bar{x}_i=(\bar{x}_i^X,\bar{x}_i^Y,\bar{x}_i^Z), i\in[n]$, and the adjacency matrix $A\in\mathbb{R}^{n\times n}$, where $A_{ij} = 1$ if nodes $i$ and $j$ are adjacent, otherwise $A_{ij} = 0$. We let $\cV_i$  be the set of neighbours of node $i$. Noisy mesh nodes $z_i, i \in [n]$, are observed. We try to recover the original mesh nodes by solving the following optimization problem \cite{repetti2015random}:
\begin{equation}
\Min_{x} ~\sum_{i=1}^n f_i(x_i)+\sum_{i=1}^n g_i(x_i)+\sum_{i} \sum_{j\in \cV_{i}} h_{i,j}(x_{i}-x_{j}),
\end{equation}
where $f_i$'s are differentiable data fidelity terms, $g_i$'s are the indicator functions of box constraints, and $\sum_{i} \sum_{j\in \cV_{i}}h_{i,j}(x_{i}-x_{j})$ is the total variation on the mesh.

We introduce a dual variable $s$ with coordinates $s_{i,j}$, for all ordered pairs of adjacent nodes $(i,j)$, and, based on the overlapping-block coordinate updating scheme $\eqref{pdoverlap}$, perform coordinate update:
\begin{equation}
\left\{
\begin{array}{l}
\text{select $i$ from $[n]$, then compute }\\
\qquad \tilde{s}_{i,j}^{k+1}=\prox_{\gamma h_{i,j}^*} (s_{i,j}^k+\gamma x_i^k-\gamma x_j^k),\forall j\in\cV_i,\\
\qquad \tilde{s}_{j,i}^{k+1}=\prox_{\gamma h_{j,i}^*} (s_{j,i}^k+\gamma x_j^k-\gamma x_i^k),\forall j\in\cV_i,\\
\qquad \textnormal{and update}\\
\qquad {x_i}^{k+1}=\prox_{\eta g_i}(x_i^k-\eta(\nabla f_i(x_i^k)+\sum_{j\in\cV_i}(2\tilde{s}_{i,j}^{k+1}-2\tilde{s}_{j,i}^{k+1}-s_{i,j}^k+{s}_{j,i}^{k}))),\\
\qquad s_{i,j}^{k+1}=s_{i,j}^k+\frac{1}{2}(\tilde{s}_{i,j}^{k+1}-s_{i,j}^k),\forall j\in\cV_i,\\
\qquad s_{j,i}^{k+1}=s_{j,i}^k+\frac{1}{2}(\tilde{s}_{j,i}^{k+1}-s_{j,i}^k),\forall j\in\cV_i.
\end{array}
\right.\nonumber
\end{equation}

\subsection{Finance}

\subsubsection{Portfolio Optimization}
Assume that we have one unit of capital and $m$ assets to invest on. The $i$th asset has an expected return rate $\xi_i\ge 0$\cut{ for $i\in[m]$}. Our goal is to find a portfolio with the minimal risk such that the expected return is no less than $c$. This problem can be formulated as

\begin{equation*}
\begin{array}{l}
\displaystyle
\Min_x ~ \frac{1}{2} x^\top Q x, \\
\displaystyle
\text{subject to}~ x\ge0, \sum_{i=1}^m x_i\le 1,\, \sum_{i=1}^m\xi_i x_i\ge c,
\end{array}
\end{equation*}
where the objective function is a measure of risk, and the last constraint imposes that the expected return is at least $c$. Let $a_1=e/\sqrt{m}$, $b_1=1/\sqrt{m}$, $a_2=\xi/\|\xi\|_2$, and $b_2=c/\|\xi\|_2$, where $e = (1, \dots, 1)^\top, \xi = (\xi_1, \dots, \xi_m)^\top$. The above problem is rewritten as
\begin{equation}\label{eq:portfolio}
\Min_x \frac{1}{2} x^\top Q x,\ \text{subject to }\ x\ge0, ~a_1^\top x\le b_1,\, a_2^\top x\ge b_2.
\end{equation}
We apply the three-operator splitting scheme \eqref{3s} to \eqref{eq:portfolio}. Let $f(x)=\frac{1}{2}x^\top Q x$, $D_1=\{x: x\ge 0\}$, $D_2=\{x: a_1^\top x \le b_1,\, a_2^\top x\ge b_2\}$, $D_{21}=\{x: a_1^\top x=b_1\}$, and $D_{22}=\{x: a_2^\top x=b_2\}$. Based on \eqref{3s}, the full update is
\begin{subequations}\label{eq:3op-portfolio}
\begin{align}
y^{k+1}=&\prj_{D_2}(x^k),\\
x^{k+1}=&x^k+\eta_k\big(\prj_{D_1}(2y^{k+1}-x^k-\gamma \nabla f(y^{k+1}))-y^{k+1}\big), \end{align}
\end{subequations}
where $y$ is an intermediate variable. As the projection to $D_1$ is simple, we discuss how to evaluate the projection to $D_2$. Assume that $a_1$ and $a_2$ are neither perpendicular nor co-linear, i.e., $a_1^\top a_2\neq 0$ and $a_1\neq \lambda a_2$ for any scalar $\lambda$. In addition, assume $a_1^\top a_2>0$ for simplicity. Let $a_3=a_2-\frac{1}{a_1^\top a_2} a_1$, $b_3=b_2-\frac{1}{a_1^\top a_2} b_1$, $a_4=a_1-\frac{1}{a_1^\top a_2} a_2$, and $b_4=b_1-\frac{1}{a_1^\top a_2} b_2$. Then we can partition the whole space into four areas by the four hyperplanes $a_i^\top x=b_i$, $i=1,\ldots,4$. Let $P_i=\{x: a_i^\top x\le b_i, a_{i+1}^\top x\ge b_{i+1}\},\, i=1,2,3$ and $P_4=\{x: a_4^\top x\le b_4, a_1^\top x\ge b_1\}$. Then
$$\prj_{D_2}(x)=\left\{
\begin{array}{ll}
x, & \text{ if }x\in P_1,\\
\prj_{D_{22}}(x), &\text{ if }x\in P_2,\\
\prj_{D_{21}\cap D_{22}}(x), & \text{ if }x\in P_3,\\
\prj_{D_{21}}(x), & \text{ if }x\in P_4.
\end{array}
\right.$$
Let $w_i=a_i^\top x-b_i, i=1,2$, and maintain $w_1,w_2$. Let $\tilde{a}_2=\frac{a_2-a_1(a_1^\top a_2)}{1-(a_1^\top a_2)^2}$, $\tilde{a}_1=\frac{a_1-a_2(a_1^\top a_2)}{1-(a_1^\top a_2)^2}$. Then
\begin{align*}
&\prj_{D_{21}}(x)=x-w_1a_1,\\
&\prj_{D_{22}}(x)=x-w_2a_2,\\
&\prj_{D_{21}\cap D_{22}}(x)=x-w_1 \tilde{a}_1-w_2\tilde{a}_2,
\end{align*}
%
%
Hence, the coordinate update of~\eqref{eq:3op-portfolio} is
\begin{subequations}\label{eq:3op-portfolio2}
\begin{align}
x^k\in P_1:\ x_i^{k+1}= & \textstyle (1-\eta_k)x_i^k+\eta_k\max(0, x_i^k-\gamma q_i^\top x^k),\\
x^k\in P_2:\  x_i^{k+1}= & \textstyle (1-\eta_k)x_i^k+\eta_kw_2^k(a_2)_i+\eta_k\max\left(0, \right.\nonumber\\
&\textstyle \left.x_i^k-\gamma q_i^\top x^k-w_2^k(2(a_{2})_i-\gamma q_i^\top a_2)\right),\\
x^k\in P_3:\ x_i^{k+1}= &\textstyle  (1-\eta_k)x_i^k+\eta_k\left(w_1^k (\tilde{a}_{1})_i+w_2^k(\tilde{a}_{2})_i\right)+\eta_k\max\left(0,\right.\nonumber\\
&\textstyle \left.x_i^k-\gamma q_i^\top x^k-w_1^k (2(\tilde{a}_{1})_i-\gamma q_i^\top \tilde{a}_1)-{w}_2^k(2(\tilde{a}_{2})_i-\gamma q_i^\top \tilde{a}_2)\right),\\
x^k\in P_4:\ x_i^{k+1}= &\textstyle  (1-\eta_k)x_i^k+\eta_k w_1^k(a_{1})_i+\eta_k\max\left(0, \right.\nonumber\\
&\textstyle \left.x_i^k-\gamma q_i^\top x^k-w_1^k(2(a_{1})_i-\gamma q_i^\top a_1)\right),
\end{align}
\end{subequations}
where $q_i$ is the $i$th column of $Q$. At each iteration, we select $i \in [m]$, and perform an update to $x_i$ according to~\eqref{eq:3op-portfolio2} based on where $x^k$ is. We then renew $w_j^{k+1}=w_j^k+a_{ij}(x_i^{k+1}-x_i^k), j=1,2$. Note that checking $x^k$ in some $P_j$ requires only $O(1)$ operations by using $w_1$ and $w_2$, so the coordinate update in \eqref{eq:3op-portfolio2} is inexpensive.

\subsection{Distributed Computing}

\subsubsection{Network}\label{sec:network}
Consider that $m$ worker agents and one master agent form a star-shaped network, where the master agent at the center connects to each of the worker agents. The $m+1$ agents collaboratively solve the consensus problem:  $$\Min_{x} \sum_{i=1}^m f_i(x),$$ where $x\in \RR^d$ is the common variable and each proximable function $f_i$ is held privately by agent $i$. The problem can be reformulated as
\begin{align}
\Min_{x_1,\dots,x_m,y\in \RR^d} F(x) := \sum_{i=1}^mf_i(x_i),\quad \St~ x_i=y, ~\forall i \in [m],
\end{align}
{which has the  KKT condition}
\begin{equation}\label{pro:decentral}
0\in\underbrace{
\begin{bmatrix}
\partial F&0&0\\
0&0&0 \\
0&0&0
\end{bmatrix}}_{\mbox{operator}~\cA}\begin{bmatrix}
x\\
y\\
s
\end{bmatrix}+\underbrace{\begin{bmatrix}
0 & 0&I\\
0 & 0 & -e^\top \\
I & -e & 0
\end{bmatrix}}_{\mbox{operator}~\cC}\begin{bmatrix}
x\\
y\\
s
\end{bmatrix},
\end{equation}
where $s$ is the dual variable.

Applying the FBFS scheme \eqref{eqn:fbf} to~\eqref{pro:decentral} yields the following full update:
\begin{subequations}
\begin{align}
x_i^{k+1} &= \prox_{\gamma f_i}(x_i^k-\gamma s_i^k)+\gamma^2 x_i^{k}-\gamma^2 y^k-2\gamma s_i^k, \label{fbfs_a} \\
y^{k+1} &= (1+m\gamma^2)y^k +3\gamma \sum_{j=1}^m s_j^k -\gamma^2 \sum_{j=1}^m x_j^k,\\
s_i^{k+1} &= s_i^k-2\gamma x_i^k-\gamma \prox_{\gamma f_i}(x_i^k-\gamma s_i^k)+3\gamma y^k+ \gamma^2 \sum_{j=1}^m s_j^k, \label{fbfs_c}
\end{align}
\end{subequations}
where \eqref{fbfs_a} and \eqref{fbfs_c} are applied to all $i\in[m]$.
Hence, for each $i$, we group  $x_i$ and $s_i$ together and assign them on agent $i$.  {We let the master agent } maintain $\sum_j s_j$ and $\sum_j x_j$. Therefore, {in the FBFS coordinate update, updating any $(x_i,s_i)$ needs only $y$ and $\sum_j s_j$ from the master agent, and updating $y$ is done on the master agent. In synchronous parallel setting, at each iteration, each worker agent $i$ computes $s_i^{k+1}, x_i^{k+1}$, then the master agent collects the updates from all of the worker agents and then updates $y$ and $\sum_j s_j$. The above update can be relaxed to be asynchronous. In this case, the master and worker agents work concurrently, the master agent updates $y$ and $\sum_j s_j$ as soon as it receives the updated $s_i$ and $x_i$ from any of the worker agents. It also periodically broadcasts $y$ back to the worker agents.


\cut{
\subsection{Total Variation Image Processing}\label{sec:tvip}
The most widely studied (non-local) Total Variation (TV) model \cite{?} for image processing can been written as
\begin{align}
\Min_x h(Ax) + f(x)
\end{align}
where $f(Ax)$ is the (non-local) TV regularization term and $f(x)$ is the data fitting term that dependent on the image processing problems. If we rearange the image pixels into a vector, and $A$ can be expressed as a sparse matrix for (non-local) TV. For example, the matrix $A$ for the isotropic TV on a $m$ by $n$ image has the size $2mn\times mn$ and can be expressed as
\begin{align}
(Ax)_{(i-1)*m+j} = x_{i*m+j}-x_{(i-1)*m+j},
\end{align}
i.e., the difference between values at the $(i+1,j)$ pixel and the $(i,j)$ pixel.

In order to solve this optimization problem, we can derive the equivalent KKT condition as follows:
\begin{equation}
0\in\underbrace{
\begin{bmatrix}
\partial f(x)\\
\partial h^*(s)
\end{bmatrix}}_{\mbox{operator}~\cA}+\underbrace{\begin{bmatrix}
0&A\\
-A&0
\end{bmatrix}\begin{bmatrix}
x\\
s
\end{bmatrix}}_{\mbox{operator}~\cB}
\end{equation}
}


\subsection{Dimension Reduction}

\subsubsection{Nonnegative Matrix Factorization}
Nonnegative Matrix Factorization (NMF) is an important dimension reduction method for nonnegative data. It was proposed by Paatero and his coworkers in \cite{paatero1994NMF}. Given a nonnegative matrix $A\in\RR^{p\times n}_+$, NMF aims at finding two nonnegative matrices $W\in\RR^{p\times r}_+$ and $H\in\RR^{n\times r}_+$ such that $WH^\top\approx A$, where $r$ is user-specified depending on the applications, and usually $r\ll \min(p,n)$. A widely used model is
\begin{equation}\label{eq:nmf}
\begin{aligned}
&\Min_{W,H}F(W,H):=\frac{1}{2}\|WH^\top-A\|_F^2, \\
& \St\ W\in\RR^{p\times r}_+,\, H\in \RR^{n\times r}_+. 
\end{aligned}
\end{equation}
Applying the projected gradient method \eqref{eq:proj-grad} to \eqref{eq:nmf}, we have
\begin{subequations}\label{pg-nmf}
\begin{align}
&W^{k+1}=\max\big(0, W^k-\eta_k \nabla_W F(W^k, H^k)\big),\\
&H^{k+1}=\max\big(0,H^k-\eta_k \nabla_H F(W^k, H^k) \big).
\end{align}
\end{subequations}
In general, we do not know the Lipschitz constant of $\nabla F$, so we have to choose $\eta_k$ by line search such that the Armijo condition is satisfied.

Partitioning the variables into $2r$ block coordinates: $(w_1, \ldots, w_r, h_1,\ldots, h_r)$ where $w_i$ and $h_i$ are the $i$th columns of $W$ and $H$, respectively, we can apply the coordinate update based on the projected-gradient method: 
\begin{equation}\label{pcg-nmf}
{\left\{
\begin{array}{l}
\text{if $w_{i_k}$  is chosen for some $i_k \in [r]$, then compute}\\
\qquad w_{i_k}^{k+1}=\max\big(0, w_{i_k}^k-\eta_k\nabla_{w_{i_k}} F(W^k, H^k)\big); \\
\text{if $h_{i_k-r}$  is chosen for some $i_k \in \{r+1, ..., 2r\}$, then compute}\\
\qquad h_{i_k-r}^{k+1}=\max\big(0, h_{i_k-r}^k-\eta_k \nabla_{h_{i_k-r}} F(W^k, H^k) \big). 
\end{array}
\right.
}\end{equation}
It is easy to see that $\nabla_{w_i} F(W^k, H^k)$ and $\nabla_{h_i} F(W^k, H^k)$ are both Lipschitz continuous with constants $\|h_i^k\|_2^2$ and $\|w_i^k\|_2^2$ respectively. Hence, we can set 
$$\eta_k=\left\{\begin{array}{ll}
\frac{1}{\|h_{i_k}^k\|_2^2},&\text{ if }1\le i_k\le r,\\[0.1cm]
\frac{1}{\|w_{i_k-r}^k\|_2^2},&\text{ if }r+1\le i_k\le 2r.
\end{array}\right.$$
However, it is possible to have $w_i^k=0$ or $h_i^k=0$ for some $i$ and $k$, and thus the setting in the above formula may have trouble of being divided by zero. To overcome this problem, one can first modify the problem \eqref{eq:nmf} by restricting $W$ to have unit-norm columns and then apply the coordinate update method in \eqref{pcg-nmf}. Note that the modification does not change the optimal value since $WH^\top=(WD)(HD^{-1})^\top$ for any $r\times r$ invertible diagonal matrix $D$. We refer the readers to \cite{XY_2014_ecd} for more details.

Note that
$\nabla_W F(W,H)= (WH^\top-A)H, \nabla_H F(W,H)= (WH^\top-A)^\top W$
and 
$\nabla_{w_i} F(W,H)= (WH^\top-A)h_i, \nabla_{h_i} F(W,H)= (WH^\top-A)^\top w_i,\, \forall i.$
Therefore, the {coordinate updates given in \eqref{pcg-nmf} are computationally worthy (by maintaining the residual $W^k(H^k)^\top-A$).

\subsection{Stylized Optimization}
\subsubsection{Second-Order Cone Programming (SOCP)}\label{sec:socp}
SOCP extends LP by incorporating second-order cones. A second-order cone in $\RR^n$ is $$Q=\big\{(x_1,x_2,\ldots,x_n)\in\RR^n:\|(x_2,\ldots,x_n)\|_2\le x_1\big\}.$$
Given a point $v\in\RR^n$, let  $\rho_1^v:=\|(v_2,\ldots,v_n)\|_2$ and $\rho_2^v:=\frac{1}{2}(v_1+\rho_1^v)$}. Then, the projection {of  $v$ to $Q$ {returns $0$ if $v_1<-\rho_1^v$, returns $v$ if $v_1 \ge \rho_1^v$, and returns $(\rho_2^v,\frac{\rho_2^v}{\rho_1^v}\cdot(v_2,\ldots,v_n))$ otherwise. Therefore, if we define the scalar couple:
$$(\xi_1^v,\xi_2^v) = \begin{cases}(0,0),&\quad v_1<-\rho_1^v,\\
(1,1), &\quad v_1 \ge \rho_1^v,\\
\big(\rho_2^v,\frac{\rho_2^v}{\rho_1^v}\big),&\quad \mbox{otherwise},\end{cases}
$$
then we have $u=\prj_{Q}(v) = \big(\xi_1 ^vv_1,\,\xi_2^v\cdot (v_2,\ldots,v_n)\big)$. Based on this, we have
\begin{proposition}\label{prop:socproj}
\begin{enumerate}
\item Let $v\in\RR^{n}$ and $v^+ := v+ \nu e_i$ for any $\nu\in\RR$. Then, given $\rho_1^v,\rho_2^v,\xi_1^v,\xi_2^v$ {defined above}, it takes $O(1)$ operations to obtain $\rho_1^{v^+},\rho_2^{v^+},\xi_1^{v^+},\xi_2^{v^+}$.
\item Let $v\in\RR^{n}$ and $A=[a_1~A_2]\in\RR^{m\times n}$, where $a_1\in\RR^m,A_2\in\RR^{m\times (n-1)}$. Given $\rho_1^v,\rho_2^v,\xi_1^v,\xi_2^v$, we have $$A(2\cdot\prj_{Q}(v)-v)=((2\xi_1 ^v-1)v_1)\cdot a_1 + (2\xi_2^v-1)\cdot A_2 (v_2,\ldots,v_n)^\top .$$
\end{enumerate}
\end{proposition}
By the proposition, if $\cT_1$ is an affine operator, then in the composition $\cT_1\circ \prj_{Q}$, {the computation of $\prj_{Q}$ is cheap as long as we maintain $\rho_1^v,\rho_2^v,\xi_1^v,\xi_2^v$.

Given $x,c\in\RR^n$, $b\in\RR^m$, and $A\in\RR^{m\times n}$, the standard form of SOCP is
\begin{subequations}\label{eq:socp}
\begin{align}
\Min_x~c^\top x,&\quad\St~ Ax=b,\\
&\hspace{56pt} x\in X=Q_1\times \cdots\times Q_{\bar{n}},
\end{align}
\end{subequations}
where each $Q_i$ is a second-order cone, and $\bar{n}\not=n$ in general.
The problem~\eqref{eq:socp} is equivalent to
$$\Min_x \big(c^\top  x+\iota_{A\cdot =b}(x)\big)+\iota_X(x),$$
to which we can apply the DRS iteration $z^{k+1} = \TDRS (z^k)$ (see \eqref{eq:DRS}), in which $\cJ_{\gamma \cA} = \prj_X$ and $\cT_{\gamma \cB}$ is a linear operator given by
$$\cJ_{\gamma \cB} (x) = \argmin_y ~ c^\top y+  \frac{1}{2\gamma}\|y-x\|^2\quad \St~ Ay=b.$$
Assume that the matrix $A$ has full row-rank (otherwise, $Ax=b$ has either redundant rows or no solution).
Then, in~\eqref{eq:DRS}, we have $\cR_{\gamma \cB}(x)= Bx+d$, where $B:=I-2A^\top (AA^\top )^{-1}A$ and $d:=2A^\top (AA^\top )^{-1}(b+\gamma Ac)-2\gamma c$.

It is easy to apply coordinate updates to $z^{k+1} = \TDRS (z^k)$ following Proposition \ref{prop:socproj}. Specifically, by maintaining the scalars $\rho_1^v,\rho_2^v,\xi_1^v,\xi_2^v$ for each $v=x_i\in Q_i$ during  coordinate updates, the computation of the projection can be completely avoided. We pre-compute $(AA^\top )^{-1}$ and cache the matrix $B$ and vector $d$. Then, $\TDRS$ is CF, and we have the following coordinate update method
\begin{equation}\label{cf-cone}
{\left\{
\begin{array}{l}
\text{select $i \in [\bar{n}]$, then compute}\\
\qquad y^{k+1}_i = B_i x^k + d_i \\
\qquad x_i^{k+1}=\,\prj_{Q_i}(y_i^{k+1}) + \frac{1}{2} ( x_i^k - y_i^{k+1}), \\
\end{array}
\right.
}\end{equation}
where $B_i \in \RR^{n_i \times n}$ is the $i$th row block submatrix of $B$, and $y_i^{k+1}$ is the intermediate variable.

It is trivial to extend this method for  SOCPs with a quadratic objective:
\begin{align*}
\Min_x~c^\top x+\frac{1}{2}x^\top Cx,&\quad\St~ Ax=b,~ x\in X=Q_1\times \cdots\times Q_{\bar{n}},
\end{align*}
because $\cJ_2$ is still linear.  Clearly, this method applies to linear programs as they are special SOCPs.

Note that many LPs and SOCPs have sparse matrices $A$, which deserve further investigation. In particular, we may prefer not to form $(AA^\top )^{-1}$ and use the results in \S\ref{sec:emp} instead.

\cut{
\subsection{SOCP by primal-dual coordinate update}
SOCP can be viewed as an extended monotropic programming problem $\eqref{emp}$, as illustrated in \S$\ref{sec:emp}$. So it can be solved by using $\eqref{pdemp}$ with a dual variable $s$:
\begin{equation}
\left\{
\begin{array}{l}
s^{k+1}=s^k+\gamma (Ax^k-b)\\
x^{k+1}=\Proj_{Q_1\times\cdots\times Q_{\bar{n}}}(x^k-\eta(c+A^\top s^k+2\gamma A^\top Ax^k-2\gamma A^\top b))
\end{array}
\right.,
\end{equation}
which is CF.
\subsection{Conic programming self-dual embedding}
Self-dual embedding reduces a pair of primal dual conic programs to a single convex feasibility problem. Unlike the previous approach, which focuses on solving either the primal or dual problem, operator splitting can be applied to solve the single feasibility problem. When the problem is optimal, it will return a solution; otherwise, it will return a certificate that proves either primal or dual infeasibility. Self-dual embedding was introduced in \cite{?} and led to solvers in \cite{?,?}. In particular, the solver SCS ....
\subsection{Quadratic programming}
The quadratic programming $\eqref{qp}$
\begin{equation}
\underset{x\in\mathbb{R}^m}{\text{minimize }} \frac{1}{2}x^\top Ux+c^\top x,\text{ subject to } Ax=b,~x\in X,
\end{equation}
where $U$ is a positive semidefinite matrix, $X=\{x:x_i \geq 0,\forall i\geq j\}$, can be viewed as $\eqref{emp}$ and solved by using $\eqref{pdemp}$ with a dual variable $s$:
\begin{equation}
\left\{
\begin{array}{l}
s^{k+1}=s^k+\gamma (Ax^k-b)\\
x^{k+1}=\Proj_{X}(x^k-\eta(Ux^k+c+A^\top s^k+2\gamma A^\top Ax^k-2\gamma A^\top b))
\end{array}
\right.,
\end{equation}
which is CF since $\Proj_X$ belongs to $\cC_1$.
}

\section{Numerical Experiments}\label{sec:numerical}
We illustrate the behavior of coordinate update algorithms for solving portfolio optimization, image processing, and sparse logistic regression problems. Our primary goal is to show the efficiency of coordinate update algorithms compared to the corresponding full update algorithms. We will also illustrate that asynchronous parallel coordinate update algorithms are more scalable than their synchronous parallel counterparts. 

Our first two experiments run on Mac OSX 10.9 with 2.4 GHz Intel Core i5 and 8 Gigabytes of RAM. The experiments were coded in Matlab. The sparse logistic regression experiment runs on 1 to 16 threads on a machine with two 2.5 Ghz 10-core Intel Xeon E5-2670v2 (20 cores in total) and $64$ Gigabytes of RAM. The experiment was coded in C++ with OpenMP enabled. We use the Eigen library\footnote{\url{http://eigen.tuxfamily.org}} for sparse matrix operations.

\cut{\subsection{Least Square}\label{lsqexperiment2}
In this subsection, we compare the efficiency of four different coordinate update schemes (cyclic, cyclic shuffle, random, and greedy with Gauss-Southwell rule) with the full gradient descent method for solving the least square problem
$$\Min_{x} \frac{1}{2} \|A x - b\|^2,$$
where $A \in \RR^{m \times n}$ and $b \in \RR^m$. We solve the above problem with the following update scheme
$$x^{k+1} = x^k - \eta_k A^{\top}(A x^k - b),$$
where $\eta_k$ is the step size. This test uses three datasets, which are summarized in Table \ref{tab:ls-data}.  
\begin{table}[!hbtp]
\centering
 \begin{tabular}{lrrrr}
  \toprule
     & $m$  & $n$ & $A$ & $b$\\
   \midrule
   Dataset I & 1000 & 1000 & \texttt{diag([1:m])} & \texttt{ones(m, 1)} \\
   Dataset II & 1000 & 500 & \texttt{randn(m, n)} & \texttt{ones(m, 1)} \\
   Dataset III & 1000 & 500 & \texttt{rand(m, n)} & \texttt{ones(m, 1)} \\
   \bottomrule
\end{tabular}
 \caption{Three datasets for the least square problem\label{tab:ls-data}}
\end{table}

For both full gradient descent method and greedy coordinate update method with Gauss-Southwell rule, the step size $\eta$ is set to $\frac{2}{\|A\|_2^2}$. For the other three methods, if coordinate $i$ is selected, then $\eta_k$ is set to $\frac{1}{(A^{\top}A)_{ii}}$. Figure \ref{fig:ls_a} shows that for Dataset I, both cyclic update and cyclic shuffle update converge to the optimal solution with one epoch. Random coordinate update converges to the optimal solution after $8$ epochs. However, the gradient descent algorithm and greedy algorithm converge very slowly. This is due to the small step size ($\eta_k = 10^{-6}$). For the other two datasets, we observe that random coordinate update and cyclic shuffle coordinate update give consistent better performance than the full gradient descent algorithm. 
\begin{figure}[!htbp] \centering
    \begin{subfigure}[b]{0.3\linewidth}
        \includegraphics[width=40mm]{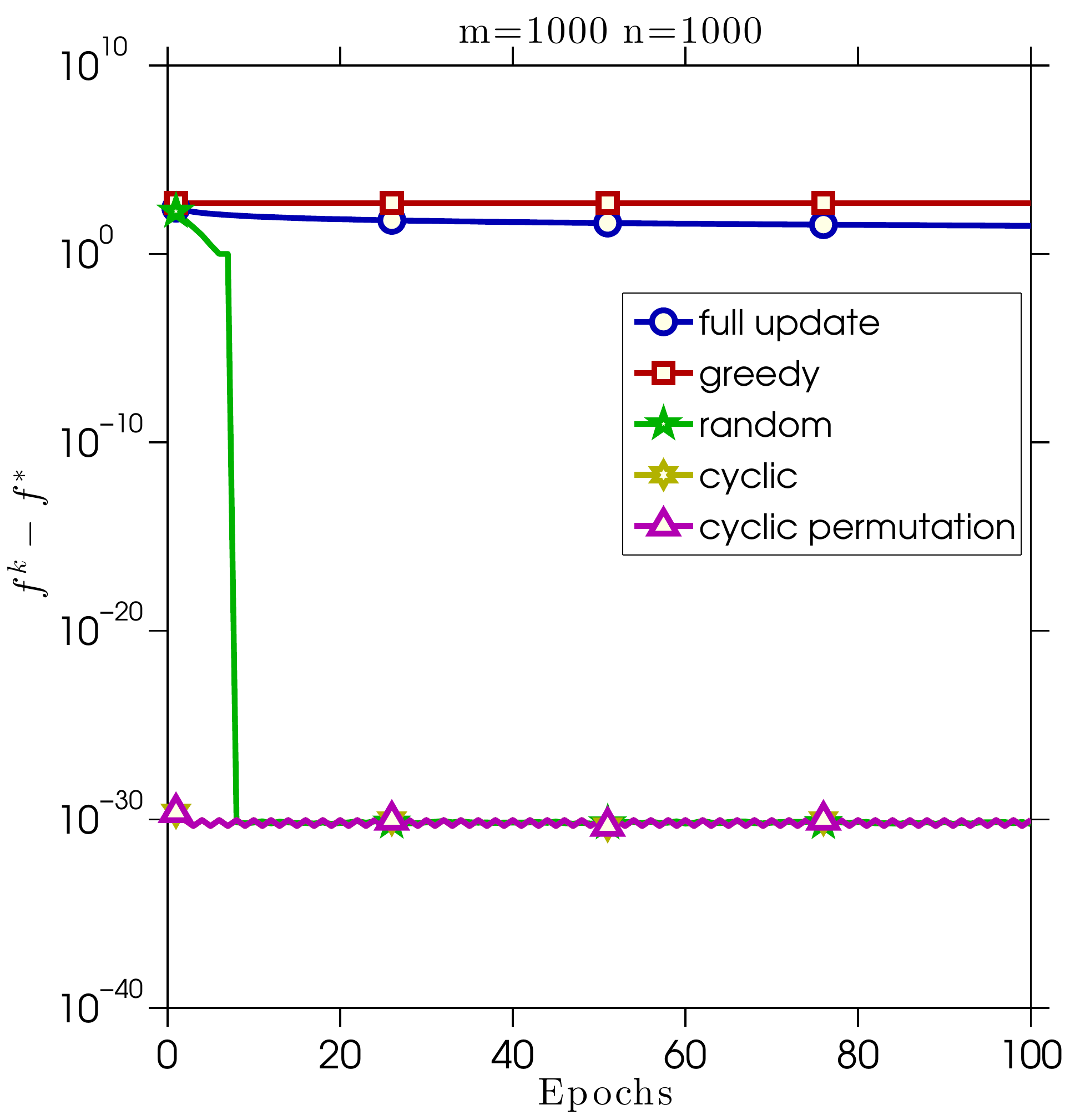}
        \caption{Dataset I}
        \label{fig:ls_a}
    \end{subfigure} %
    \quad
    \begin{subfigure}[b]{0.3\linewidth}
        \includegraphics[width=40mm]{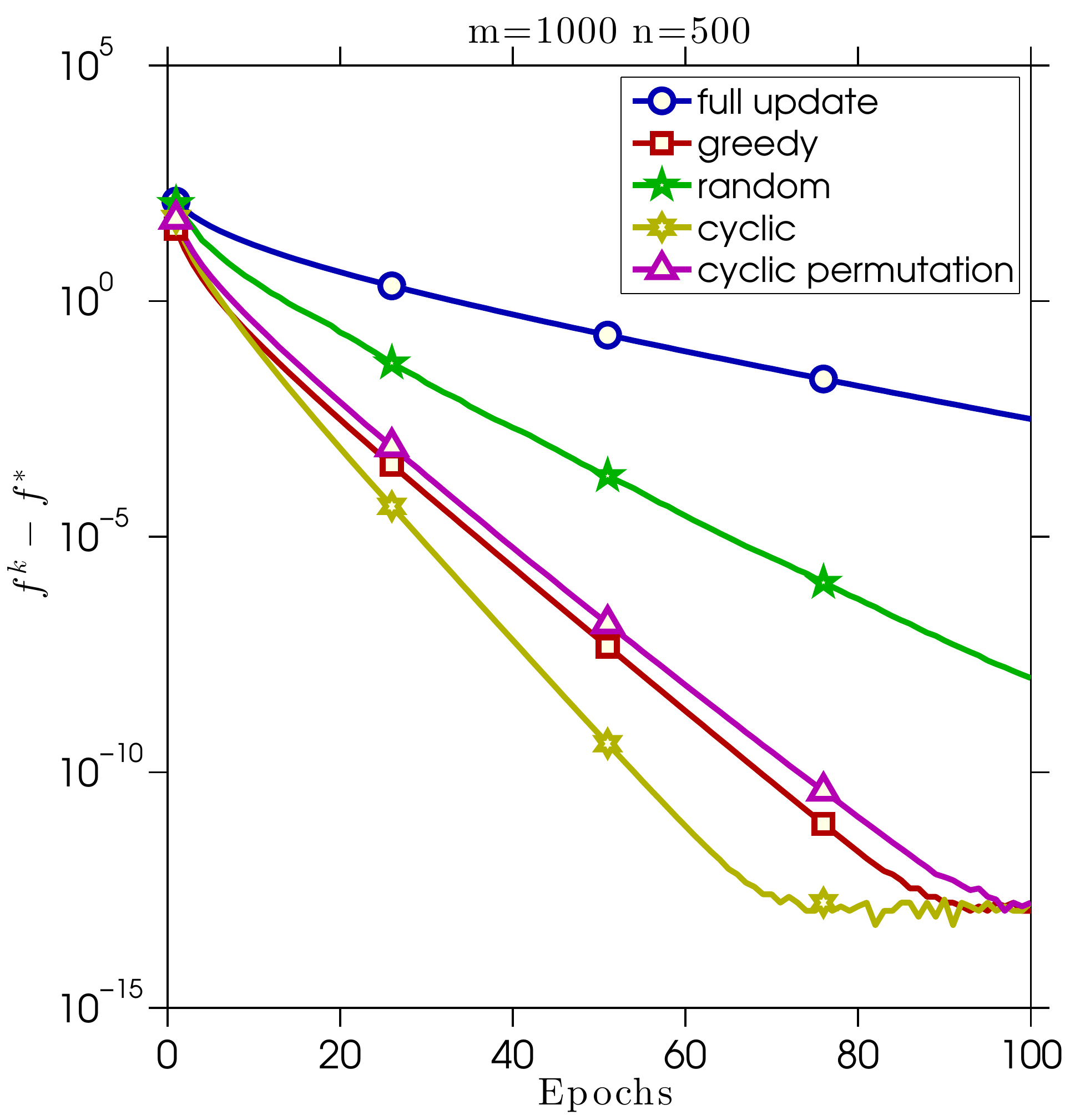}
        \caption{Dataset II}
        \label{fig:ls_b}
    \end{subfigure} %
    \quad
    \begin{subfigure}[b]{0.3\linewidth}
        \includegraphics[width=40mm]{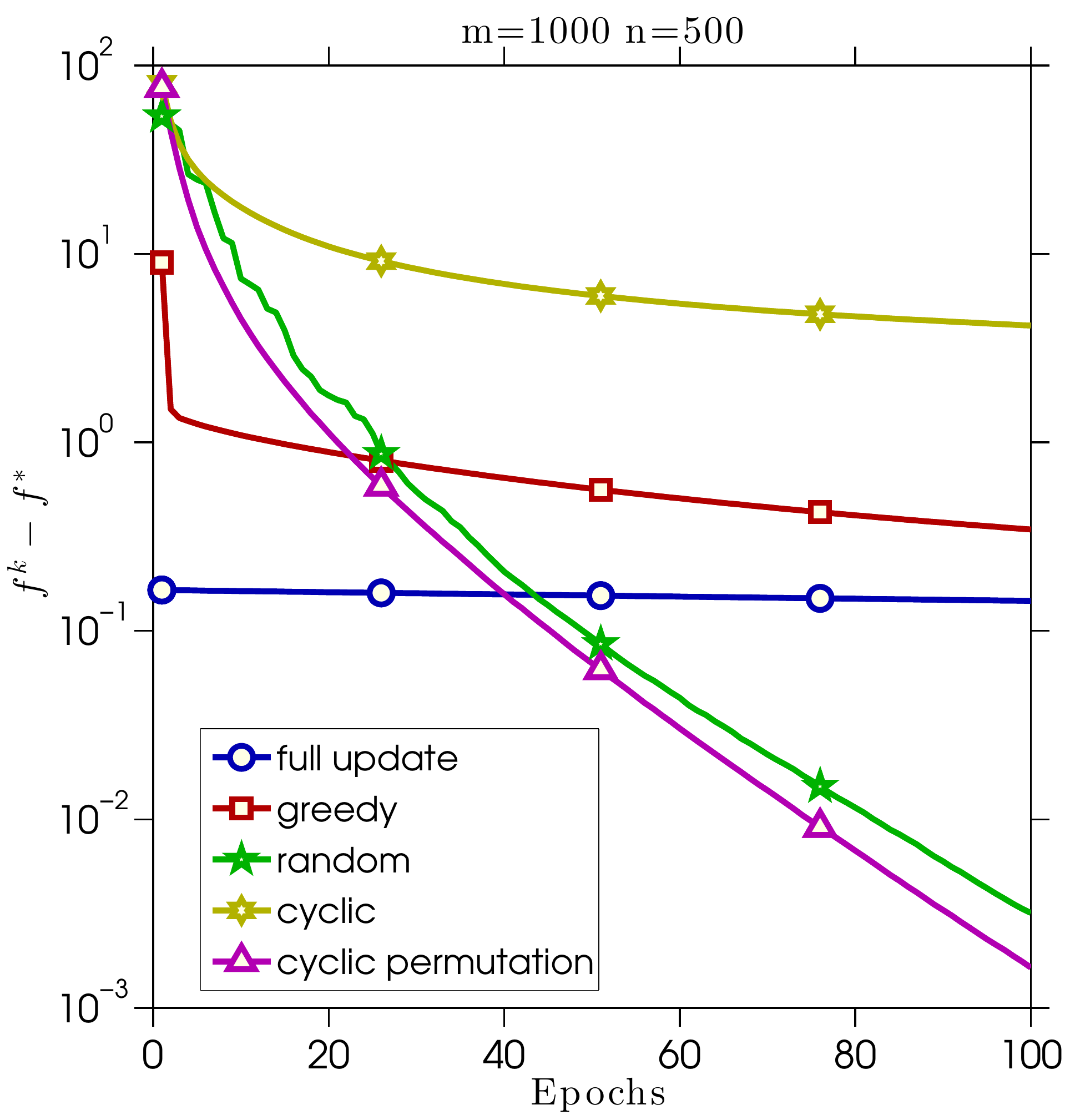}
        \caption{Dataset III}
        \label{fig:ls_c}
    \end{subfigure} %
    \caption{Compare the convergence of four different coordinate update algorithms with full gradient descent algorithm.}
    \label{fig:3s_results}
\end{figure}
}

\subsection{Portfolio Optimization}
In this subsection, we compare the performance of the 3S splitting scheme~\eqref{eq:3op-portfolio} with the corresponding coordinate update algorithm~\eqref{eq:3op-portfolio2} for solving the portfolio optimization problem~\eqref{eq:portfolio}. In this problem, our goal is to distribute our investment resources to all the assets so that the investment risk is minimized and the expected return is greater than $c$. This test uses two datasets, which are summarized in Table~\ref{tab:3s-data}. The NASDAQ dataset is collected through Yahoo! Finance. We collected one year (from 10/31/2014 to 10/31/2015) of historical closing prices for 2730 stocks. 

\begin{table}[htbp]
\centering
 \begin{tabular}{rrr}
  \toprule
    & Synthetic data  & NASDAQ data\\
   \midrule
   Number of assets (N) & 1000 & 2730 \\
   Expected return rate & 0.02 & 0.02 \\
   Asset return rate & \texttt{3 * rand(N, 1) - 1} & mean of 30 days return rate \\
   Risk & covariance matrix + $0.01\cdot I$ & positive definite matrix \\
   \bottomrule
\end{tabular}
 \caption{Two datasets for portfolio optimization \label{tab:3s-data}}
\end{table}

In our numerical experiments, for comparison purposes, we first obtain a high accurate solution by solving~\eqref{eq:portfolio} with an interior point solver. For both full update and coordinate update, $\eta_k$ is set to 0.8. However, we use different $\gamma$. For 3S full update, we used the step size parameter $\gamma_1 = \frac{2}{\|Q\|_2}$, and for 3S coordinate update, $\gamma_2 = \frac{2}{\max\{Q_{11}, ..., Q_{NN}\}}$. In general, coordinate update can benefit from more relaxed parameters. The results are reported in Figure \ref{fig:3s_results}. We can observe that the coordinate update method converges much faster than the 3S method for the synthetic data. This is due to the fact that $\gamma_2$ is much larger than $\gamma_1$. However, for the NASDAQ dataset, $\gamma_1 \approx \gamma_2$, so 3S coordinate update is only moderately faster than 3S full update.

\begin{figure} \centering
    \begin{subfigure}[b]{0.45\linewidth}
        \includegraphics[width=60mm]{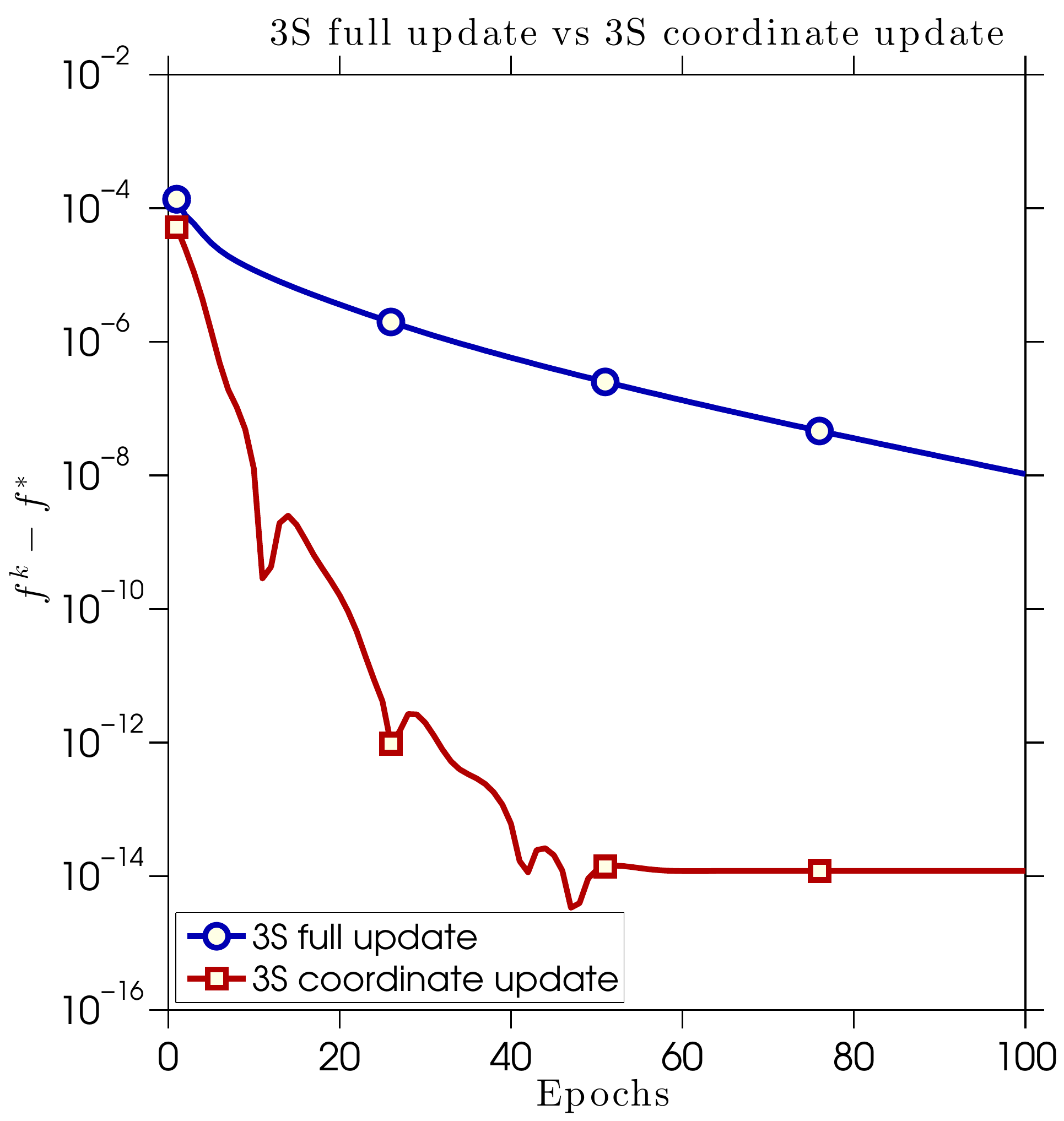}
        \caption{Synthesis dataset}
        \label{fig:3s_synth}
    \end{subfigure} %
    \quad
    \begin{subfigure}[b]{0.45\linewidth}
        \includegraphics[width=60mm]{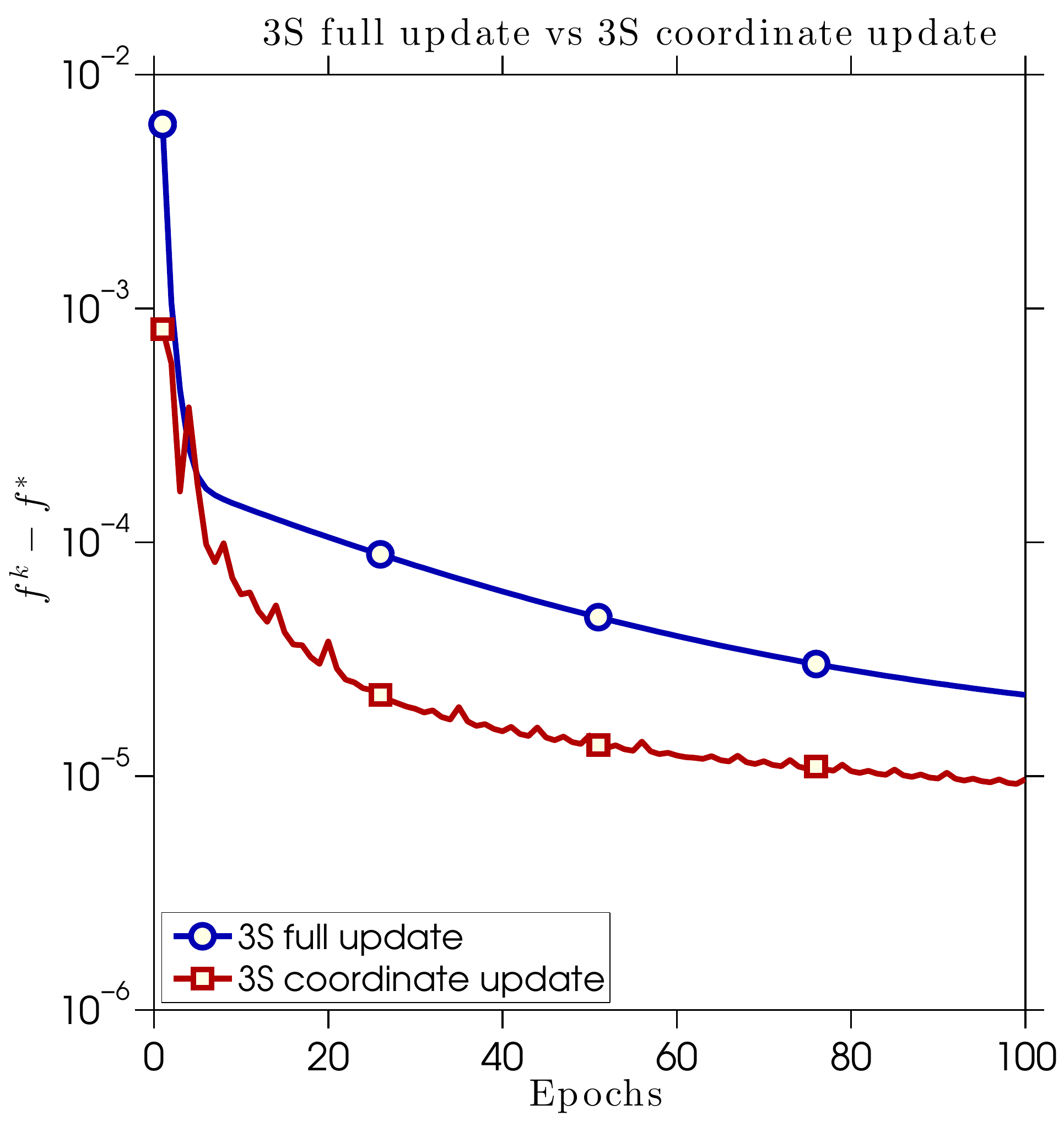}
        \caption{NASDAQ dataset}
        \label{fig:3s_real}
    \end{subfigure} %
    \caption{Compare the convergence of 3S full update with 3S coordinate update algorithms.}
    \label{fig:3s_results}
\end{figure}

\subsection{Computed Tomography Image Reconstruction}\label{sec:tv}
We compare the performance of algorithm~\eqref{eqn:pd_tvl2} and its corresponding coordinate version on Computed Tomography (CT) image reconstruction. We generate a thorax phantom of size $284\times 284$ to simulate spectral CT measurements.  We then apply the Siddon's algorithm~\cite{Siddon} to form the sinogram data. There are 90 parallel beam projections and, for each projection, there are 362 measurements. Then the sinogram data is corrupted with Gaussian noise. We formulate the image reconstruction problem in the form of~\eqref{eqn:tvl2}. The primal-dual full update corresponds to \eqref{eqn:pd_tvl2}. For coordinate update, the block size for $x$ is set to 284, which corresponds to a column of the image. The dual variables $s, t$ are also partitioned into 284 blocks accordingly. A block of $x$ and the corresponding blocks of $s$ and $t$ are bundled together as a single block. In each iteration, a bundled block is randomly chosen and updated. The reconstruction results are shown in Figure~\ref{fig:pds_results}. After 100 epochs, the image recovered by the coordinate version is better than that by~\eqref{eqn:pd_tvl2}. As shown in Figure~\ref{fig:pds_d}, the coordinate version converges faster than~\eqref{eqn:pd_tvl2}.


\begin{figure}[!htb] \centering
    \begin{subfigure}{0.45\linewidth}
        \centering
        \includegraphics[width=0.9\linewidth]{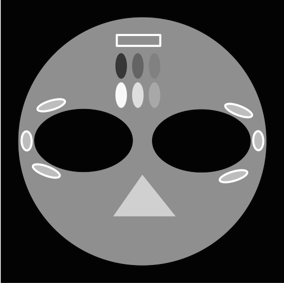}
         \caption{Phantom image}\label{fig:pds_a}               
    \end{subfigure} %
    \quad
    \begin{subfigure}{0.45\linewidth}
        \centering
        \includegraphics[width=0.9\linewidth]{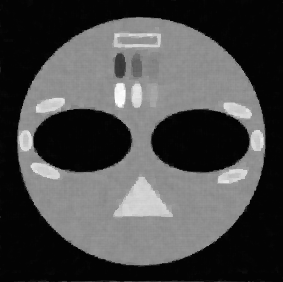}
         \caption{Recovered by PDS}\label{fig:pds_b}                    
    \end{subfigure} %
    \begin{subfigure}{0.45\linewidth}
        \centering
        \includegraphics[width=0.9\linewidth]{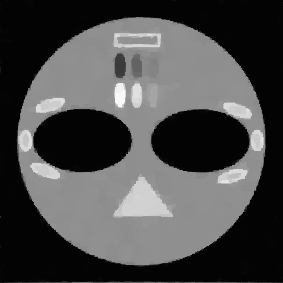}
         \caption{Recovered by PDS coord}\label{fig:pds_c}                      
    \end{subfigure} %
    \quad
    \begin{subfigure}{0.45\linewidth}
        \centering
        \includegraphics[width=\linewidth]{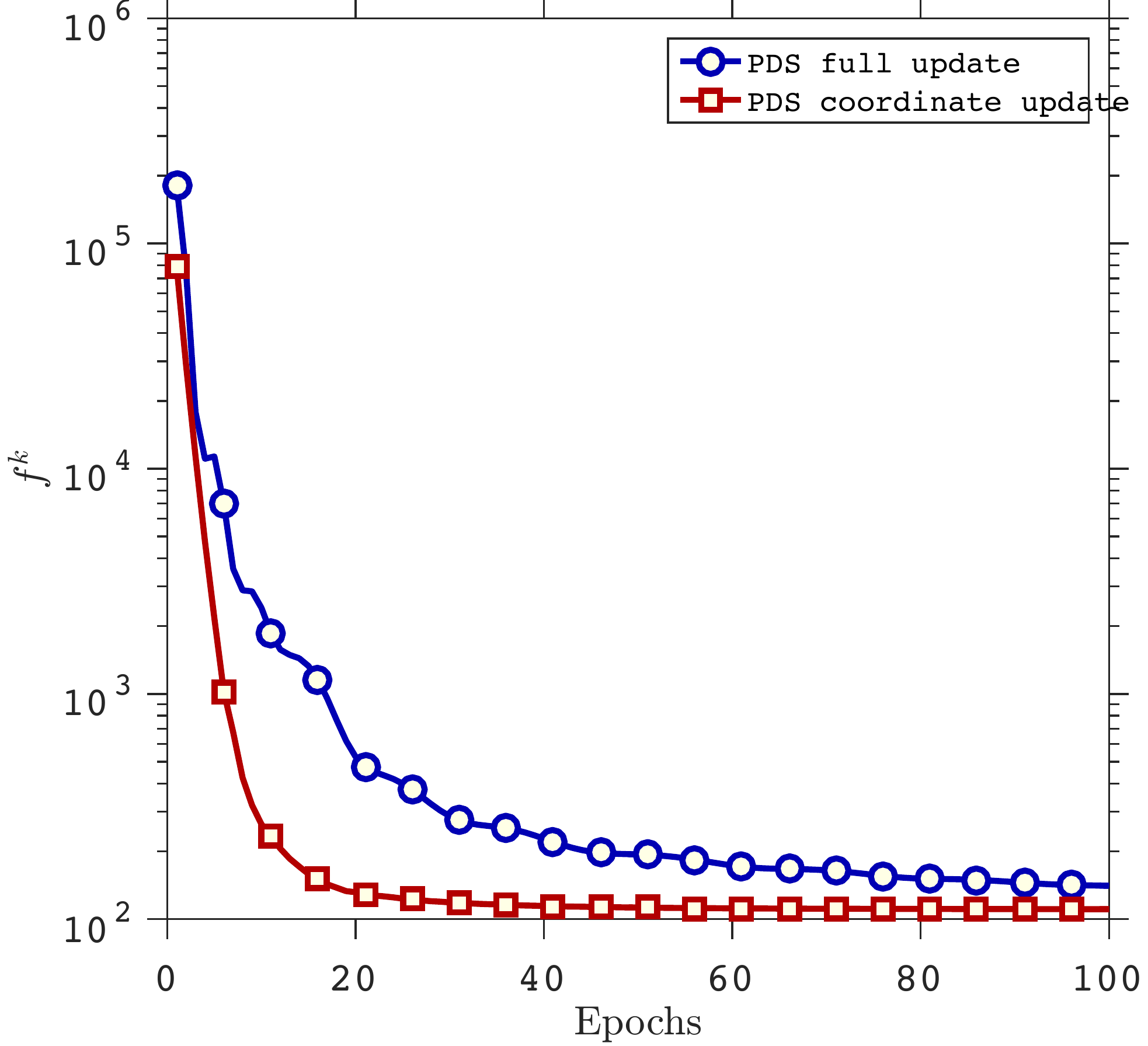}
         \caption{Objective function value}\label{fig:pds_d}                            
    \end{subfigure} %
    \caption{CT image reconstruction.}
    \label{fig:pds_results}
\end{figure}

\subsection{$\ell_1$ Regularized Logistic Regression}
In this subsection, we compare the performance of sync-parallel coordinate update and  async-parallel coordinate update for solving the sparse logistic regression problem 
\begin{equation}\label{eqn:log}
\Min_{x \in \mathbb{R}^n} \lambda \|x\|_1 + \frac{1}{N} \sum_{j=1}^N \log\big(1 + \exp(-b_j \cdot a_j^{\top} x)\big),
\end{equation}
where $\{(a_j, b_j)\}_{j=1}^N$ is the set of sample-label pairs with $b_j \in \{1, -1\}$, $\lambda=0.0001$, and $n$ and $N$ represent the numbers of features and samples, respectively. This test uses the datasets\footnote{\url{http://www.csie.ntu.edu.tw/~cjlin/libsvmtools/datasets/}}: real-sim and news20, which are summarized in Table \ref{tab:log_data}.

\begin{table}[htbp]
\centering
 \begin{tabular}{rrr}
\hline
  Name & \# samples & \# features \\
  \hline
 real-sim & 72, 309 & 20, 958 \\
  news20 & 19,996 & 1,355,191\\
  \hline
 \end{tabular}
  \caption{\label{tab:log_data}Two  datasets for  sparse logistic regression.}
\end{table}

We let each coordinate hold roughly 50 features. {Since the total number of features is not divisible by 50, some coordinates have 51 features.} We let each thread draw a coordinate uniformly at random at each iteration. We stop all the tests after 10 epochs since they have nearly identical progress per epoch. 
The step size is set to $\eta_k=0.9,\,\forall k$. Let $A = [a_1, \ldots, a_N]^{\top}$ and $b = [b_1, ..., b_N]^{\top}$. In global memory, we store $A, b$ and $x$. We also store the product $Ax$ in global memory so that the forward step can be efficiently computed. Whenever a coordinate of $x$ gets updated, $Ax$ is immediately updated at a low cost. Note that if $Ax$ is \emph{not} stored in global memory, every coordinate update will have to compute $Ax$ from scratch, which involves the entire $x$ and will be very expensive.  

Table \ref{tab:log_time} gives the running times of  the sync-parallel and async-parallel implementations on the two datasets. We can observe that async-parallel achieves almost-linear speedup, but sync-parallel scales very poorly as we explain below.

In the sync-parallel implementation,  all the running threads have to wait for the last thread to finish an iteration, and therefore if a thread has a large load, it slows down the iteration. Although every thread is (randomly) assigned to roughly the same number of features (either 50 or 51 components of $x$) at each iteration, their  $a_i$'s have very different numbers of nonzeros, and the thread with the largest number of nonzeros is the slowest. (Sparse matrix computation is used for both datasets, which are very large.) As more threads are used,  despite that they altogether do more work at each iteration,  the per-iteration time may increase as the slowest thread tends to be slower. On the other hand, async-parallel coordinate update does not suffer from the  load imbalance. Its performance grows nearly linear with the number of threads.

Finally, we have observed that the progress toward solving \eqref{eqn:log} is mainly a function of the number of epochs and does not change appreciably  when the number of threads increases or between sync-parallel and async-parallel. Therefore, we always stop at 10 epochs.

\begin{table}[htbp]
\centering
 \begin{tabular}{|c|r|r|r|r|r|r|r|r|}
  \hline
  \multirow{3}{*}{\# threads} & \multicolumn{4}{|c|}{real-sim} & \multicolumn{4}{c|}{news20} \\
  \cline{2-9}
  & \multicolumn{2}{|c|}{time (s)} &  \multicolumn{2}{c|}{speedup} &  \multicolumn{2}{c|}{time (s)} & \multicolumn{2}{c|}{speedup}\\
  \cline{2-9}
  & async & sync &  async & sync &  async & sync &  async & sync \\
  \hline
   1 &   81.6 &  82.1 & 1.0   & 1.0 & 591.1   & 591.3 & 1.0   & 1.0\\
   2 &   45.9   &  80.6 & 1.8   & 1.0 & 304.2   & 590.1 & 1.9   & 1.0\\
   4 &   21.6   &  63.0   & 3.8   & 1.3 & 150.4   & 557.0 & 3.9   & 1.1\\
   8 &   16.1   &  61.4   & 5.1   & 1.3 & 78.3     & 525.1 & 7.5   & 1.1\\
   16 & 7.1     &  46.4   & 11.5 & 1.8 & 41.6     & 493.2 & 14.2 & 1.2\\
  \hline
 \end{tabular}
 \caption{\label{tab:log_time}Running times of async-parallel and sync-parallel FBS implementations for $\ell_1$ regularized logistic regression on two datasets. Sync-parallel has very poor speedup  due to the large distribution of coordinate sparsity and thus the large load imbalance across threads.}
\end{table}

\section{Conclusions}
We have presented a coordinate update method for fixed-point iterations, which updates one coordinate (or a few variables) at every iteration and can be applied to solve linear systems,\cut{ partial differential equations,} optimization problems, saddle point problems, variational inequalities, and so on. We proposed a new concept called CF operator. When an operator is CF, its coordinate update is computationally worthy and often preferable over the full update method, in particular in a parallel computing setting. We gave examples of CF operators and also discussed how the properties can be preserved by composing two or more such operators such as in operator splitting and primal-dual splitting schemes. In addition, we have developed CF algorithms for problems arising in several different areas including machine learning, imaging, finance, and distributed computing. Numerical experiments on portfolio optimization, CT imaging, and logistic regression have been provided to demonstrate the superiority of CF methods over their counterparts that update all coordinates at every iteration.


\bibliographystyle{spmpsci}
\bibliography{asyn}

\begin{thebibliography}{10}
\providecommand{\url}[1]{{#1}}
\providecommand{\urlprefix}{URL }
\expandafter\ifx\csname urlstyle\endcsname\relax
  \providecommand{\doi}[1]{DOI~\discretionary{}{}{}#1}\else
  \providecommand{\doi}{DOI~\discretionary{}{}{}\begingroup
  \urlstyle{rm}\Url}\fi

\bibitem{attouch2010proximal}
Attouch, H., Bolte, J., Redont, P., Soubeyran, A.: Proximal alternating
  minimization and projection methods for nonconvex problems: An approach based
  on the {K}urdyka-{L}ojasiewicz inequality.
\newblock Mathematics of Operations Research \textbf{35}(2), 438--457 (2010)

\bibitem{BMR1997asyn-multisplit}
Bahi, J., Miellou, J.C., Rhofir, K.: Asynchronous multisplitting methods for
  nonlinear fixed point problems.
\newblock Numerical Algorithms \textbf{15}(3-4), 315--345 (1997)

\bibitem{Baudet_1978_asynchronous}
Baudet, G.M.: Asynchronous iterative methods for multiprocessors.
\newblock J. ACM \textbf{25}(2), 226--244 (1978).
\newblock \doi{10.1145/322063.322067}

\bibitem{bauschke1993convergence}
Bauschke, H.H., Borwein, J.M.: On the convergence of von {N}eumann's
  alternating projection algorithm for two sets.
\newblock Set-Valued Analysis \textbf{1}(2), 185--212 (1993)

\bibitem{bauschke2011convex}
Bauschke, H.H., Combettes, P.L.: Convex analysis and monotone operator theory
  in Hilbert spaces.
\newblock Springer Science \& Business Media (2011)

\bibitem{Baz200591}
Baz, D.E., Frommer, A., Spiteri, P.: Asynchronous iterations with flexible
  communication: contracting operators.
\newblock Journal of Computational and Applied Mathematics \textbf{176}(1), 91
  -- 103 (2005)

\bibitem{Baz1998429}
Baz, D.E., Gazen, D., Jarraya, M., Spiteri, P., Miellou, J.: Flexible
  communication for parallel asynchronous methods with application to a
  nonlinear optimization problem.
\newblock In: E.~D'Hollander, F.~Peters, G.~Joubert, U.~Trottenberg, R.~Volpel
  (eds.) Parallel Computing Fundamentals, Applications and New Directions,
  \emph{Advances in Parallel Computing}, vol.~12, pp. 429 -- 436. North-Holland
  (1998)

\bibitem{beck2013convergence}
Beck, A., Tetruashvili, L.: On the convergence of block coordinate descent type
  methods.
\newblock SIAM Journal on Optimization \textbf{23}(4), 2037--2060 (2013)

\bibitem{bengio2006label}
Bengio, Y., Delalleau, O., Le~Roux, N.: Label propagation and quadratic
  criterion.
\newblock In: Semi-Supervised Learning, pp. 193--216. MIT Press (2006)

\bibitem{bertsekas1983distributed}
Bertsekas, D.P.: Distributed asynchronous computation of fixed points.
\newblock Mathematical Programming \textbf{27}(1), 107--120 (1983)

\bibitem{bertsekas1999nonlinear}
Bertsekas, D.P.: Nonlinear programming.
\newblock Athena Scientific (1999)

\bibitem{bertsekas1989parallel}
Bertsekas, D.P., Tsitsiklis, J.N.: Parallel and distributed computation:
  numerical methods.
\newblock Prentice hall Englewood Cliffs, NJ (1989)

\bibitem{bolte2014proximal}
Bolte, J., Sabach, S., Teboulle, M.: Proximal alternating linearized
  minimization for nonconvex and nonsmooth problems.
\newblock Mathematical Programming \textbf{146}(1-2), 459--494 (2014)

\bibitem{bradley2011parallel}
Bradley, J.K., Kyrola, A., Bickson, D., Guestrin, C.: Parallel coordinate
  descent for l1-regularized loss minimization.
\newblock In: Proceedings of the 28th International Conference on Machine
  Learning (ICML-11), pp. 321--328 (2011)

\bibitem{briceno2015FDRS}
Brice{\~n}o-Arias, L.M.: Forward-{D}ouglas--{R}achford splitting and
  forward-partial inverse method for solving monotone inclusions.
\newblock Optimization \textbf{64}(5), 1239--1261 (2015)

\bibitem{briceno2013monotone}
Briceno-Arias, L.M., Combettes, P.L.: Monotone operator methods for {N}ash
  equilibria in non-potential games.
\newblock In: Computational and Analytical Mathematics, pp. 143--159. Springer
  (2013)

\bibitem{chazan1969chaotic}
Chazan, D., Miranker, W.: Chaotic relaxation.
\newblock Linear Algebra and its Applications \textbf{2}(2), 199--222 (1969)

\bibitem{combettes2014forward}
Combettes, P.L., Condat, L., Pesquet, J.C., Vu, B.C.: A forward-backward view
  of some primal-dual optimization methods in image recovery.
\newblock In: Proceedings of the 2014 IEEE International Conference on Image
  Processing (ICIP), pp. 4141--4145 (2014)

\bibitem{Patrick_2015}
Combettes, P.L., Pesquet, J.C.: Stochastic quasi-{F}ej{\'e}r block-coordinate
  fixed point iterations with random sweeping.
\newblock SIAM Journal on Optimization \textbf{25}(2), 1221--1248 (2015).
\newblock \doi{10.1137/140971233}

\bibitem{condat2013primal}
Condat, L.: A primal--dual splitting method for convex optimization involving
  {L}ipschitzian, proximable and linear composite terms.
\newblock Journal of Optimization Theory and Applications \textbf{158}(2),
  460--479 (2013)

\bibitem{DangLan-SBMD}
Dang, C.D., Lan, G.: Stochastic block mirror descent methods for nonsmooth and
  stochastic optimization.
\newblock SIAM Journal on Optimization \textbf{25}(2), 856--881 (2015).
\newblock \doi{10.1137/130936361}

\bibitem{davis2015n}
Davis, D.: An $ o (n \log (n)) $ algorithm for projecting onto the ordered
  weighted $\ell_1$ norm ball.
\newblock arXiv preprint arXiv:1505.00870  (2015)

\bibitem{davis2014convergence}
Davis, D.: Convergence rate analysis of primal-dual splitting schemes.
\newblock SIAM Journal on Optimization \textbf{25}(3), 1912--1943 (2015)

\bibitem{davis2015three}
Davis, D., Yin, W.: A three-operator splitting scheme and its optimization
  applications.
\newblock arXiv preprint arXiv:1504.01032  (2015)

\bibitem{dhillon2011nearest}
Dhillon, I.S., Ravikumar, P.K., Tewari, A.: Nearest neighbor based greedy
  coordinate descent.
\newblock In: Advances in Neural Information Processing Systems, pp. 2160--2168
  (2011)

\bibitem{douglas1956DRS}
Douglas, J., Rachford, H.H.: On the numerical solution of heat conduction
  problems in two and three space variables.
\newblock Transactions of the American Mathematical Society \textbf{82}(2),
  421--439 (1956)

\bibitem{el1998flexible}
El~Baz, D., Gazen, D., Jarraya, M., Spiteri, P., Miellou, J.C.: Flexible
  communication for parallel asynchronous methods with application to a
  nonlinear optimization problem.
\newblock Advances in Parallel Computing \textbf{12}, 429--436 (1998)

\bibitem{fercoq2015coordinate}
Fercoq, O., Bianchi, P.: A coordinate descent primal-dual algorithm with large
  step size and possibly non separable functions.
\newblock arXiv preprint arXiv:1508.04625  (2015)

\bibitem{Frommer2000201}
Frommer, A., Szyld, D.B.: On asynchronous iterations.
\newblock Journal of Computational and Applied Mathematics \textbf{123}(1-2),
  201--216 (2000)

\bibitem{gabay1976ADMM}
Gabay, D., Mercier, B.: A dual algorithm for the solution of nonlinear
  variational problems via finite element approximation.
\newblock Computers \& Mathematics with Applications \textbf{2}(1), 17--40
  (1976)

\bibitem{glowinski1975ADMM}
Glowinski, R., Marroco, A.: Sur l'approximation, par {\'e}l{\'e}ments finis
  d'ordre un, et la r{\'e}solution, par p{\'e}nalisation-dualit{\'e} d'une
  classe de probl{\`e}mes de dirichlet non lin{\'e}aires.
\newblock Revue fran{\c{c}}aise d'automatique, informatique, recherche
  op{\'e}rationnelle. Analyse num{\'e}rique \textbf{9}(2), 41--76 (1975)

\bibitem{Grippo-Sciandrone-00}
Grippo, L., Sciandrone, M.: On the convergence of the block nonlinear
  {G}auss-{S}eidel method under convex constraints.
\newblock Operations Research Letters \textbf{26}(3), 127--136 (2000)

\bibitem{Han-88}
Han, S.: A successive projection method.
\newblock Mathematical Programming \textbf{40}(1), 1--14 (1988)

\bibitem{hastie2005elements}
Hastie, T., Tibshirani, R., Friedman, J., Franklin, J.: The elements of
  statistical learning: data mining, inference and prediction.
\newblock The Mathematical Intelligencer \textbf{27}(2), 83--85 (2005)

\bibitem{hildreth1957quadprog}
Hildreth, C.: A quadratic programming procedure.
\newblock Naval Research Logistics Quarterly \textbf{4}(1), 79--85 (1957)

\bibitem{hong2015iteration}
Hong, M., Wang, X., Razaviyayn, M., Luo, Z.Q.: Iteration complexity analysis of
  block coordinate descent methods.
\newblock arXiv preprint arXiv:1310.6957v2  (2015)

\bibitem{hsieh2015passcode}
Hsieh, C.j., Yu, H.f., Dhillon, I.: {PASSCoDe}: Parallel asynchronous
  stochastic dual co-ordinate descent.
\newblock In: Proceedings of the 32nd International Conference on Machine
  Learning (ICML-15), pp. 2370--2379 (2015)

\bibitem{jacob2009group}
Jacob, L., Obozinski, G., Vert, J.P.: Group lasso with overlap and graph lasso.
\newblock In: Proceedings of the 26th International Conference on Machine
  Learning (ICML-09), pp. 433--440. ACM (2009)

\bibitem{krasnosel1955two}
Krasnosel'skii, M.A.: Two remarks on the method of successive approximations.
\newblock Uspekhi Matematicheskikh Nauk \textbf{10}(1), 123--127 (1955)

\bibitem{lebedev1967duality}
Lebedev, V., Tynjanski{\i}, N.: Duality theory of concave-convex games.
\newblock In: Soviet Math. Dokl, vol.~8, pp. 752--756 (1967)

\bibitem{li2009gcoord}
Li, Y., Osher, S.: Coordinate descent optimization for $\ell_1$ minimization
  with application to compressed sensing; a greedy algorithm.
\newblock Inverse Problems and Imaging \textbf{3}(3), 487--503 (2009)

\bibitem{liu2014asynchronous}
Liu, J., Wright, S.J.: Asynchronous stochastic coordinate descent: Parallelism
  and convergence properties.
\newblock SIAM Journal on Optimization \textbf{25}(1), 351--376 (2015)

\bibitem{liu2013asynchronous}
Liu, J., Wright, S.J., R{\'e}, C., Bittorf, V., Sridhar, S.: An asynchronous
  parallel stochastic coordinate descent algorithm.
\newblock Journal of Machine Learning Research \textbf{16}, 285--322 (2015)

\bibitem{Lu_Xiao_rbcd_2015}
Lu, Z., Xiao, L.: {On the complexity analysis of randomized block-coordinate
  descent methods}.
\newblock Mathematical Programming \textbf{152}(1-2), 615--642 (2015).
\newblock \doi{10.1007/s10107-014-0800-2}

\bibitem{luo1992convergence}
Luo, Z.Q., Tseng, P.: On the convergence of the coordinate descent method for
  convex differentiable minimization.
\newblock Journal of Optimization Theory and Applications \textbf{72}(1), 7--35
  (1992)

\bibitem{mclinden1974extension}
McLinden, L.: An extension of {F}enchel's duality theorem to saddle functions
  and dual minimax problems.
\newblock Pacific Journal of Mathematics \textbf{50}(1), 135--158 (1974)

\bibitem{nedic2001distributed}
Nedi{\'c}, A., Bertsekas, D.P., Borkar, V.S.: Distributed asynchronous
  incremental subgradient methods.
\newblock Studies in Computational Mathematics \textbf{8}, 381--407 (2001)

\bibitem{nesterov2012cd}
Nesterov, Y.: Efficiency of coordinate descent methods on huge-scale
  optimization problems.
\newblock SIAM Journal on Optimization \textbf{22}(2), 341--362 (2012)

\bibitem{schmidt2014coordinate}
Nutini, J., Schmidt, M., Laradji, I., Friedlander, M., Koepke, H.: Coordinate
  descent converges faster with the {G}auss-{S}outhwell rule than random
  selection.
\newblock In: Proceedings of the 32nd International Conference on Machine
  Learning (ICML-15), pp. 1632--1641 (2015)

\bibitem{o2014primal}
O'Connor, D., Vandenberghe, L.: Primal-dual decomposition by operator splitting
  and applications to image deblurring.
\newblock SIAM Journal on Imaging Sciences \textbf{7}(3), 1724--1754 (2014)

\bibitem{paatero1994NMF}
Paatero, P., Tapper, U.: Positive matrix factorization: A non-negative factor
  model with optimal utilization of error estimates of data values.
\newblock Environmetrics \textbf{5}(2), 111--126 (1994)

\bibitem{passty1979FBS}
Passty, G.B.: Ergodic convergence to a zero of the sum of monotone operators in
  {H}ilbert space.
\newblock Journal of Mathematical Analysis and Applications \textbf{72}(2),
  383--390 (1979)

\bibitem{peaceman1955PRS}
Peaceman, D.W., Rachford Jr, H.H.: The numerical solution of parabolic and
  elliptic differential equations.
\newblock Journal of the Society for Industrial and Applied Mathematics
  \textbf{3}(1), 28--41 (1955)

\bibitem{Peng_2015_AROCK}
{Peng}, Z., {Xu}, Y., {Yan}, M., {Yin}, W.: {ARock: an algorithmic framework
  for asynchronous parallel coordinate updates}.
\newblock ArXiv e-prints arXiv:1506.02396  (2015)

\bibitem{PYY_2013_GRock}
Peng, Z., Yan, M., Yin, W.: Parallel and distributed sparse optimization.
\newblock In: Proceedings of the 2013 Asilomar Conference on Signals, Systems
  and Computers, pp. 659--646 (2013)

\bibitem{pesquet2014class}
Pesquet, J.C., Repetti, A.: A class of randomized primal-dual algorithms for
  distributed optimization.
\newblock Journal of Nonlinear and Convex Analysis \textbf{16}(12), 2453--2490
  (2015)

\bibitem{Sargent-Sebastian-73}
Polak, E., Sargent, R., Sebastian, D.: On the convergence of sequential
  minimization algorithms.
\newblock Journal of Optimization Theory and Applications \textbf{12}(6),
  567--575 (1973)

\bibitem{razaviyayn2013unified}
Razaviyayn, M., Hong, M., Luo, Z.Q.: A unified convergence analysis of block
  successive minimization methods for nonsmooth optimization.
\newblock SIAM Journal on Optimization \textbf{23}(2), 1126--1153 (2013)

\bibitem{recht2011hogwild}
Recht, B., Re, C., Wright, S., Niu, F.: Hogwild: A lock-free approach to
  parallelizing stochastic gradient descent.
\newblock In: Advances in Neural Information Processing Systems, pp. 693--701
  (2011)

\bibitem{repetti2015random}
Repetti, A., Chouzenoux, E., Pesquet, J.C.: A random block-coordinate
  primal-dual proximal algorithm with application to 3d mesh denoising.
\newblock In: Proceedings of the 2015 IEEE International Conference on
  Acoustics, Speech and Signal Processing (ICASSP), pp. 3561--3565 (2015)

\bibitem{richtarik2014iteration}
Richt{\'a}rik, P., Tak{\'a}{\v{c}}, M.: Iteration complexity of randomized
  block-coordinate descent methods for minimizing a composite function.
\newblock Mathematical Programming \textbf{144}(1-2), 1--38 (2014)

\bibitem{richtarik2012parallel}
Richt{\'a}rik, P., Tak{\'a}{\v{c}}, M.: Parallel coordinate descent methods for
  big data optimization.
\newblock Mathematical Programming \textbf{156}(1), 433--484 (2016).
\newblock \doi{10.1007/s10107-015-0901-6}

\bibitem{rockafellar1997convex}
Rockafellar, R.T.: Convex analysis.
\newblock Princeton University Press (1997)

\bibitem{rue2005gaussian}
Rue, H., Held, L.: Gaussian Markov random fields: theory and applications.
\newblock CRC Press (2005)

\bibitem{scholkopf2001learning}
Scholkopf, B., Smola, A.J.: Learning with kernels: support vector machines,
  regularization, optimization, and beyond.
\newblock MIT press (2001)

\bibitem{Siddon}
Siddon, R.L.: Fast calculation of the exact radiological path for a
  three-dimensional {CT} array.
\newblock Medical Physics \textbf{12}(2), 252--255 (1985)

\bibitem{Strikwerda2002125}
Strikwerda, J.C.: A probabilistic analysis of asynchronous iteration.
\newblock Linear Algebra and its Applications \textbf{349}(1), 125--154 (2002)

\bibitem{tseng1991applications}
Tseng, P.: Applications of a splitting algorithm to decomposition in convex
  programming and variational inequalities.
\newblock SIAM Journal on Control and Optimization \textbf{29}(1), 119--138
  (1991)

\bibitem{tseng1991rate-asyn}
Tseng, P.: On the rate of convergence of a partially asynchronous gradient
  projection algorithm.
\newblock SIAM Journal on Optimization \textbf{1}(4), 603--619 (1991)

\bibitem{Tseng-93}
Tseng, P.: Dual coordinate ascent methods for non-strictly convex minimization.
\newblock Mathematical Programming \textbf{59}(1), 231--247 (1993)

\bibitem{FBF_Tseng}
Tseng, P.: A modified forward-backward splitting method for maximal monotone
  mappings.
\newblock {SIAM} J. Control and Optimization \textbf{38}(2), 431--446 (2000).
\newblock \doi{10.1137/S0363012998338806}

\bibitem{Tseng-01}
Tseng, P.: Convergence of a block coordinate descent method for
  nondifferentiable minimization.
\newblock Journal of Optimization Theory and Applications \textbf{109}(3),
  475--494 (2001)

\bibitem{tseng2009block-linear}
Tseng, P., Yun, S.: Block-coordinate gradient descent method for linearly
  constrained nonsmooth separable optimization.
\newblock Journal of Optimization Theory and Applications \textbf{140}(3),
  513--535 (2009)

\bibitem{tseng2009_CGD}
Tseng, P., Yun, S.: A coordinate gradient descent method for nonsmooth
  separable minimization.
\newblock Mathematical Programming \textbf{117}(1-2), 387--423 (2009)

\bibitem{von1949rings}
Von~Neumann, J.: On rings of operators. reduction theory.
\newblock Annals of Mathematics \textbf{50}(2), 401--485 (1949)

\bibitem{vu2013splitting}
V{\~u}, B.C.: A splitting algorithm for dual monotone inclusions involving
  cocoercive operators.
\newblock Advances in Computational Mathematics \textbf{38}(3), 667--681 (2013)

\bibitem{Warga-63}
Warga, J.: Minimizing certain convex functions.
\newblock Journal of the Society for Industrial and Applied Mathematics
  \textbf{11}(3), 588--593 (1963)

\bibitem{wright2015coordinate}
Wright, S.J.: Coordinate descent algorithms.
\newblock Mathematical Programming \textbf{151}(1), 3--34 (2015)

\bibitem{wu2008coordinate}
Wu, T.T., Lange, K.: Coordinate descent algorithms for lasso penalized
  regression.
\newblock The Annals of Applied Statistics \textbf{2}(1), 224--244 (2008)

\bibitem{Xu2015_APG_NTD}
Xu, Y.: {Alternating proximal gradient method for sparse nonnegative Tucker
  decomposition}.
\newblock Mathematical Programming Computation \textbf{7}(1), 39--70 (2015).
\newblock \doi{10.1007/s12532-014-0074-y}

\bibitem{XY_2013_multiblock}
Xu, Y., Yin, W.: A block coordinate descent method for regularized multiconvex
  optimization with applications to nonnegative tensor factorization and
  completion.
\newblock SIAM Journal on Imaging Sciences \textbf{6}(3), 1758--1789 (2013)

\bibitem{XY_2014_ecd}
Xu, Y., Yin, W.: A globally convergent algorithm for nonconvex optimization
  based on block coordinate update.
\newblock arXiv preprint arXiv:1410.1386  (2014)

\bibitem{XY_2015_bsg}
Xu, Y., Yin, W.: Block stochastic gradient iteration for convex and nonconvex
  optimization.
\newblock SIAM Journal on Optimization \textbf{25}(3), 1686--1716 (2015).
\newblock \doi{10.1137/140983938}

\bibitem{YL2006GrpLasso}
Yuan, M., Lin, Y.: Model selection and estimation in regression with grouped
  variables.
\newblock Journal of the Royal Statistical Society: Series B (Statistical
  Methodology) \textbf{68}(1), 49--67 (2006)

\bibitem{zadeh1970note}
Zadeh, N.: A note on the cyclic coordinate ascent method.
\newblock Management Science \textbf{16}(9), 642--644 (1970)

\end{thebibliography}

\appendix
\section{Some Key Concepts of Operators}\label{sec:op-concept}
In this section, we go over a few key concepts in monotone operator theory and operator splitting theory.
\cut{some of them and discuss when they are CF. First, 
}

\begin{definition}[monotone operator]\label{def:max-mon-op}
A \emph{set-valued} operator $\cT:\HH\rightrightarrows\HH$ is \emph{monotone} if
$\langle x-y, u-v\rangle\ge 0,\ \forall x,y\in\HH,\, u\in \cT x,\, v\in\cT y.$
Furthermore, $\cT$ is \emph{maximally monotone} if its graph $\Grph(T)=\{(x,u)\in \HH\times\HH: u\in \cT x\}$ is not strictly contained in the graph of any other monotone operator. 
\end{definition}

\begin{example}\label{exmp:mon-op}
An important maximally monotone operator is the subdifferential $\partial f$ of a closed proper convex function $f$.
\end{example}

\begin{definition}[nonexpansive operator]
An operator $\cT:\HH\to\HH$ is \emph{nonexpansive} if
$\|\cT x-\cT y\|\le\|x-y\|,\ \forall x, y\in\HH.$ We say $\cT$ is averaged, or $\alpha$-averaged, if there is one nonexpansive operator $\cR$ such that $\cT=(1-\alpha)\cI+\alpha\cR$ for some $0<\alpha<1$.  A $\frac{1}{2}$-averaged operator $\cT$ is also called \emph{firmly-nonexpansive}.
\end{definition}
By definition, a nonexpansive operator is single-valued. Let $\cT$ be averaged. If $\cT$ has a fixed point, the iteration \eqref{fpi} converges to a fixed point; otherwise, the iteration diverges unboundedly. Now let  $\cT$ be nonexpansive. The damped update of $\cT$: $x^{k+1} = x^k-\eta(x^k- \cT x^k)$, is equivalent to applying the averaged operator $(1-\eta)\cI+\eta\cT$. 

\begin{example}
A common firmly-nonexpansive operator is the resolvent of a maximally monotone map $\cT$, written as 
\begin{equation}\label{def-resolvent}\cJ_\cA := (\cI+\cA)^{-1}.
\end{equation} Given $x\in\HH$, $\cJ_\cA(x) =  \{y:x\in y+\cA y\}$. (By monotonicity of $\cA$, $\cJ_\cA$ is a singleton, and by maximality of $\cA$, $\cJ_\cA(x)$ is well defined for all $x\in\HH$. ) A reflective resolvent is \begin{equation}\label{def-ref}\cR_{\cA}:= 2\cJ_\cA-\cI.
\end{equation}
\end{example}

\begin{definition}[proximal map]\label{def-prox-map}
The proximal map for a function $f$ is a special resolvent defined as:
\begin{equation}\label{def-prox}\prox_{\gamma f}(y) = \argmin_x \big\{f(x)+\frac{1}{2\gamma}\|x-y\|^2 \big\},
\end{equation}
where $\gamma > 0$. The first-order variational condition of the minimization yields $\prox_{\gamma f}(y)=(\cI+\gamma\partial f)^{-1}$; hence, $\prox_{\gamma f}$ is firmly-nonexpansive. When $x\in\RR^m$ and $\prox_{\gamma f}$ can be computed in $O(m)$ or $O(m\log m)$ operations, we call $f$ \emph{proximable}.

Examples of proximable functions include $\ell_1,\ell_2,\ell_\infty$-norms, several matrix norms, the owl-norm \cite{davis2015n}, (piece-wise) linear functions, certain quadratic functions, and many more.
\end{definition}

\begin{example}
 A special proximal map is the projection map. Let $X$ be a nonempty closed convex set, and $\iota_S$ be its indicator function. Minimizing $\iota_S(x)$  enforces $x\in S$,  so $\prox_{\gamma \iota_S}$ reduces to the projection map $\prj_{S}$ for any $\gamma>0$. Therefore, $\prj_{S}$ is also firmly nonexpansive.
 \end{example}

\cut{A firmly nonexpansive operator is always nonexpansive and maximally monotone \cite[Example 20.27]{bauschke2011convex}. }

\begin{definition}[$\beta$-cocoercive operator]
An operator $\cT:\HH\to\HH$ is \emph{$\beta$-cocoercive} if 
$\langle x-y, \cT x-\cT y\rangle \ge \beta \|\cT x-\cT y\|^2,\ \forall x,y\in\HH.$ 
\end{definition}

\begin{example}
A special example of  cocoercive operator is the gradient of a smooth function. Let $f$ be a differentiable function. Then $\nabla f$ is $\beta$-Lipschitz continuous \emph{if and only if} $\nabla f$ is $\frac{1}{\beta}$-cocoercive \cite[Corollary 18.16]{bauschke2011convex}. \cut{If $\cT$ is $\beta$-cocoercive, $\beta \cT$ must be firmly nonexpansive \cite[Remark 4.24]{bauschke2011convex}, and $\cT$ is maximally monotone.}
\end{example}
\cut{
At last we give a counterexample to show naively extending existing algorithms to coordinate update schemes may result in divergence or wrong solutions.
\begin{example}
Consider the problem:
\begin{equation}
\begin{array}{l}
\underset{x_1,x_2\in\RR^m}{\textnormal{minimize  }} ~f(x_1)+g(x_2)\\
\textnormal{subject to } x_1-x_2=0,
\end{array}\label{c-e}
\end{equation}
We define $x=\begin{bmatrix}
x_1\\
x_2
\end{bmatrix},A=[I_m ~-I_m],F(x)=f(x_1)+g(x_2)$. $\eqref{c-e}$ can be solved by using $\eqref{pdemp}$:
\begin{equation}
\left\{
\begin{array}{l}
s^{k+1}=s^k+\gamma Ax^k\\
x^{k+1}=\prox_{\eta F}(x^k-\eta(A^\top s^k+2\gamma A^\top Ax^k))
\end{array}\label{c-efull}
\right.
\end{equation}
Where $s^k\in \RR^m,k=1,2,\ldots$ is a dual variable.
Defining $t^k=A^\top s^k$, $\eqref{c-efull}$ can be written as
\begin{equation}
\left\{
\begin{array}{l}
t^{k+1}=t^k+\gamma A^\top Ax^k\\
x^{k+1}=\prox_{\eta F}(x^k-\eta(t^k+2\gamma A^\top Ax^k))
\end{array}\label{c-efull2}
\right.
\end{equation}
Both $\eqref{c-efull}$ and $\eqref{c-efull2}$ (with $t^0=A^\top s^0$) converge in practice and so does coordinate update versions of $\eqref{c-efull}$. But our experiments show that coordinate update versions of $\eqref{c-efull2}$ do not converge to the solution of problem $\eqref{c-e}$. A key factor is that with full update versions of $\eqref{c-efull}$ and $\eqref{c-efull2}$, as well as coordinate update versions of $\eqref{c-efull}$, $t^k$ always stay in the image of $A^\top $. But with coordinate update versions of $\eqref{c-efull2}$, $t^k$ do not always stay in the image of $A^\top $, which causes trouble.
\end{example}
}
\section{Derivation of ADMM from the DRS Update}\label{sec:drs-admm}
We derive the ADMM update in \eqref{eq:admmx-y} from the DRS update \begin{subequations}\label{eqop}
\begin{align}
s^k=&\ \cJ_{\eta \cB}(t^k),\label{eqop-s}\\
t^{k+1}=&\ \left(\frac{1}{2}(2\cJ_{\eta\cA}-\cI)\circ(2\cJ_{\eta \cB}-\cI)+\frac{1}{2}\cI\right)(t^k)\label{eqop-t},
\end{align}
\end{subequations} 
where $\cA=-\partial f^*(-\cdot)$ and $\cB=\partial g^*$.

Note \eqref{eqop-s} is equivalent to $t^k\in s^k+\eta \partial g^*(s^k)$, i.e., there is a $y^k\in \partial g^*(s^k)$ such that $t^k = s^k+\eta y^k$, so
\begin{equation}\label{tempeq1}
t^k-\eta y^k=s^k\in\partial g(y^k).
\end{equation} 
In addition, \eqref{eqop-t} can be written as
\begin{align}
t^{k+1}=&\ \cJ_{\eta\cA}(2s^k-t^k)+t^k-s^k\cr
=&\ s^k+(\cJ_{\eta\cA}-\cI)(2s^k-t^k)\cr
=&\ s^k+(\cI-(\cI+\eta\partial f^*)^{-1})(t^k-2s^k)\cr
=&\ s^k+\eta(\eta\cI+\partial f)^{-1}(t^k-2s^k)\cr
=&\ s^k+\eta(\eta\cI+\partial f)^{-1}(\eta y^k-s^k),\label{tempeq2}
\end{align} 
where in the fourth equality, we have used the Moreau's Identity \cite{rockafellar1997convex}: $(\cI+\partial h)^{-1}+(\cI+\partial h^*)^{-1}=\cI$ for any closed convex function $h$. Let
\begin{equation}\label{up-x}
x^{k+1}=(\eta\cI+\partial f)^{-1}(\eta y^k-s^k)=(\cI+\frac{1}{\eta}\partial f)^{-1}(y^k-\frac{1}{\eta}s^k).
\end{equation}
Then \eqref{tempeq2} becomes
\begin{equation*}
t^{k+1}=s^k+\eta x^{k+1},
\end{equation*}
and 
\begin{equation}\label{up-s}
s^{k+1}\overset{\eqref{tempeq1}}=t^{k+1}-\eta y^{k+1}=s^k+\eta x^{k+1}-\eta y^{k+1},
\end{equation}
which together with $s^{k+1}\in\partial g(y^{k+1})$ gives
\begin{equation}\label{up-y}
y^{k+1}=(\eta\cI+\partial g)^{-1}(s^k+\eta x^{k+1})=(\cI+\frac{1}{\eta}\partial g)^{-1}(x^{k+1}+\frac{1}{\eta}s^k).
\end{equation}
Hence, from \eqref{up-x}, \eqref{up-s}, and \eqref{up-y}, the ADMM update in \eqref{eq:admmx-y} is equivalent to the DRS update in \eqref{eqop} with $\eta=\frac{1}{\gamma}$.
\section{Representing the Condat-V\~{u} Algorithm as a Nonexpansive Operator}\label{sec:vc-op}
We show how to derive the Condat-V\~{u} algorithm $\eqref{vucondat}$ by applying a forward-backward operator to the optimality condition $\eqref{pdkkt}$:

\begin{equation}
0\in\bigg(\underbrace{\begin{bmatrix}
\nabla f & 0\\
0 & 0
\end{bmatrix}}_{\mbox{operator}~\cA}+\underbrace{
\begin{bmatrix}
\partial g & 0 \\
0 & \partial h^*
\end{bmatrix}+\begin{bmatrix}
0&A^\top\\
-A&0
\end{bmatrix}}_{\mbox{operator}~\cB}\bigg) \underbrace{\begin{bmatrix}
x\\
s
\end{bmatrix}}_{z},
\end{equation}

It can be written as $0\in\cA z+\cB z$ after we define $z=\begin{bmatrix}x\\ s\end{bmatrix}$. Let $M$ be a symmetric positive definite matrix, we have
\begin{align*}
&0\in\cA z+\cB z\\
\Leftrightarrow& Mz-\cA z \in Mz+\cB z\\
\Leftrightarrow& z-M^{-1}\cA z \in z+M^{-1}\cB z\\
\Leftrightarrow& z=(\cI+M^{-1}\cB)^{-1}\circ(\cI-M^{-1}\cA)z.
\end{align*}
Convergence and other results can be found in \cite{davis2014convergence}. The last equivalent relation is due to $M^{-1}\cB$ being a maximally monotone operator under the norm induced by $M$. We let $$M=\begin{bmatrix}
\frac{1}{\eta}I&A^\top\\
A&\frac{1}{\gamma}I
\end{bmatrix}\succ 0$$
and iterate $$z^{k+1}=\cT z^k=(\cI+M^{-1}\cB)^{-1}\circ(\cI-M^{-1}\cA)z^k.$$
We have $Mz^{k+1}+{\cB}z^{k+1}= Mz^k-\cA z^k$:
$$\left\{\begin{array}{ll}
\frac{1}{\eta}x^k  +A^\top s^{k}- \nabla f(x^k)&\in \frac{1}{\eta}x^{k+1} + A^\top s^{k+1}+A^\top s^{k+1}+\partial g (x^{k+1}),\\ 
\frac{1}{\gamma}s^k + A~ x^{k~}&\in \frac{1}{\gamma}s^{k+1}+ A~ x^{k+1~} -A ~x^{k+1} ~+\partial h^*(s^{k+1}),
\end{array}\right.$$
which is equivalent to
$$\left\{
\begin{array}{l}
s^{k+1}=\prox_{\gamma h^*} (s^k+\gamma Ax^k),\\
x^{k+1}=\prox_{\eta g}(x^k-\eta(\nabla f(x^k)+A^\top(2s^{k+1}-s^k))).
\end{array}
\right.$$
Now we derived the Condat-V\~{u} algorithm. With proper choices of $\eta$ and $\gamma $, the forward-backward operator $\cT=(\cI+M^{-1}\cB)^{-1}\circ(\cI-M^{-1}\cA)$ can be shown to be $\alpha$-averaged if we use the inner product $\langle z_1,z_2\rangle_M=z_1^\top Mz_2$ and norm $\|z\|_M=\sqrt{z^\top Mz}$ on the space of $z=\begin{bmatrix}x\\ s\end{bmatrix}$. More details can be found in~\cite{davis2014convergence}.

If we change the matrix $M$ to $\begin{bmatrix}
\frac{1}{\eta}I&-A^\top\\
-A&\frac{1}{\gamma}I
\end{bmatrix}$, the other algorithm $\eqref{vucondat2}$ can be derived similarly.

\section{Proof of Convergence for Async-parallel Primal-dual Coordinate Update Algorithms}\label{pf:pdasync}
Algorithms~\ref{alg:asyn_core} and~\ref{alg:asyn_overlap} differ from that in~\cite{Peng_2015_AROCK} in the following aspects:
\begin{enumerate}
\item the operator $\TVC$ is nonexpansive under a norm induced by a symmetric positive definite matrix $M$ (see Appendix~\ref{sec:vc-op}), instead of the standard Euclidean norm;
\item the coordinate updates are no longer orthogonal to each other under the norm induced by $M$;
\item the block coordinates may overlap each other.
\end{enumerate}
Because of these differences, we make two major modifications to the proof in~\cite[Section 3]{Peng_2015_AROCK}: (i) adjusting parameters in~\cite[Lemma 2]{Peng_2015_AROCK} and modify its proof to accommodate for the new norm; (2) modify the inner product and induced norm used in \cite[Theorem 2]{Peng_2015_AROCK} and adjust the constants in \cite[Theorems 2 and 3]{Peng_2015_AROCK}.

We assume the same inconsistent case as in~\cite{Peng_2015_AROCK}, i.e., the relationship between $\hat z^k$ and $z^k$ is
\begin{equation}
\hat{z}^k=z^k+\sum_{d\in J(k)}(z^d-z^{d+1}),
\end{equation}
where $J(k)\subseteq \{k-1,...,k-\tau\}$ and $\tau$ is the maximum number of other updates to $z$ during the computation of the update. 
Let $\cS=\cI-\TVC$. 
Then the coordinate update can be rewritten as $z^{k+1}=z^k-\frac{\eta_k}{(m+p)q_{i_k}}\cS_{i_k}\hat{z}^k$, where $\cS_{i}\hat{z}^k=(\hat z_1^k,\dots, \hat z_{i-1}^k, (\cS\hat z^k)_i, \hat z_{i+1}^k,\dots,\hat z_{m+p}^k)$  for Algorithm~\ref{alg:asyn_core}. For Algorithm~\ref{alg:asyn_overlap}, the update is  
\begin{align}\label{eqn:asyn_update_a2}
z^{k+1}=z^k-\frac{\eta_k}{mq_{i_k}}\cS_{i_k}\hat{z}^k,
\end{align} where 
$$\cS_{i}\hat{z}^k=\begin{bmatrix}
0&&&&&&&&&\\
&\ddots&&&&&&&&\\
&&0&&&&&&&\\
&&&\cI_{\HH_i}&&&&&&\\
&&&&0&&&&&\\
&&&&&\ddots&&&&\\
&&&&&&0&&&\\
&&&&&&&\rho_{i,1}\cI_{\GG_1}&&\\
&&&&&&&&\ddots&\\
&&&&&&&&&\rho_{i,p}\cI_{\GG_p}
\end{bmatrix}\cS\hat{z}^k.$$
Let $\lambda_{\max}$ and $\lambda_{\min}$ be the
maximal and minimal eigenvalues of the matrix $M$, respectively, and $\kappa=\frac{\lambda_{\max}}{\lambda_{\min}}$ be the condition number.
Then we have the following lemma.
\begin{lemma}
For both Algorithms~\ref{alg:asyn_core} and~\ref{alg:asyn_overlap},
\begin{align}
\sum_{i}\cS_i\hat{z}^k&=\cS\hat{z}^k,\\
\sum_{i}\|\cS_i\hat{z}^k\|_M^2&\leq \kappa\|\cS\hat{z}^k\|_M^2,\label{eqn:bound_lemma1}
\end{align}
where $i$ runs from $1$ to $m+p$ for Algorithm~\ref{alg:asyn_core} and $1$ to $m$ for Algorithm~\ref{alg:asyn_overlap}.
\end{lemma}
\begin{proof} The first part comes immediately from the definition of $\cS$ for both algorithms. For the second part, we have
\begin{align}
\sum_{i}\|\cS_i\hat{z}^k\|_M^2&\leq \sum_{i} \lambda_{\max}\|\cS_i\hat{z}^k\|^2 = \lambda_{\max}\|\cS\hat{z}^k\|^2\leq {\lambda_{\max}\over\lambda_{\min}}\|\cS\hat{z}^k\|^2_M,
\end{align}
for Algorithm~\ref{alg:asyn_core}. For Algorithm~\ref{alg:asyn_overlap}, the equality is replaced by ``$\leq$".
\end{proof}

At last we define
\begin{align}\label{eqn:def_bar_x}
\bar{z}^{k+1} := z^k - \eta_k \cS\hat{z}^{k},
\end{align}
$q_{\min}=\min_iq_i>0$, and $|J(k)|$ be the number of elements in $J(k)$. 
It is shown in~\cite{davis2014convergence} that with proper choices of $\eta$ and $\gamma$, $\TVC$ is nonexpansive under the norm induced by $M$. Then Lemma~\ref{lemma:a-avg} shows that $\cS$ is 1/2-cocoercive under the same norm.
\begin{lemma}\label{lemma:a-avg} An operator $\cT:\FF\to\FF$ is nonexpansive under the induced norm by $M$ if and only if $\cS = \cI - \cT$ is {${1}/{2}$-cocoercive} under the same norm, i.e.,
\begin{equation}\label{eqn:alpha_avg}
\langle z-\tilde{z},\cS z- \cS\tilde{z}\rangle_M \ge \frac{1}{2}\|\cS z- \cS\tilde{z}\|_M^2,\quad
{\forall~z,\tilde{z}\in\FF}.
\end{equation}
\end{lemma}
The proof is the same as that of \cite[Proposition 4.33]{bauschke2011convex}. 

We state the complete theorem for Algorithm~\ref{alg:asyn_overlap}. The theorem for Algorithm~\ref{alg:asyn_core} is similar (we need to change $m$ to $m+p$ when necessary).
\begin{thm}\label{thm:async-convergence2}
Let $Z^*$ be the set of optimal solutions of~\eqref{pdproblem} and $(z^k)_{k\geq0}\subset \FF$ be the sequence generated by Algorithm~\ref{alg:asyn_overlap} (with proper choices of $\eta$ and $\gamma$ such that $\TVC$ is nonexpansive under the norm induced by $M$), under the following conditions:
\begin{enumerate}[(i)]
\item $f,g,h^*$ are closed proper convex functions. In addition, $f$ is differentiable and $\nabla f$ is Lipschitz continuous with $\beta$;
\item $\eta_k
\in [\eta_{\min}, \eta_{\max}]$ for certain $0<\eta_{\max}<\frac{mq_{\min}}{2\tau
\sqrt{\kappa q_{\min}}+\kappa}$ and any $0<\eta_{\min}\leq\eta_{\max}$.
\end{enumerate}
Then $(z^k)_{k\geq 0}$ converges to a $Z^*$-valued random variable with probability 1.
\end{thm}
The proof directly follows~\cite[Section 3]{Peng_2015_AROCK}. Here we only present the key modifications. Interested readers are referred to~\cite{Peng_2015_AROCK} for the complete procedure.
 
The next lemma shows that the conditional expectation of the distance between $z^{k+1}$ 
and any $z^*\in \mathbf{Fix} \TVC=Z^*$ for given $\cZ^k=\{z^0,z^1,\cdots,z^k\}$ has an
upper bound that depends on $\cZ^k$ and $z^*$ only.
\begin{lemma}\label{lemma:fund}
Let $(z^k)_{k\geq 0}$ be the sequence generated by Algorithm
\ref{alg:asyn_overlap}.  Then for
any $z^*\in \mathbf{Fix} \TVC$, we have
\begin{align}\label{eqn:fund_inquality0}
\begin{aligned}
\mathbb{E}\big(\|z^{k+1} - z^* \|_M^2 \,\big|\, \cZ^k\big)  \leq & \|z^{k} - z^*
\|_M^2  +{\sigma\over m}\sum_{d\in J(k)}\|z^d-z^{d+1}\|_M^2\\
+& {1\over m}\left({{|J(k)|}\over \sigma}+{\kappa\over
mq_{\min}}-{1\over \eta_k}\right)\|z^k-\bar z^{k+1}\|_M^2
\end{aligned}
\end{align}
where $\mathbb{E}(\cdot\,|\,\cZ^k)$ denotes conditional  expectation  on $\cZ^k$ and $\sigma>0$ (to be optimized later).
\end{lemma}
\begin{proof}
We have
\begin{equation}\label{eqn:equality_inconsistent}
\begin{aligned}
&\mathbb{E}\left(\|z^{k+1} - z^*\|_M^2\,|\,\cZ^k\right)\\
\overset{\eqref{eqn:asyn_update_a2}}=&\mathbb{E}\left(\|z^{k}  -
\textstyle\frac{\eta_k}{mq_{i_k}}\cS_{i_k} \hat{z}^{k}- z^*
\|_M^2\,|\,\cZ^k\right)\\
=& \|z^k -
z^*\|_M^2+\mathbb{E}\left(\textstyle\frac{2\eta_k}{mq_{i_k}} \left\langle
\cS_{i_k} \hat{z}^{k}, z^* - z^k \right\rangle_M +
\textstyle\frac{\eta_k^2}{m^2q_{i_k}^2}
\|\cS_{i_k}\hat{z}^{k}\|_M^2\,\big|\,\cZ^k\right)\\
=&\|z^k -
z^*\|_M^2+\textstyle\frac{2\eta_k}{m} \sum_{i=1}^m\left\langle \cS_i
\hat{z}^{k}, z^* - z^k \right\rangle_M +
\frac{\eta_k^2}{m^2}\sum_{i=1}^m\frac{1}{q_i}\|\cS_i\hat{z}^{k}\|_M^2\\
=&\|z^k
- z^*\|_M^2+\textstyle\frac{2\eta_k}{m} \left\langle \cS \hat{z}^{k}, z^* - z^k
\right\rangle_M +\frac{\eta_k^2}{m^2} \sum_{i=1}^m\frac{1}{q_i}
\|\cS_i\hat{z}^{k}\|_M^2,
\end{aligned}
\end{equation}
where the third equality holds because the probability of choosing $i$ is $q_i$.

Note that
\begin{equation}\label{term2}
\begin{aligned}
\textstyle\sum_{i=1}^m\frac{1}{q_i} \|\cS_i\hat{z}^{k}\|_M^2&\le\frac{1}{q_{\min}}
\sum_{i=1}^m\|\cS_i\hat{z}^{k}\|_M^2\overset{\eqref{eqn:bound_lemma1}}{\leq}
\frac{\kappa}{q_{\min}} \sum_{i=1}^m\|\cS\hat{z}^{k}\|^2_M\\
&\overset{\eqref{eqn:def_bar_x}}{=}\frac{\kappa}{\eta_k^2q_{\min}}\|z^k-\bar{z}^{k+1}\|_M^2,
\end{aligned}
\end{equation}
and
\begin{equation}\label{term1}
\begin{aligned}
&\langle \cS \hat{z}^{k}, z^* - z^k \rangle_M\\
=&\textstyle\langle \cS \hat{z}^{k},z^* - \hat{z}^k + \sum_{d\in J(k)} (z^{d} -
z^{d+1})\rangle_M\cr \overset{\eqref{eqn:def_bar_x}}=&\textstyle\langle \cS
\hat{z}^{k}, z^* - \hat{z}^k\rangle_M + \frac{1}{\eta_k}\sum_{d\in J(k)}\langle
z^k-\bar{z}^{k+1}, z^{d} - z^{d+1}\rangle_M\cr \le&\textstyle \langle \cS
\hat{z}^{k}-\cS z^*, z^* - \hat{z}^k\rangle_M+\frac{1}{2\eta_k}\sum_{d\in
J(k)}\big(\frac{1}{\sigma}\|z^k-\bar{z}^{k+1}\|_M^2+ \sigma\|z^{d} -
z^{d+1}\|_M^2\big)\cr \overset{\eqref{eqn:alpha_avg}}\le&\textstyle
-\frac{1}{2}\|\cS \hat{z}^{k}\|_M^2+\frac{1}{2\eta_k}\sum_{d\in
J(k)}(\frac{1}{\sigma}\|z^k-\bar{z}^{k+1}\|_M^2+ \sigma\|z^{d} -
z^{d+1}\|_M^2)\cr
\overset{\eqref{eqn:def_bar_x}}=&\textstyle-\frac{1}{2\eta_k^2}\|z^k-\bar{z}^{k+1}\|_M^2+
\frac{|J(k)|}{2\sigma\eta_k}\|z^k-\bar{z}^{k+1}\|_M^2+\frac{\sigma}{2\eta_k}\sum_{d\in
J(k)}\|z^{d} - z^{d+1}\|_M^2,
\end{aligned}
\end{equation}
where the first inequality follows from the Young's inequality. Plugging~\eqref{term2} and~\eqref{term1} into~\eqref{eqn:equality_inconsistent} gives the desired result.\hfill\end{proof}

Let $\FF^{\tau+1}=\prod_{i=0}^{\tau}\FF$ be a product space and $\langle\cdot\, |\,\cdot \rangle$ be the induced  inner product:
$$\langle (z^0,\ldots,z^{\tau})\,|\,(\tilde{z}^0,\ldots,\tilde{z}^{\tau})\rangle=\sum_{i=0}^{\tau}\langle z^i,\tilde{z}^i\rangle_M,\quad\forall (z^0,\ldots,z^{\tau}), (\tilde{z}^0,\ldots,\tilde{z}^{\tau})\in\FF^{\tau+1}.$$ 
Define a $(\tau+1)\times(\tau+1)$ matrix $U'$ by 
\begin{equation*}
U':=\begin{bmatrix}1 & 0 & \cdots &0\\
0 & 0 &\cdots & 0\\ \vdots &\vdots & \ddots & \vdots\\
0 & 0 &\cdots & 0 \end{bmatrix}
+\sqrt{\frac{q_{\min}}{\kappa}}\begin{bmatrix} \tau & -\tau &  & \\
-\tau & 2\tau-1 & 1-\tau & \\
 & 1-\tau & 2\tau-3 & 2-\tau  & \\
 & & \ddots & \ddots & \ddots &\\
 & & & -2 & 3  & -1 \\
 & & & &-1 & 1
\end{bmatrix},
\end{equation*}
and let $U=U'\otimes \cI_\FF$. Here $\otimes$ represents the Kronecker product. For a given $(y^0,\cdots,y^\tau)\in\FF^{\tau+1}$, $(z^0,\cdots,z^\tau)=U(y^0,\cdots,y^\tau)$ is given by:
\begin{align*}
&z^0=\textstyle
y^0+\tau\sqrt{\frac{q_{\min}}{\kappa}} (y^0-y^1),\\
&z^i =
\textstyle\sqrt{\frac{q_{\min}}{\kappa}}((i-\tau-1)y^{i-1}+(2\tau-2i+1)y^i+(i-\tau)y^{i+1}),\text{ if } 1\le i\le \tau-1,\\
&z^{\tau}=\textstyle\sqrt{\frac{q_{\min}}{\kappa}} (y^{\tau}-y^{\tau-1}).
\end{align*} 
Then $U$ is a self-adjoint and positive definite linear operator since $U'$ is
symmetric and positive definite, and we define $\langle\cdot\, |\,
\cdot\rangle_U=\langle\cdot\, |\, U\cdot\rangle$ as the $U$-weighted inner
product and $\|\cdot\|_U$ the induced norm.
 
Let 
\begin{equation*}
\vz^k=(z^k,z^{k-1},\ldots,z^{k-\tau})\in \FF^{\tau+1},~k\ge 0,~\vz^* =(z^*,\ldots,z^*)\in\vZ^*\subseteq\FF^{\tau+1},
\end{equation*}
where $z^{k}=z^{0}$ for $k<0$. With
\begin{equation}\label{eqn:xi}
\textstyle \xi_k(\vz^*) := \|\vz^k-\vz^*\|_U^2=\|z^{k} - z^*\|_M^2 +
\sqrt{q_{\min}\over \kappa}\sum_{i=k-\tau}^{k-1} (i-(k-\tau)+1) \|
z^{i} - z^{i+1}\|_M^2,
\end{equation}
we have the following fundamental inequality:
\begin{thm}[fundamental inequality]\label{thm:fund_inquality}
Let $(z^k)_{k\geq 0}$ be the sequence generated by Algorithm~\ref{alg:asyn_overlap}. Then for any $\vz^*\in\vZ^*$, it holds that 
\begin{equation*}
\begin{aligned}
\mathbb{E}\left(\xi_{k+1}(\vz^*) \,\big|\, \cZ^k\right)  
\leq  \xi_k(\vz^*)  + \frac{1}{m}
\left(\frac{2\tau\sqrt{\kappa}}{m\sqrt{q_{\min}}} +
{\kappa\over mq_{\min}} - \frac{1}{ \eta_k}\right)
\|\bar{z}^{k+1} - z^k \|_M^2.
\end{aligned}
\end{equation*}
\end{thm}
\begin{proof}Let $\sigma=m\sqrt{\frac{q_{\min}}{\kappa}}$. We have 
\begin{align*}
~&\mathbb{E} (\xi_{k+1}(\vz^*) | \cZ^k)  \\
\overset{\eqref{eqn:xi}}= & \textstyle \mathbb{E} (\|z^{k+1} - z^*\|_M^2| \cZ^k)
+ \sigma\sum_{i=k+1-\tau}^{k} \frac{i-(k-\tau)}{m} \mathbb{E} (\| z^{ i} -
z^{i+1}\|_M^2 | \cZ^k) \\
  \overset{\eqref{eqn:asyn_update_a2}}= & \textstyle\mathbb{E} (\|z^{k+1} -
  z^*\|_M^2| \cZ^k) + \frac{\sigma\tau}{m}
  \mathbb{E}(\frac{\eta_k^2}{m^2q_{i_k}^2}\| S_{i_k}\hat z^k\|_M^2|\cZ^k) + 
  \sigma\sum_{i=k+1-\tau}^{k-1} \frac{i-(k-\tau)}{m} \| z^{i} - z^{i+1}\|_M^2\\
          \le & \textstyle\mathbb{E} (\|z^{k+1} - z^*\|_M^2| \cZ^k) +
          \frac{\sigma\tau\kappa}{m^3q_{\min}} \| z^{k } -
          \bar{z}^{k+1}\|_M^2 + \sigma\sum_{i=k+1-\tau}^{k-1} \frac{i-(k-\tau)}{m} \| z^{i} -
          z^{i+1}\|_M^2\\
  \overset{\eqref{eqn:fund_inquality0}}\leq &\textstyle\|z^k - z^*\|_M^2 +
  \frac{1}{m} \left({|J(k)|\over \sigma} +
  \frac{\sigma\tau\kappa}{m^2q_{\min}} +
  {\kappa\over mq_{\min}} - \frac{1}{\eta_k}\right) \|z^k -
  \bar{z}^{k+1}\|_M^2   \\
&\textstyle+\frac{\sigma}{m}\sum_{d\in J(k)}\|{z}^{d} - {z}^{d+1}\|_M^2 + 
\sigma\sum_{i=k+1-\tau}^{k-1} \frac{i-(k-\tau)}{m} \| z^{i} - z^{i+1}\|_M^2\\
 \leq &\textstyle\|z^k - z^*\|_M^2 + \frac{1}{m} \left({\tau\over \sigma} +
 \frac{\sigma\tau\kappa}{m^2q_{\min}} +
 {\kappa\over mq_{\min}} - \frac{1}{ \eta_k}\right) \|z^k -
 \bar{z}^{k+1}\|_M^2   \\
&\textstyle+\frac{\sigma}{m}\sum_{i=k-\tau}^{k - 1}\|{z}^{i} - {z}^{i+1}\|_M^2 + 
\sigma\sum_{i=k+1-\tau}^{k-1} \frac{i-(k-\tau)}{m} \| z^{i} - z^{i+1}\|_M^2\\
 \overset{\eqref{eqn:xi}}= &\textstyle\xi_k(\vx^*)  + \frac{1}{m}
 \left(\frac{2\tau\sqrt{\kappa}}{m\sqrt{q_{\min}}} +
 {\kappa\over mq_{\min}} - \frac{1}{ \eta_k}\right) \|z^k -
 \bar{z}^{k+1}\|_M^2.
\end{align*}
The first inequality follows from the computation of the conditional expectation on $\cZ^k$
and~\eqref{term2}, the third inequality holds because
$J(k)\subset\{k-1,k-2,\cdots,k-\tau\}$, and the last equality uses
$\sigma=m\sqrt{\frac{q_{\min}}{\kappa}}$, which minimizes ${\tau\over
\sigma} + \frac{\sigma\tau\kappa}{m^2q_{\min}}$ over
$\sigma>0$.
Hence, the desired inequality holds.
\hfill\end{proof}

\end{document}